\documentclass[a4paper,10pt,bibliography=totoc,listof=totoc,headinclude=true,cleardoublepage=empty,oneside]{scrartcl}

\usepackage[cm]{fullpage}
\usepackage[english,ngerman]{babel}
\usepackage[utf8]{inputenc}
\usepackage{fullpage}
\usepackage{ifthen}
\usepackage{color}
\usepackage{amsmath,amsthm,amssymb,amsfonts,bbm}
\usepackage{mathtools}
\usepackage{graphicx, float}
\usepackage{psfrag}
\usepackage{enumitem}
\usepackage{bm}
\usepackage{mparhack}
\usepackage{wasysym}
\usepackage{pdflscape} 
\usepackage{xr}
\usepackage{enumitem}

\usepackage[unicode,colorlinks=true,pagebackref=false]{hyperref}

\KOMAoptions{footinclude=false} 
\KOMAoptions{headsepline=true}  
\KOMAoptions{DIV=12}            
\KOMAoptions{BCOR=8mm}          

\recalctypearea 
\numberwithin{equation}{section}

\graphicspath{ {fig/} }

\newcommand\smax{\hat{s}}

\newcommand\norm[1]{\left\|#1\right\|}
\newcommand\trinorm[1]{|||#1|||}
\newcommand\triscalar[2]{\langle \langle #1, #2 \rangle \rangle}

\newcommand*\curl{\mathop{}\!\mathrm{curl}}
\newcommand*\scalarcurl{\mathop{}\!\mathrm{curl}}
\renewcommand*\div{\mathop{}\!\mathrm{div}}

\newcommand*\HOne{H^1(\Omega)}

\newcommand*\HZeroOne{H_0^1(\Omega)}

\newcommand*\HDiv{\pmb{H}(\div, \Omega)}

\newcommand*\HZeroDiv{\pmb{H}_0(\div, \Omega)}

\newcommand*\HScalarCurl{\pmb{H}(\scalarcurl,\Omega)}

\newcommand*\HZeroScalarCurl{\pmb{H}_0(\scalarcurl,\Omega)}

\newcommand*\HCurl{\pmb{H}(\curl, \Omega)}

\newcommand*\HZeroCurl{\pmb{H}_0(\curl, \Omega )}
\newcommand*\MyHCurl{\pmb{Y}}

\newcommand*\Sp{S_{p_s}(\mathcal{T}_h)}
\newcommand*\SpZero{S^0_{p_s}(\mathcal{T}_h)}

\newcommand*\Nedelec{\pmb{\mathrm{N}}_{p_v}(\mathcal{T}_h)}

\newcommand*\BDM{\pmb{\mathrm{BDM}}_{p_v}(\mathcal{T}_h)}

\newcommand*\RT{\pmb{\mathrm{RT}}_{p_v-1}(\mathcal{T}_h)}

\newcommand*\RTBDM{\pmb{\mathrm{V}}_{p_v}(\mathcal{T}_h)}
\newcommand*\RTBDMZero{\pmb{\mathrm{V}}^0_{p_v}(\mathcal{T}_h)}

\newcommand*\Ih{\pmb{I}_h}
\newcommand*\IhZero{\pmb{I}_h^0}
\newcommand*\IhGamma{\pmb{I}^\Gamma_h}

\newcommand\restr[2]{{
  \left.\kern-\nulldelimiterspace 
  #1 
  \vphantom{\big|} 
  \right|_{#2} 
  }}

\def\productspacerobin{\pmb{V} \times W}

\newcommand{\eremk}{\hbox{}\hfill\rule{0.8ex}{0.8ex}}

\newtheorem{theorem}{Theorem}[section]

\newtheorem{lemma}[theorem]{Lemma}
\newtheorem{corollary}[theorem]{Corollary}
\newtheorem{proposition}[theorem]{Proposition}

\theoremstyle{definition}

\newtheorem{example}[theorem]{Example}
\newtheorem{remark}[theorem]{Remark}

\newtheorem{assumption}[theorem]{Assumption}

\begin{document}

\title{Optimal convergence rates in $L^2$ for a first order system \\ least squares finite element method - \\ Part II: inhomogeneous Robin boundary conditions}
\author{
  M. Bernkopf
  \and
  J.M. Melenk
}

\maketitle
\selectlanguage{english}
\pagenumbering{arabic}

\begin{abstract}
We consider divergence-based high order discretizations 
of an $L^2$-based first order system least squares formulation of a second order 
elliptic equation with Robin boundary conditions. For smooth geometries, 
we show optimal convergence rates in the $L^2(\Omega)$ norm for the scalar variable. Convergence rates 
for the $L^2(\Omega)$-norm error of the gradient of the scalar variable as well as vectorial variable are also derived. 
Numerical examples illustrate the analysis. 
\end{abstract}

\section{Introduction}\label{section:introduction}

Least Squares Methods and related techniques are an established tool for the numerical treatment of 
partial differential equations as witnessed by the monographs \cite{bochev-gunzburger09,jiang98}, which provide 
both mathematical analysis and examples of applications in fluid and solid mechanics. 
Least Squares Methods methods are successfully used in  computational fluid mechanics
(see, e.g., \cite{jiang98,cai-lee-wang04}), solid mechanics (see, e.g., \cite{cai-starke04,bertrand-cai-park19,fuehrer-heuer-niemi22}),  
electromagnetics (see, e.g., \cite{bramble-kolev-pasciak05a,bramble-kolev-pasciak05}), 
and eigenvalue problems (see, e.g., \cite{bramble-kolev-pasciak05,bertrand-boffi22, bertrand-boffi21}). 
Reasons for the popularity of these methods include their flexibility to deal with a variety
of equations and the ease of coupling different equations, the fact that they lead to symmetric
positive definite systems by construction, and that they naturally come with error estimators. 

For scalar second order problems, an important approach in Least Squares methodologies is to reformulate it 
as a first order system based on a scalar variable and a vectorial variable and to subsequently
minimize the residuum in the $L^2$-norm. This method, called 
First Order System Least Squares Method (FOSLS), is computationally attractive and leads 
to quasi-optimality in a residual norm, \cite{cai-lazarov-manteuffel-mccormick94,bochev-gunzburger09}. 
Obtaining optimal error estimates 
in norms other than the natural residual norm, say, $L^2$ for the scalar variable is the purpose of the present work. 
Here, optimality refers not only to the optimal achievable convergence rate under the assumption of \emph{sufficient} 
smoothness of the solution but relates to the fact that the regularity of both the scalar variable and 
vectorial variable are dictated by the regularity of the data. For example, for data $f \in L^2$, the vectorial variable 
(later denoted $\pmb{\varphi}$) is merely in $\HDiv$ so that one cannot expect for the convergence in the 
residual norm a rate. The tools to overcome this 
obstacle are duality arguments and approximation operators with suitable orthogonality properties  
as previously done in \cite{ku11} and \cite{bernkopf-melenk22}. 
Under regularity assumptions for the appropriate dual problems, optimal convergence rates 
can then be established. 

In the first part \cite{bernkopf-melenk22} of this series of papers, 
we analyzed high order finite element discretizations of a 
first order system least squares (FOSLS) formulation
of a Poisson-type second order elliptic 
problem with homogeneous boundary conditions and obtained optimal error estimates for the 
$L^2$-error of the scalar variable. 
Here, we generalize the approach of \cite{bernkopf-melenk22} to Poisson-type problems with inhomogeneous 
Robin boundary conditions. Compared to \cite{bernkopf-melenk22} and \cite{ku11} the presence of the boundary terms, which
are $L^2$-terms, requires additional duality arguments and corresponding approximation 
results for the operator $\IhGamma$ that effects the required orthogonalities. As an aside, we mention that elliptic problems with Robin boundary conditions arise in 
applications, for example, in wave propagation problems with impedance boundary conditions. The analysis 
of the Helmholtz equation with impedance conditions is, however, beyond the scope of the present work
as one leaves the realm of strong ellipticity; we refer, however, to \cite{chen-qiu17,bernkopf-melenk19} for 
analyses of FOSLS discretizations of the Helmholtz equation 
and also to \cite{monsuur2023pollutionfree, demkowicz-gopalakrishnan-muga-zitelli12,gopalakrishnan-muga-olivares14} for Discontinuous Petrov Galerkin (DPG) discretizations.

The need for rather elaborate duality arguments in the analysis of the FOSLS may be understood 
as a consequence of the choice of norms in the method, in particular the choice of the computationally
convenient  $L^2$-norm. For example, for right-hand sides that are not in $L^2$, the methodology
requires some regularization of the right-hand side. We refer to \cite{fuehrer-heuer-karkulik22} 
for a method that features a regularization of the 
input data to make minimal residual methods applicable to problems with low-regularity input data. 

The regularity assumptions on the data imposed by the least squares approach in 
$L^2$-based spaces can also be relaxed by changing the norms and performing a minimization in a weaker one. 
This approach also has a substantial history. We refer to the recent \cite{monsuur2023minimal} for such an approach, 
where the relevant dual norms are realized computationally.

An important class of minimum residual methods 
that appeared after the monograph \cite{bochev-gunzburger09} are DPG methods, 
\cite{demkowicz-gopalakrishnan10,demkowicz-gopalakrishnan11,demkowicz-gopalakrishnan11a,
carstensen-demkowicz-gopalakrishnan16}, which may be understood as minimizing the residual in 
a norm other than $L^2$, \cite{demkowicz-gopalakrishnan14}.


\subsubsection*{Contribution of the present work}
Our primary contribution are optimal $L^2(\Omega)$
based convergence results for the least squares approximation of the scalar variable $u$,
the gradient of the scalar variable $u$
and the traces of the scalar and the vector variable. 
Furthermore, we derive improved $L^2(\Omega)$ estimates for the vector variable $\pmb{\varphi}$. 
These estimates are explicit in the mesh
size $h$ of the quasi-uniform meshes employed and the polynomial degree $p$ utilized, which 
extends the results of \cite{ku11}. 

\subsubsection*{Outline}
In Section~\ref{section:extensions_to_robin_boundary} we introduce the model problem, its  
FOSLS formulation and prove a norm equivalence that ensures unique solvability
of both the continuous and the discrete least squares formulation.
Section~\ref{section:duality_argument_robin} provides regularity assertions for the representations 
in terms of a dual formulation of the scalar variable, its gradient, the vector variable, and the 
traces. 
These duality results are given without assuming full elliptic regularity so as to 
provide the tools for a possible extension to situations without full elliptic shift such as 
non-convex geometries. 
In Section~\ref{section:error_analysis_robin} we present
several error estimates for different quantities of interest, namely, 
the $L^2$ and $H^1$-error  of the scalar variable and the $L^2$ and $H(\div)$-error of the 
vectorial variable. This is obtained in a bootstrapping fashion by systematically improving 
estimates for these quantities of interest by repeated duality arguments. 
As a tool for our error analysis, we develop in Lemma~\ref{lemma:properties_of_IhGamma}
the constrained approximation operator $\IhGamma$ with certain orthogonality properties that is instrumental for our error analysis. 
This operator generalizes the corresponding operators $\pmb{I}_h$, $\pmb{I}_h^0$ for problems with homogeneous boundary conditions
in \cite{bernkopf-melenk22}. Compared to \cite{bernkopf-melenk22}, the bootstrapping argument is more 
involved since the presence of the boundary terms in the bilinear lowers the regularity of some components
of dual solutions compared to \cite{bernkopf-melenk22} (see Remark~\ref{remk:difference-to-bernkopf-melenk22} for more details).

We close the paper in 
Section~\ref{section:numerical_examples_robin} with numerical results that 
showcase the proved convergence rates for the case of solutions with finite (low) Sobolev regularity.

\subsubsection*{Notation}
Throughout this work, $\Omega$ denotes a bounded simply connected domain in $\mathbb{R}^d$, $d =2$, $3$
with a (piecewise) smooth boundary $\Gamma \coloneqq \partial \Omega$ and outward unit normal vector $\pmb{n}$. 
We flag that the convergence analysis will be performed under the assumption of a full elliptic regularity
shift, i.e., $\Omega$ is convex or has a smooth boundary $\Gamma$. 

We use standard notation for the gradient, the divergence, and the curl: 
for scalar functions in $d$ variables we write 
$\nabla u = (\partial_{x_1} u, \cdots, \partial_{x_d} u)^{T}$ and for 
$\mathbb{R}^d$-valued functions $\pmb{\varphi}$ the divergence is 
$\nabla \cdot \pmb{\varphi} = \sum_{i=1}^d \partial_{x_i} \pmb{\varphi}_i$.
For $d=3$ the curl operator $\nabla \times $ of a vector field $\pmb{\varphi}$
is $\nabla \times \pmb{\varphi}
  =
  (\partial_{x_2} \pmb{\varphi}_3 - \partial_{x_3} \pmb{\varphi}_2,
  -(\partial_{x_1} \pmb{\varphi}_3 - \partial_{x_3} \pmb{\varphi}_1),
  \partial_{x_1} \pmb{\varphi}_2 - \partial_{x_2} \pmb{\varphi}_1 )^T$. 
In spatial dimension $d=2$ the scalar-valued curl operator
acting on vector fields is given by
$\scalarcurl \, \pmb{\varphi} = \partial_{x_1} \pmb{\varphi}_2 - \partial_{x_2} \pmb{\varphi}_1$ and the
vector-valued curl operator acting on scalar functions is written 
$\nabla \times  u = (\partial_{x_2} u, - \partial_{x_1} u)$.
For $d = 3$, we recall the exact sequence 
\begin{equation*}
  \mathbb{R} \stackrel{\mathrm{id}}{\longrightarrow}
  \HOne          \stackrel{\nabla}{\longrightarrow}
  \HCurl         \stackrel{\nabla \times}{\longrightarrow}
  \HDiv          \stackrel{\nabla \cdot}{\longrightarrow}
  L^2(\Omega)      \stackrel{0}{\longrightarrow}
  \left\{ 0 \right\},
\end{equation*}
and, if homogeneous boundary conditions are imposed, the exact sequence 
\begin{equation*}
  \left\{ 0 \right\} \stackrel{\mathrm{id}}{\longrightarrow}
  \HZeroOne          \stackrel{\nabla}{\longrightarrow}
  \HZeroCurl         \stackrel{\nabla \times}{\longrightarrow}
  \HZeroDiv          \stackrel{\nabla \cdot}{\longrightarrow}
  L^2_0(\Omega)      \stackrel{0}{\longrightarrow}
  \left\{ 0 \right\}.
\end{equation*}
For $d=2$ there are two exact sequences, which are isomorphic to each other via a rotation: 
\begin{align}
\label{eq:exact_sequence_2d_1}
&  \mathbb{R}     \stackrel{\mathrm{id}}{\longrightarrow}
  \HOne          \stackrel{\nabla}{\longrightarrow}
  \HScalarCurl   \stackrel{\scalarcurl}{\longrightarrow}
  L^2(\Omega)    \stackrel{0}{\longrightarrow}
  \left\{ 0 \right\}, \\
\label{eq:exact_sequence_2d_2}
  & \mathbb{R}     \stackrel{\mathrm{id}}{\longrightarrow}
  \HOne          \stackrel{\curl}{\longrightarrow}
  \HDiv          \stackrel{\nabla \cdot}{\longrightarrow}
  L^2(\Omega)    \stackrel{0}{\longrightarrow}
  \left\{ 0 \right\}.
\end{align}
Indeed, with the matrix $ R =
  \big(\begin{smallmatrix}
      0 & 1\\
      -1 & 0
    \end{smallmatrix}\big)
$
the 
sequence in~\eqref{eq:exact_sequence_2d_2}
can be obtained from the sequence in~\eqref{eq:exact_sequence_2d_1}:
\begin{align}
\label{eq:curl-div-2D}
  \scalarcurl \, \pmb{\varphi} &= \nabla \cdot (R \pmb{\varphi})
\qquad \mbox{and} \qquad 
\curl \, u = R \nabla u. 
\end{align}

We employ standard Sobolev spaces $H^s(\Omega)$ in $\Omega$ 
as defined in \cite{mclean00}. For $s \ge 0$, Sobolev spaces $H^s(\Gamma)$ 
on the boundary $\Gamma$ are understood in the standard way as described in
\cite{mclean00} for smooth boundaries $\Gamma$. For Lipschitz boundaries $\Gamma$, 
in particular polygonal/polyhedral boundaries, the intrinsic definition
of \cite{mclean00} is valid for $s \in [0,1]$. For $s > 1$, we define 
$H^s(\Gamma):= \{v|_{\Gamma}\,|\, v \in H^{s+1/2}(\Omega)\}$, and the norm is 
defined by the minimal extension into $\Omega$. We write $(\cdot,\cdot)_{\Omega}$
for the $L^2(\Omega)$ inner product and $\langle \cdot,\cdot\rangle_\Gamma$ for the 
duality pairing that extends the $L^2(\Gamma)$ inner product.

We employ the following standard spaces for $d \in \{2,3\}$:
  \begin{equation*}
    \HDiv     = \{ \pmb{\varphi} \in \pmb{L}^2(\Omega) \colon \nabla \cdot \pmb{\varphi} \in L^2(\Omega) \},        \quad
    \HZeroDiv = \{ \pmb{\varphi} \in \HDiv             \colon \pmb{\varphi} \cdot \pmb{n} = 0 \text{ on } \Gamma \}.
  \end{equation*}

  Furthermore, we introduce spaces related to the vector-valued curl operator in spatial dimension $d=3$ and to 
the scalar valued curl operator in spatial dimension $d=2$:

\noindent\begin{minipage}{.48\linewidth}
  \begin{align*}
    \HCurl     & = \{ \pmb{\varphi} \in \pmb{L}^2(\Omega) \colon \nabla \times \pmb{\varphi} \in \pmb{L}^2(\Omega) \},   \\
    \HZeroCurl & = \{ \pmb{\varphi} \in \HCurl            \colon \pmb{n} \times \pmb{\varphi} = 0 \text{ on } \Gamma \},
  \end{align*}
\end{minipage}%
\quad
\noindent\begin{minipage}{.48\linewidth}
  \begin{align*}
    \HScalarCurl     & = \{ \pmb{\varphi} \in \pmb{L}^2(\Omega)       \colon \scalarcurl \, \pmb{\varphi} \in L^2(\Omega) \},       \\
    \HZeroScalarCurl & = \{ \pmb{\varphi} \in \HScalarCurl            \colon \pmb{t} \cdot \pmb{\varphi} = 0 \text{ on } \Gamma \}.
  \end{align*}
\end{minipage}
\\
where $\pmb{t} = R^T \pmb{n}$ is the corresponding tangential vector in spatial dimension $d=2$. 
Finally, for $s \ge 0$ we require 
the spaces $\pmb{H}^s(\div, \Omega)$ of functions in $\pmb{H}^s(\Omega)$ with divergence also in $H^s(\Omega)$.
We refer \cite{demkowicz-babushka03,monk03,boffi-brezzi-fortin13,melenk-sauter21} for further references and details. 

\subsubsection*{Meshes and finite element spaces}
We consider regular (i.e., no hanging nodes), shape-regular triangulations ${\mathcal T}_h$ of $\Omega$. That is, we assume 
\begin{enumerate}[nosep, label={(\roman*)},leftmargin=*]
\item 
the (open) elements $K \in {\mathcal T}_h$ cover $\Omega$, i.e., $\cup_{K \in {\mathcal T}_h} \overline{K} = \overline{\Omega}$; 
\item associated with each element $K$ is a $C^1$-diffeomorphism $F_K: \widehat{K} \rightarrow K$ from the (fixed) reference simplex $\widehat K$;  
\item denoting $h_K = \operatorname{diam} K$, there holds with the shape-regularity constant $\kappa > 0$ 
$$
h^{-1}_K \|F_K^\prime \|_{L^\infty(\widehat K)} + h_K \|(F_K^\prime)^{-1} \|_{L^\infty(\widehat K)} \leq \kappa; 
$$
\item
the intersection of two elements $K$, $K^\prime \in {\mathcal T}_h$ is either empty or a $j$-face for some $j \in \{0,\ldots,d\}$ 
of both $K$, $K^\prime$, i.e., a vertex, an edge, a facet, or $K = K^\prime$. The parametrization of a common $j$-face is compatible, 
i.e., if $K$, $K^\prime$ share a $j$-face $f$ with $f = F_K(\hat f) = F_{K^\prime}(\hat f')$ for $j$-faces $\hat f$, $\hat f'$ of the 
reference simplex $\widehat K$, then $F_K^{-1} \circ F_{K^\prime}: \hat f' \rightarrow \hat f$ is an affine bijection. 
\end{enumerate}
Throughout the present work, we make the following additional assumptions on the element maps of the triangulation ${\mathcal T}$:  
\begin{assumption}[quasi-uniform regular meshes, \cite{melenk-sauter10}]\label{assumption:quasi_uniform_regular_meshes}
  Let $\widehat{K}$ be the reference simplex.
  Each element map $F_K \colon \widehat{K} \to K$ can be written as $F_K = R_K \circ A_K$, where $A_K$ is an affine map and the maps $R_K$ and $A_K$ satisfy, for constants $C_\mathrm{affine}, C_\mathrm{metric}, \rho > 0$ independent of $K$:
  \begin{equation*}
    \begin{alignedat}{2}
      &\norm{A^\prime_K}_{L^\infty( \widehat{K} )}        \leq C_\mathrm{affine} h_K, \qquad &&\norm{ (A^\prime_K)^{-1} }_{L^\infty( \widehat{K} )} \leq C_\mathrm{affine} h^{-1}_K, \\
      &\norm{ (R^\prime_K)^{-1} }_{L^\infty( \tilde{K} )} \leq C_\mathrm{metric},     \qquad &&\norm{ \nabla^n R_K }_{L^\infty( \tilde{K} )} \leq C_\mathrm{metric} \rho^n n! \qquad \forall n \in \mathbb{N}_0.
    \end{alignedat}
  \end{equation*}
  Here, $\tilde{K} = A_K(\widehat{K})$ and $h_K > 0$ denotes the element diameter.
\end{assumption}
\begin{remark}
The element maps $F_K$ in Assumption~\ref{assumption:quasi_uniform_regular_meshes} are required to be analytic, which is stronger than needed for the 
ensuing analysis. The structure of the element maps $F_K$, namely, the fact that $F_K$ is the concatination of an affine map and a smooth map, provides
a simple mechanism for scaling arguments familiar from affine triangulations. The specific form of Assumption~\ref{assumption:quasi_uniform_regular_meshes}
is taken from \cite{melenk-sauter10}, where the analyticity of the maps $R_K$ is exploited. 
\eremk
\end{remark}

On the reference simplex $\widehat{K}$ we introduce the Raviart-Thomas and Brezzi-Douglas-Marini elements by: 
\begin{align*}
  \mathcal{P}_{p}(\widehat{K})         & \coloneqq \mathrm{span}\left\{ \pmb{x}^{\pmb \alpha} \colon |\pmb{\alpha}| \leq p \right\},                                              \\
  \pmb{\mathrm{RT}}_{p-1}(\widehat{K}) & \coloneqq \left\{ \pmb{p} + \pmb{x}q \colon \pmb{p} \in \mathcal{P}_{p-1}(\widehat{K})^d, q \in \mathcal{P}_{p-1}(\widehat{K}) \right\}, \\
  \pmb{\mathrm{BDM}}_{p}(\widehat{K})  & \coloneqq \mathcal{P}_{p}(\widehat{K})^d.
\end{align*}
The N\'ed\'elec type I and II elements in spatial dimension $d=2$ and $d=3$ are given by: 
\begin{align*}
  \pmb{\mathrm{N}}_{p-1}^{I}(\widehat{K}) & \coloneqq \left\{ \pmb{p} + q (y, -x)^T \colon \pmb{p} \in \mathcal{P}_{p-1}(\widehat{K})^2, q \in \mathcal{P}_{p-1}(\widehat{K}) \right\}
  \text{ for } d=2,                                                                                                                                                                    \\
  \pmb{\mathrm{N}}_{p-1}^{I}(\widehat{K}) & \coloneqq \left\{ \pmb{p} + \pmb{x} \times \pmb{p}  \colon \pmb{q}, \pmb{q} \in \mathcal{P}_{p-1}(\widehat{K})^3 \right\}
  \text{ for } d=3,                                                                                                                                                                    \\
  \pmb{\mathrm{N}}_{p}^{II}(\widehat{K})  & \coloneqq \mathcal{P}_{p}(\widehat{K})^d.
\end{align*}
The appropriate change of variables for functions from $\HDiv$ is the so-called Piola transform: 
For a function $\pmb{\varphi} : K \to \mathbb{R}^d$ and a smooth bijection $F_K \colon \widehat{K} \to K$
its Piola transform $\widehat{\pmb{\varphi}} : \widehat{K} \to \mathbb{R}^d$ is given by
\begin{equation*}
  \widehat{\pmb{\varphi}} = (\det F_K^\prime ) (F_K^\prime)^{-1} \pmb{\varphi} \circ F_K.
\end{equation*}
Based on the mesh $\mathcal{T}_h$, the spaces $S_p(\mathcal{T}_h)$, $\pmb{\mathrm{BDM}}_p(\mathcal{T}_h)$, and $\pmb{\mathrm{RT}}_{p-1}(\mathcal{T}_h)$ 
are given by standard transformation and (contravariant) Piola transformation of functions on the reference element:
\begin{align*}
  S^{-1}_{p}(\mathcal{T}_h)                   & \coloneqq \left\{ u \in L^2(\Omega) \colon \left.\kern-\nulldelimiterspace{u}\vphantom{\big|} \right|_{K} \circ F_K \in \mathcal{P}_{p}(\widehat{K}) \text{ for all } K \in \mathcal{T}_h \right\},                                                                                             \\
  S_{p}(\mathcal{T}_h)                   & \coloneqq \left\{ u \in H^1(\Omega) \colon \left.\kern-\nulldelimiterspace{u}\vphantom{\big|} \right|_{K} \circ F_K \in \mathcal{P}_{p}(\widehat{K}) \text{ for all } K \in \mathcal{T}_h \right\},                                                                                             \\
  \pmb{\mathrm{BDM}}_p(\mathcal{T}_h)    & \coloneqq \left\{ \pmb{\varphi} \in \pmb{H}(\operatorname{div}, \Omega) \colon (\det F_K^\prime) (F_K^\prime)^{-1} \left.\kern-\nulldelimiterspace{\pmb{\varphi}}\vphantom{\big|} \right|_{K} \circ F_K \in \pmb{\mathrm{BDM}}_p(\widehat{K}) \text{ for all } K \in \mathcal{T}_h \right\},    \\
  \pmb{\mathrm{RT}}_{p-1}(\mathcal{T}_h) & \coloneqq \left\{ \pmb{\varphi} \in \pmb{H}(\operatorname{div}, \Omega) \colon (\det F_K^\prime) (F_K^\prime)^{-1} \left.\kern-\nulldelimiterspace{\pmb{\varphi}}\vphantom{\big|} \right|_{K} \circ F_K \in \pmb{\mathrm{RT}}_{p-1}(\widehat{K}) \text{ for all } K \in \mathcal{T}_h \right\}.
\end{align*}
Similarly, for $d = 3$ the N\'ed\'elec elements of type I on $d=3$ are defined as 
\begin{align*}
  \pmb{\mathrm{N}}^{I}_p(\mathcal{T}_h) & \coloneqq \left\{ \pmb{\varphi} \in \HCurl \colon (F_K^\prime)^{T} \left.\kern-\nulldelimiterspace{\pmb{\varphi}}\vphantom{\big|} \right|_{K} \circ F_K \in \pmb{\mathrm{N}}^{I}_p(\widehat{K}) \text{ for all } K \in \mathcal{T}_h \right\};
\end{align*}
we proceed analogously  for N\'ed\'elec type II elements and the case $d=2$.
Observe that for $d = 2$, the N\'ed\'elec elements are just the rotated Raviart-Thomas and Brezzi-Douglas-Marini elements.

We refer to \cite[Prop.~{2.5.4}]{boffi-brezzi-fortin13} as a standard reference for the approximation properties of the above $\HDiv$-conforming spaces in the case of affine elements
and without the $p$-aspect. Under Assumption~\ref{assumption:quasi_uniform_regular_meshes}, the approximation properties of these spaces in an $hp$-context are analyzed in 
\cite[Sec.~{4}]{bernkopf-melenk19}.

The first order system formulation of a second order equation requires us to choose 
two finite element spaces, one for the scalar variable $u$, i.e., the solution of the second order equation, and one for the vector variable $\pmb{\varphi}$, which we 
select as $\pmb{\phi} = -\nabla u$. Hence, for the numerical discretization of the first order system we consider the following finite element spaces:
\begin{equation*}
  \Sp     \subseteq \HOne, \quad
  \SpZero \subseteq \HZeroOne,
\end{equation*}
\begin{equation*}
  \RTBDM       \subseteq \HDiv,  \quad
  \RTBDMZero   \subseteq \HZeroDiv,
\end{equation*}
where the polynomial approximation of the scalar and vector variable is denoted by $p_s \geq 1$ and $p_v \geq 1$.  
We set 
$$
p:= \min\{p_s, p_v\}. 
$$
For brevity of notation, we  denote by $\RTBDM$ either the Raviart-Thomas space $\RT$ or the Brezzi-Douglas-Marini space $\BDM$.
The space $\RTBDMZero$, which includes the boundary conditions,  is understood analogously. 
Furthermore, depending on the choice of the space $\RTBDM$, the N\'ed\'elec space $\Nedelec$ is either of type I (if $\RTBDM = \RT$) or II (if $\RTBDM = \BDM$). 
We apply the same convention to spaces incorporating homogeneous boundary conditions. We refer to
\cite{demkowicz-babushka03,monk03,boffi-brezzi-fortin13,melenk-sauter21} for further details.

\subsubsection*{Notational conventions}

As in \cite{bernkopf-melenk22}, further notational conventions are: 
\begin{itemize}
  \item lower case roman letters such as $u$ and $v$ are employed to denote for scalar valued functions;
  \item lower case boldface greek letters such as $\pmb{\varphi}$ and $\pmb{\psi}$ are used for vector valued functions;
  \item a subscript $h$ as in $u_h$ and $\pmb{\varphi}_h$ indicates membership in a finite element space; 
  \item if not otherwise stated finite element functions without a \, $\tilde{\cdot}$ \, are in some sense fixed, e.g., they are Galerkin approximations
        whereas functions with a \, $\tilde{\cdot}$ \, are arbitrary and arise, e.g., in quasi-optimality results;
  \item generic constants are either denoted by $C>0$ or are hidden inside the symbol $\lesssim$; 
        they are independent of the mesh size $h$ and the polynomial degree $p$ unless otherwise stated.
        We will not track the parameters $\gamma$ and $\alpha$ appearing in the model problem \eqref{eq:model_problem_robin}.  
\end{itemize}

\section{Model problem with Robin boundary conditions}\label{section:extensions_to_robin_boundary}

We assume $f \in L^2(\Omega)$ and $g \in L^2(\Gamma)$.
For fixed $\gamma$, $\alpha > 0$ we consider the following model problem: 
\begin{equation}\label{eq:model_problem_robin}
  \begin{alignedat}{2}
    - \Delta u + \gamma u   &= f \quad   &&\text{in } \Omega,\\
    \partial_n u + \alpha u &= g \quad   &&\text{on } \Gamma.
  \end{alignedat}
\end{equation}
As in \cite{bernkopf-melenk22} by setting $\pmb{\varphi} = -\nabla u$ we arrive at the system
\begin{equation}\label{eq:model_problem_first_order_system_robin}
  \begin{alignedat}{2}
    \nabla \cdot \pmb{\varphi} + \gamma u  &= f  \quad   &&\text{in } \Omega,   \\
    \nabla u + \pmb{\varphi}  &= 0  \quad   &&\text{in } \Omega,   \\
    \pmb{\varphi} \cdot \pmb{n} - \alpha u &= -g \quad   &&\text{on } \Gamma.
  \end{alignedat}
\end{equation}
Furthermore, we introduce the Hilbert spaces
\begin{align*}
\label{eq:spaces-V-W}
  \pmb{V} & \coloneqq \{ \pmb{\varphi} \in \HDiv \colon \pmb{\varphi} \cdot \pmb{n} \in L^2(\Gamma) \} && \text{  and  } & W & \coloneqq \HOne,
\end{align*}
where $\pmb{V}$ is equipped with the graph norm $( \norm{\pmb{\varphi}}_{\HDiv}^2 + \norm{\pmb{\varphi} \cdot \pmb{n}}_{L^2(\Gamma)}^2 )^{1/2}$ 
in order to control the $L^2(\Gamma)$ normal trace.
This is necessary since for general $\pmb{\varphi} \in \HDiv$ one only has $\pmb{\varphi} \cdot \pmb{n} \in H^{-1/2}(\Gamma)$.
The bilinear form $b$ and the linear functional $F$ are given by 
\begin{align*}
  b( (\pmb{\varphi}, u), (\pmb{\psi}, v) )
   & \coloneqq ( \nabla \cdot \pmb{\varphi} + \gamma u , \nabla \cdot \pmb{\psi} + \gamma v )_{\Omega}
  + ( \nabla u + \pmb{\varphi} , \nabla v + \pmb{\psi} )_{\Omega}
  + \langle \pmb{\varphi} \cdot \pmb{n} - \alpha u , \pmb{\psi} \cdot \pmb{n} - \alpha v \rangle_{\Gamma}, \\
  F( (\pmb{\varphi}, v) )
   & \coloneqq ( f,  \nabla \cdot \pmb{\psi} + \gamma v )_{\Omega}
  + \langle -g , \pmb{\psi} \cdot \pmb{n} - \alpha v \rangle_{\Gamma}.
\end{align*}
\begin{remark}
\label{remk:difference-to-bernkopf-melenk22}
The boundary terms in the bilinear form $b$ reflect the fact that the Robin boundary conditions, which reduce 
to Neumann condition in the case $\alpha =0$, are realized as ``natural boundary conditions'' in the $L^2$-minimization process.         
In comparison, the bilinear form of \cite{bernkopf-melenk22} for the Neumann problem lacks the boundary terms since
the Neumann conditions are enforced as ``essential boundary conditions'' on the vectorial variable. This difference in bilinear 
form leads to slightly different regularity assertions for dual problems in Section~\ref{section:duality_argument_robin} ahead: see
the regularity of $\pmb{\psi}$ in Theorem~\ref{theorem:duality_argument_robin}, which asserts up to $\pmb{H}^2$-regularity for $\pmb{\psi}$ 
whereas the corresponding \cite [Thm.~3.3]{bernkopf-melenk22} even asserts up to $\pmb{H}^3$-regularity. 
\eremk
\end{remark}

We start our analysis with a norm equivalence theorem.

\begin{theorem}[Norm equivalence - Robin version of {\cite[Thm.~2.1]{bernkopf-melenk22}}]\label{theorem:norm_equivalence_robin}
  For all $(\pmb{\varphi}, u) \in \productspacerobin$ there holds
  \begin{equation*}
    \norm{u}_{\HOne}^2 + \norm{\pmb{\varphi}}_{\HDiv}^2 + \norm{\pmb{\varphi} \cdot \pmb{n}}_{L^2(\Gamma)}^2
    \lesssim b( (\pmb{\varphi}, u), (\pmb{\varphi}, u) )
    \lesssim \norm{u}_{\HOne}^2 + \norm{\pmb{\varphi}}_{\HDiv}^2 + \norm{\pmb{\varphi} \cdot \pmb{n}}_{L^2(\Gamma)}^2.
  \end{equation*}
\end{theorem}

\begin{proof}
The upper bound follows directly from the Cauchy-Schwarz inequality. For the lower bound, we proceed similarly
to \cite[Thm.~{2.1}]{bernkopf-melenk22}. 
  By definition we have
  \begin{equation*}
    b( (\pmb{\varphi}, u), (\pmb{\varphi}, u) ) =
    \| \underbrace{ \nabla \cdot \pmb{\varphi} + \gamma u}_{\eqqcolon w} \|_{L^2(\Omega)}^2 +
    \| \underbrace{ \nabla u + \pmb{\varphi} }_{\eqqcolon \pmb{\eta}} \|_{L^2(\Omega)}^2 +
    \| \underbrace{ \pmb{\varphi} \cdot \pmb{n} - \alpha u }_{\eqqcolon \mu} \|_{L^2(\Gamma)}^2.
  \end{equation*}
We write $\pmb{\varphi} = \pmb{\varphi}_1 + \pmb{\varphi}_2$ and $u = u_1 + u_2$, where 
  \\
  \\
  \noindent\begin{minipage}{.5\linewidth}
    \begin{equation*}
      \begin{alignedat}{2}
        \nabla \cdot \pmb{\varphi}_1 + \gamma u_1 &= w \qquad  &&\text{in } \Omega,   \\
        \nabla u_1 + \pmb{\varphi}_1 &= 0         &&\text{in } \Omega,   \\
        \pmb{\varphi}_1 \cdot \pmb{n} - \alpha u_1 &= 0         &&\text{on } \Gamma,
      \end{alignedat}
    \end{equation*}
  \end{minipage}%
  \begin{minipage}{.5\linewidth}
    \begin{equation*}
      \begin{alignedat}{2}
        \nabla \cdot \pmb{\varphi}_2 + \gamma u_2 &= 0 \qquad   &&\text{in } \Omega,   \\
        \nabla u_2 + \pmb{\varphi}_2 &= \pmb{\eta} &&\text{in } \Omega,   \\
        \pmb{\varphi}_2 \cdot \pmb{n} - \alpha u_2 &= \mu        &&\text{on } \Gamma.
      \end{alignedat}
    \end{equation*}
  \end{minipage}
\\

Eliminating $\pmb{\varphi}_1$, $\pmb{\varphi}_2$ leads, in strong form, to 
\\

  \noindent\begin{minipage}{.5\linewidth}
    \begin{equation*}
      \begin{alignedat}{2}
        -\Delta u_1 + \gamma u_1 &= w \qquad  &&\text{in } \Omega, \\
        \partial_n u_1 + \alpha u_1 &= 0         &&\text{on } \Gamma,
      \end{alignedat}
    \end{equation*}
  \end{minipage}%
  \begin{minipage}{.5\linewidth}
    \begin{equation*}
      \begin{alignedat}{2}
        -\Delta u_2 + \gamma u_2 &= - \nabla \cdot \pmb{\eta} \qquad  &&\text{in } \Omega,   \\
        \partial_n u_2 + \alpha u_2 &= - \mu + \pmb{\eta} \cdot \pmb{n}                             &&\text{on } \Gamma.
      \end{alignedat}
    \end{equation*}
  \end{minipage}
  \\
  \\
Lax-Milgram provides $\|u_1\|_{H^1(\Omega)} \lesssim \|w\|_{L^2(\Omega)}$ and 
$\|u_2\|_{H^1(\Omega)} \lesssim \|\pmb{\eta} \|_{L^2(\Omega)} + \|\mu\|_{H^{-1/2}(\Gamma)}$. 
We set $\pmb{\varphi}_1 = - \nabla u_1$ and $\pmb{\varphi}_2 = \pmb{\eta} - \nabla u_2$ and check that 
the pairs $(\pmb{\varphi}_1, u_1)$ and $(\pmb{\varphi}_2, u_2)$ satisfy the above two systems
and $\|\pmb{\varphi}_1\|_{\HDiv} \lesssim \|w\|_{L^2(\Omega)}$ 
as well as $\|\pmb{\varphi}_2\|_{\HDiv} \lesssim \|\pmb{\eta}\|_{L^2(\Omega)} + \|\mu\|_{H^{-1/2}(\Gamma)}$. 
We note $\|\pmb{\varphi}_1 \cdot \pmb{n}\|_{L^2(\Gamma)} \lesssim \|u_1\|_{L^2(\Gamma)}$ and 
$\|\pmb{\varphi}_2 \cdot \pmb{n}\|_{L^2(\Gamma)} \lesssim \|\mu\|_{L^2(\Gamma)} + \|u_2\|_{L^2(\Gamma)}$.  
We conclude the proof by observing that $\pmb{\varphi} = \pmb{\varphi_1} + \pmb{\varphi}_2$ and 
$u = u_1 + u_2$.  
\end{proof}

\section{Duality argument}\label{section:duality_argument_robin}

Our error analysis in the following sections relies on several duality arguments for a range of quantities. 
The following Assumption~\ref{assumption:smax_shift} characterizes the range in which the elliptic 
shift theorem is valid: 
\begin{assumption}[$\smax$ shift property]\label{assumption:smax_shift}
  Let $\hat s \ge -1$. Then for every
  $f \in H^s(\Omega)$,
  $g \in H^{s+1/2}(\Gamma)$ and
  $s \in [-1,\hat s]$, the problem
  \begin{equation*}
    \begin{alignedat}{2}
      - \Delta u + \gamma u   &= f \quad   &&\text{in } \Omega,\\
      \partial_n u + \alpha u &= g \quad   &&\text{on } \Gamma
    \end{alignedat}
  \end{equation*}
  admits the
  regularity shift $u \in H^{s+2}(\Omega)$ with
  $\|u\|_{H^{s+2}(\Omega)} \lesssim \|f\|_{H^s(\Omega)} + \|g\|_{H^{s+1/2}(\Gamma)}$
  if $s \ge 0$ and,
  if $s < 0$,
  $ \|u\|_{H^{s+2}(\Omega)} \lesssim \|f\|_{\widetilde H^{s}(\Omega)} + \|g\|_{H^{s+1/2}(\Gamma)}$.
  Here, for $s \in (-1,0)$, we set
  $H^{s}(\Omega) = (\widetilde{H}^{-s}(\Omega) )^\prime$,
  $\widetilde H^{s}(\Omega) = \left(H^{-s}(\Omega) \right)^\prime $
  with the Sobolev spaces $H^{-s}(\Omega) = (L^2(\Omega), H^1(\Omega))_{-s,2}$
  and $\widetilde H^{-s}(\Omega) = (L^2(\Omega), H^1_0(\Omega))_{-s,2}$
  defined by the real method of interpolation (see \cite{mclean00} for details).
\end{assumption}

\begin{remark}\label{remark:regularity_shift_for_neumann_problem}
\begin{enumerate}[nosep, label={(\roman*)},leftmargin=*]
\item 
\label{item:remark:regularity_shift_for_neumann_problem-i}
The parameter $\hat s$ encodes properties of $\Gamma$. For example, 
by standard elliptic regularity Assumption~\ref{assumption:smax_shift} holds for any $\hat s \ge 0$ for smooth boundaries $\Gamma$. 
For polygonal/polyhedral domains, $\hat s$ is determined by the minimal angle at corners in 2D 
or corners and edges in 3D. For convex domains $\Omega$, 
one has $\hat s \ge 0$. The implied constants in the norm estimates additionally depend on $\gamma$ and $\alpha$. 
\item 
\label{item:remark:regularity_shift_for_neumann_problem-ii}
  Under Assumption~\ref{assumption:smax_shift} 
  an analogous regularity shift also holds for the Neumann problem
  \begin{equation*}
    \begin{alignedat}{2}
      - \Delta u       &= f \quad   &&\text{in } \Omega,\\
      \partial_n u     &= g \quad   &&\text{on } \Gamma
    \end{alignedat}
  \end{equation*}
  for $f \in H^s(\Omega)$ and $g \in H^{s+1/2}(\Gamma)$ satisfying
  the compatibility condition $\int_\Omega f \, \mathrm{d}x + \int_\Gamma g \, \mathrm{d}s = 0$.
  This can be easily seen by adding $\gamma u$ and $\alpha u$ in the formulation of the Neumann problem 
  to the volume and boundary terms, respectively, and then applying Assumption~\ref{assumption:smax_shift}.
\eremk
\end{enumerate}
\end{remark}

\begin{theorem}[Duality argument for the scalar variable --- Robin version of {\cite[Thm.~3.3]{bernkopf-melenk22}}]\label{theorem:duality_argument_robin}
  Let Assumption~\ref{assumption:smax_shift} be valid for some $\hat s \ge -1$. 
  Then, given $(\pmb{\varphi}, w) \in \productspacerobin$ there is a pair $(\pmb{\psi}, v) \in \productspacerobin$
  with $\norm{w}_{L^2(\Omega)}^2 = b( (\pmb{\varphi}, w), (\pmb{\psi}, v) )$.
  Furthermore, $\pmb{\psi} \in \pmb{H}^{\min(\smax+1, 2)}(\Omega)$,
  $\nabla \cdot \pmb{\psi} \in H^{\min(\smax+2, 2)}(\Omega)$, 
  $\pmb{\psi} \cdot \pmb{n} \in H^{\min(\smax+3/2, 3/2)}(\Gamma)$,
  and $v \in H^{\min(\smax+2, 2)}(\Omega)$ with 
  \begin{equation*}
    \norm{v}_{H^{\min(\smax+2, 2)}(\Omega)} +
    \norm{\pmb{\psi}}_{H^{\min(\smax+1, 2)}(\Omega)} +
    \norm{\nabla \cdot \pmb{\psi}}_{H^{\min(\smax+2, 2)}(\Omega)}  +
    \norm{\pmb{\psi} \cdot \pmb{n}}_{H^{\min(\smax+3/2, 3/2)}(\Gamma)} \lesssim \norm{w}_{L^2(\Omega)}.
  \end{equation*}
\end{theorem}

\begin{proof}
By the coercivity result of Theorem~\ref{theorem:norm_equivalence_robin} and Lax-Milgram, there is a unique 
solution $(\pmb{\psi},v) \in \productspacerobin$ to the following variational problem: 
  \begin{equation}
 \label{eq:theorem:duality_argument_robin-10}
    (u, w)_\Omega = b( (\pmb{\varphi}, u), (\pmb{\psi}, v) ) ~~~ \forall \, (\pmb{\varphi}, u) \in  \productspacerobin. 
  \end{equation}
In order to show the regularity assertions about $(\pmb{\psi},v)$, we 
  introduce the new quantities $z$, $\pmb{\mu}$, and $\sigma$ by 
  \begin{equation}\label{eq:actual_dual_problem_u_robin}
    \begin{alignedat}{2}
      \nabla \cdot \pmb{\psi} + \gamma v   &= z            \qquad  &\text{in } \Omega, \\
      \nabla v + \pmb{\psi}                &= \pmb{\mu}    \qquad  &\text{in } \Omega, \\
      \pmb{\psi} \cdot \pmb{n} - \alpha v  &= \sigma       \qquad  &\text{on } \Gamma.
    \end{alignedat}
  \end{equation}
In terms of these quantities, 
 (\ref{eq:theorem:duality_argument_robin-10}) reads 
  \begin{equation}\label{eq:axiliary_actual_dual_problem_u_robin}
    (u, w)_\Omega =
    (\nabla u + \pmb{\varphi}, \pmb{\mu})_\Omega +
    (\nabla \cdot \pmb{\varphi} + \gamma u, z)_\Omega +
    \langle \pmb{\varphi} \cdot \pmb{n} - \alpha u , \sigma \rangle_{\Gamma}
    ~~~ \forall \, (\pmb{\varphi}, u) \in  \productspacerobin.
  \end{equation}
  Selecting $u = 0$, we find with an integration by parts 
  \begin{equation*}
    0 = (\pmb{\varphi}, \pmb{\mu})_\Omega + (\nabla \cdot \pmb{\varphi}, z)_\Omega + \langle \pmb{\varphi} \cdot \pmb{n} , \sigma \rangle_{\Gamma} =
    (\pmb{\varphi}, \pmb{\mu} - \nabla z)_\Omega + \langle \pmb{\varphi} \cdot \pmb{n} , \sigma + z \rangle_{\Gamma}, 
  \end{equation*}
  which gives $\pmb{\mu} = \nabla z$ as well as $\sigma = -z|_\Gamma$. Therefore we find by taking $\pmb{\varphi} = 0$
  in (\ref{eq:axiliary_actual_dual_problem_u_robin})
  \begin{equation*}
    ( u,  w)_\Omega = (\nabla u, \nabla z)_\Omega + (\gamma u, z)_\Omega + \langle \alpha u, z \rangle_{\Gamma}~~~ \forall \, u \in \HOne.
  \end{equation*}
  That is, $z$ satisfies, in strong form, 
  \begin{equation}\label{eq:auxiliary_dual_problem_robin}
    \begin{alignedat}{2}
      - \Delta z + \gamma z &= w \quad   &&\text{in } \Omega,\\
      \partial_n z + \alpha z &= 0         &&\text{on } \Gamma.
    \end{alignedat}
  \end{equation}
  Assumption~\ref{assumption:smax_shift} provides 
  $z \in H^{\min(\smax+2, 2)}(\Omega)$ together with the estimate $\norm{z}_{H^{\min(\smax+2, 2)}(\Omega)} \lesssim \norm{w}_{L^2(\Omega)}$.
  We next proceed as in the proof of \cite[Thm.~3.3]{bernkopf-melenk22}.
  To highlight the fact that $\pmb{\psi}$ is only in $\pmb{H}^{\min(\smax+1, 2)}(\Omega)$ compared to \cite[Thm.~3.3]{bernkopf-melenk22} 
we write down the equations for $v$ and $z-v$:
      \begin{align*}
        - \Delta v + \gamma v &= w + (1-\gamma)z &&\text{in } \Omega,
&\qquad 
        - \Delta (z-v) + \gamma (z-v) &= (\gamma - 1)z &&\text{in } \Omega,\\ 
        \partial_n v + \alpha v &= (1-\alpha)z             &&\text{on } \Gamma,
&\qquad 
        \partial_n (z-v) + \alpha (z-v) &= (\alpha - 1)z         &&\text{on } \Gamma.
      \end{align*}
  Assumption~\ref{assumption:smax_shift} gives $v \in H^{\min(\smax+2, 2)}(\Omega)$ since the volume right-hand side
is only in $L^2(\Omega)$.
  The regularity of $z-v$ is limited by the exploitable regularity of the boundary data 
$(\alpha - 1)z \in H^{\min(\smax+2, 2)-1/2}(\Gamma)  = H^{\min(\smax+1,1) + 1/2}(\Gamma) \subset H^{\min( \min(\smax+1,1),\smax)+1/2}(\Gamma)
 = H^{\min(\smax,1) +1/2}(\Gamma)$. 
  Therefore, by Assumption~\ref{assumption:smax_shift},
  we have $z-v \in H^{\smax+2}(\Omega)  = H^{\min(\smax+2,3)}(\Omega)$ together with the estimate
  \begin{equation*}
    \norm{z-v}_{H^{\min(\smax+2, 3)}(\Omega)} \lesssim \norm{w}_{L^2(\Omega)},
  \end{equation*}
  and consequently $\pmb{\psi} = \nabla (z-v) \in \pmb{H}^{\min(\smax+1, 2)}(\Omega)$.
The regularity of $\nabla \cdot \pmb{\psi}$ now follows from the representation (\ref{eq:actual_dual_problem_u_robin})$_1$
and that of $\pmb{\psi} \cdot \pmb{n}$ from (\ref{eq:actual_dual_problem_u_robin})$_3$ and $\sigma = -z|_\Gamma$. For this last regularity
assertion, we recall that Sobolev spaces $H^s(\Gamma)$ with $s > 1$ are defined in terms of traces of Sobolev functions
of $H^{s+1/2}(\Omega)$. 
\end{proof}

\begin{theorem}[Duality argument for the gradient of the scalar variable - Robin version of {\cite[Thm.~3.4]{bernkopf-melenk22}}]\label{theorem:duality_argument_grad_u_robin}
  Let Assumption~\ref{assumption:smax_shift} be valid for some $\hat s \ge -1$. 
  Then, given $(\pmb{\varphi}, w) \in \productspacerobin$ there is a pair $(\pmb{\psi}, v) \in \productspacerobin$
  with $\norm{\nabla w}_{L^2(\Omega)}^2 = b( (\pmb{\varphi}, w), (\pmb{\psi}, v) )$.
  Furthermore, $\pmb{\psi} \in \pmb{H}^{\min(\smax+1, 1)}(\Omega)$,
  $\nabla \cdot \pmb{\psi} \in H^1(\Omega)$, $\pmb{\psi} \cdot \pmb{n} \in H^{1/2}(\Gamma)$, and $v \in H^1(\Omega)$ with 
  \begin{equation*}
    \norm{v}_{H^1(\Omega)} +
    \norm{\pmb{\psi}}_{H^{\min(\smax+1, 1)}(\Omega)} +
    \norm{\nabla \cdot \pmb{\psi}}_{H^1(\Omega)} + \norm{\pmb{\psi}\cdot\pmb{n}}_{H^{1/2}(\Gamma)} 
\lesssim \norm{\nabla w}_{L^2(\Omega)}.
  \end{equation*}
\end{theorem}

\begin{proof}
By the coercivity result of Theorem~\ref{theorem:norm_equivalence_robin} and Lax-Milgram, there is a unique 
$(\pmb{\psi},v) \in \productspacerobin$ satisfying 
  \begin{equation}
\label{eq:theorem:duality_argument_grad_u_robin-10}
    (\nabla u, \nabla w)_\Omega = b( (\pmb{\varphi}, u), (\pmb{\psi}, v) ) ~~~ \forall \, (\pmb{\varphi}, u) \in  \productspacerobin.
  \end{equation}
  In order to prove the regularity assertion, we introduce $z$, $\pmb{\mu}$, and $\sigma$ by 
  \begin{equation}\label{eq:actual_dual_problem_grad_u_robin}
    \begin{alignedat}{2}
      \nabla \cdot \pmb{\psi} + \gamma v   &= z            \qquad  &\text{in } \Omega, \\
      \nabla v + \pmb{\psi}                &= \pmb{\mu}    \qquad  &\text{in } \Omega, \\
      \pmb{\psi} \cdot \pmb{n} - \alpha v  &= \sigma       \qquad  &\text{on } \Gamma.
    \end{alignedat}
  \end{equation}
In terms of these quantities, (\ref{eq:theorem:duality_argument_grad_u_robin-10}) reads
  \begin{equation}\label{eq:axiliary_actual_dual_problem_grad_u_robin}
    (\nabla u, \nabla w)_\Omega =
    (\nabla u + \pmb{\varphi}, \pmb{\mu})_\Omega +
    (\nabla \cdot \pmb{\varphi} + \gamma u, z)_\Omega +
    \langle \pmb{\varphi} \cdot \pmb{n} - \alpha u , \sigma \rangle_{\Gamma}
    ~~~ \forall \, (\pmb{\varphi}, u) \in  \productspacerobin.
  \end{equation}
  Selecting $u = 0$, we find after an integration by parts 
  \begin{equation*}
    0 = (\pmb{\varphi}, \pmb{\mu})_\Omega + (\nabla \cdot \pmb{\varphi}, z)_\Omega + \langle \pmb{\varphi} \cdot \pmb{n} , \sigma \rangle_{\Gamma} =
    (\pmb{\varphi}, \pmb{\mu} - \nabla z)_\Omega + \langle \pmb{\varphi} \cdot \pmb{n} , \sigma + z \rangle_{\Gamma}, 
  \end{equation*}
  which gives $\pmb{\mu} = \nabla z$ and $\sigma = -z|_\Gamma$. Next, selecting $\pmb{\varphi} = 0$
in (\ref{eq:axiliary_actual_dual_problem_grad_u_robin}) we see 
  \begin{equation*}
    (\nabla u, \nabla w)_\Omega = (\nabla u, \nabla z)_\Omega + (\gamma u, z)_\Omega + \langle \alpha u, z \rangle_{\Gamma}~~~ \forall \, u \in \HOne.
  \end{equation*}
  Viewing this as an equation for $z$, the Lax-Milgram theorem gives the estimate $\norm{z}_{H^1(\Omega)} \lesssim \norm{\nabla w}_{L^2(\Omega)}$.
  In fact, in strong form, $z$ satisfies
  \begin{equation}\label{eq:auxiliary_dual_problem_grad_u_robin}
    \begin{alignedat}{2}
      - \Delta z + \gamma z &= - \nabla \cdot \nabla w \quad   &&\text{in } \Omega,\\
      \partial_n z + \alpha z &= \nabla w \cdot \pmb{n}         &&\text{on } \Gamma.
    \end{alignedat}
  \end{equation}
  We next may proceed as in the proof of \cite[Thm.~{3.4}]{bernkopf-melenk22}.
  The equations satisfied by $v$ and $z-v$ are easily derived:
      \begin{align*}
        - \Delta v + \gamma v &= (1-\gamma)z - \nabla \cdot \nabla w   \quad   &&\text{in } \Omega,
&\qquad 
        - \Delta (z-v) + \gamma (z-v) &= (\gamma - 1)z \quad   &&\text{in } \Omega,\\
        \partial_n v + \alpha v &= (1-\alpha)z + \nabla w \cdot \pmb{n}              &&\text{on } \Gamma,
&\qquad 
        \partial_n (z-v) + \alpha (z-v) &= (\alpha - 1)z         &&\text{on } \Gamma.
      \end{align*}
 Again by the Lax-Milgram theorem we have $v \in H^1(\Omega)$ with $\norm{v}_{H^1(\Omega)} \lesssim \norm{\nabla w}_{L^2(\Omega)}$.
  The regularity of $z-v$ is limited by the
  exploitable regularity of the boundary data
  $(\alpha - 1)z \in H^{1/2}(\Gamma)$.
  Assumption~\ref{assumption:smax_shift} implies 
  $z-v \in H^{\min(\smax+2, 2)}(\Omega)$ with the estimate
  \begin{equation*}
    \norm{z-v}_{H^{\min(\smax+2, 2)}(\Omega)} \lesssim \norm{\nabla w}_{L^2(\Omega)},
  \end{equation*}
  and consequently $\pmb{\psi} = \nabla (z-v) \in \pmb{H}^{\min(\smax+1, 1)}(\Omega)$.
  The regularity of $\nabla \cdot \pmb{\psi}$ follows from 
  (\ref{eq:actual_dual_problem_grad_u_robin})$_{1}$ and that of $\pmb{\psi}\cdot\pmb{n}$ from 
  (\ref{eq:actual_dual_problem_grad_u_robin})$_{3}$.  
\end{proof}

\begin{theorem}[Duality argument for the vector valued variable --- Robin version of {\cite[Thm.~3.5]{bernkopf-melenk22}}]\label{theorem:duality_argument_phi_robin}
  Let Assumption~\ref{assumption:smax_shift} be valid for some $\hat s \ge -1$.
  Then, given $(\pmb{\eta}, u) \in \productspacerobin$ there is a pair $(\pmb{\psi}, v) \in \productspacerobin$
  with $\norm{\pmb{\eta}}_{L^2(\Omega)}^2 = b( (\pmb{\eta}, u), (\pmb{\psi}, v) )$.
  Furthermore, $\pmb{\psi} \in \pmb{L}^2(\Omega)$,
  $\nabla \cdot \pmb{\psi} \in H^1(\Omega)$,
  $\pmb{\psi} \cdot \pmb{n} \in H^{1/2}(\Gamma)$, and $v \in H^{\min(\smax+2, 2)}(\Omega)$ with 
  \begin{equation*}
    \norm{v}_{H^{\min(\smax+2, 2)}(\Omega)}  +
    \norm{\pmb{\psi}}_{L^2(\Omega)} +
    \norm{\nabla \cdot \pmb{\psi}}_{H^1(\Omega)} +
    \norm{\pmb{\psi} \cdot \pmb{n}}_{H^{1/2}(\Gamma)}  \lesssim \norm{\pmb{\eta}}_{L^2(\Omega)}.
  \end{equation*}
\end{theorem}

\begin{proof}
By the coercivity result of Theorem~\ref{theorem:norm_equivalence_robin} and Lax-Milgram, there is a unique 
$(\pmb{\psi},v) \in \productspacerobin$ satisfying 
  \begin{equation}
\label{eq:theorem:duality_argument_phi_robin-10}
    (\pmb{\varphi}, \pmb{\eta})_\Omega = b( (\pmb{\varphi}, u), (\pmb{\psi}, v) ) ~~~ \forall \, (\pmb{\varphi}, u) \in  \productspacerobin.
  \end{equation}
  To show the stated regularity assertions, we introduce the abbreviations $z$, $\pmb{\mu}$, and $\sigma$ by 
  \begin{equation}\label{eq:actual_dual_problem_phi_robin}
    \begin{alignedat}{2}
      \nabla \cdot \pmb{\psi} + \gamma v   &= z            \qquad  &\text{in } \Omega, \\
      \nabla v + \pmb{\psi}                &= \pmb{\mu}    \qquad  &\text{in } \Omega, \\
      \pmb{\psi} \cdot \pmb{n} - \alpha v  &= \sigma       \qquad  &\text{on } \Gamma. 
    \end{alignedat}
  \end{equation}
In terms of these quantities, (\ref{eq:theorem:duality_argument_phi_robin-10}) reads 
  \begin{equation}\label{eq:axiliary_actual_dual_problem_phi_robin}
    (\pmb{\varphi}, \pmb{\eta})_\Omega =
    (\nabla u + \pmb{\varphi}, \pmb{\mu})_\Omega +
    (\nabla \cdot \pmb{\varphi} + \gamma u, z)_\Omega +
    \langle \pmb{\varphi} \cdot \pmb{n} - \alpha u , \sigma \rangle_{\Gamma}
    ~~~ \forall \, (\pmb{\varphi}, u) \in  \productspacerobin.
  \end{equation}
  Selecting $u = 0$ we find by an integration by parts 
  \begin{equation*}
    (\pmb{\varphi}, \pmb{\eta})_\Omega
    = (\pmb{\varphi}, \pmb{\mu})_\Omega + (\nabla \cdot \pmb{\varphi}, z)_\Omega + \langle \pmb{\varphi} \cdot \pmb{n} , \sigma \rangle_{\Gamma}
    = (\pmb{\varphi}, \pmb{\mu} - \nabla z)_\Omega + \langle \pmb{\varphi} \cdot \pmb{n} , \sigma + z \rangle_{\Gamma}, 
  \end{equation*}
  which gives $\pmb{\mu} - \nabla z = \pmb{\eta}$ and $\sigma = -z|_\Gamma$. Choosing $\pmb{\varphi} = 0$ in 
  (\ref{eq:axiliary_actual_dual_problem_phi_robin}), we find 
  \begin{equation*}
    0 = (\nabla u, \pmb{\eta} + \nabla z)_\Omega + (\gamma u, z)_\Omega + \langle \alpha u, z \rangle_{\Gamma} ~~~ \forall \, u \in \HOne.
  \end{equation*}
  Viewing this as an equation for $z$, we infer from the Lax-Milgram theorem the estimate $\norm{z}_{H^1(\Omega)} \lesssim \norm{\pmb{\eta}}_{L^2(\Omega)}$.
  In fact $z$ satisfies, in strong form, 
  \begin{equation}\label{eq:auxiliary_dual_problem_phi_robin}
    \begin{alignedat}{2}
      - \Delta z + \gamma z &= \nabla \cdot \pmb{\eta} \quad   &&\text{in } \Omega,\\
      \partial_n z + \alpha z &= -\pmb{\eta} \cdot \pmb{n}         &&\text{on } \Gamma.
    \end{alignedat}
  \end{equation}
  The equations satisfied by $v$ are easily derived:
  \begin{equation*}
    \begin{alignedat}{2}
      - \Delta v + \gamma v &= (1-\gamma)z  \quad   &&\text{in } \Omega,\\
      \partial_n v + \alpha v &= (1-\alpha)z          &&\text{on } \Gamma.
    \end{alignedat}
  \end{equation*}
  Assumption~\ref{assumption:smax_shift} gives $v \in H^{\min(\smax+2, 2)}(\Omega)$ with the estimate
  \begin{equation*}
    \norm{v}_{H^{\min(\smax+2, 2)}(\Omega)} \lesssim  \norm{(1-\gamma)z}_{L^2(\Omega)} + \norm{(1-\alpha)z}_{H^{1/2}(\Gamma)} \lesssim \norm{z}_{H^1(\Omega)} \lesssim \norm{\pmb{\eta}}_{L^2(\Omega)}.
  \end{equation*}
  Finally, we have $\pmb{\psi} = \pmb{\eta} + \nabla (z-v) \in \pmb{L}^2(\Omega)$.
  The asserted regularity of $\nabla \cdot \pmb{\psi}$ follows from (\ref{eq:actual_dual_problem_phi_robin})$_{1}$
and that of $\pmb{\psi} \cdot \pmb{n}$ from (\ref{eq:actual_dual_problem_phi_robin})$_{3}$. 
\end{proof}

\begin{theorem}[Duality argument for the normal trace of the vector valued variable]\label{theorem:duality_argument_normal_trace_robin}
  Let Assumption~\ref{assumption:smax_shift} be valid for some $\hat s \ge -1$.
  Then, given $(\pmb{\eta}, u) \in \productspacerobin$ there is a pair $(\pmb{\psi}, v) \in \productspacerobin$
  with $\norm{\pmb{\eta} \cdot \pmb{n}}_{L^2(\Gamma)}^2 = b( (\pmb{\eta}, u), (\pmb{\psi}, v) )$.
  Furthermore, $\pmb{\psi} \in \pmb{H}^{\min(\smax+1, 1/2)}(\Omega)$,
  $\nabla \cdot \pmb{\psi} \in H^{\min(\smax+2, 3/2)}(\Omega)$,
  $\pmb{\psi} \cdot \pmb{n} \in L^2(\Gamma)$, and
  $v \in H^{\min(\smax+2, 3/2)}(\Omega)$ with 
  \begin{equation*}
    \norm{v}_{H^{\min(\smax+2, 3/2)}(\Omega)} +
    \norm{\pmb{\psi}}_{^{\min(\smax+1, 1/2)}(\Omega)}  +
    \norm{\nabla \cdot \pmb{\psi}}_{H^{\min(\smax+2, 3/2)}(\Omega)}  +
    \norm{\pmb{\psi} \cdot \pmb{n}}_{L^{2}(\Gamma)}  \lesssim \norm{\pmb{\eta} \cdot \pmb{n}}_{L^2(\Gamma)}.
  \end{equation*}
\end{theorem}

\begin{proof}
  By the coercivity result of Theorem~\ref{theorem:norm_equivalence_robin} and Lax-Milgram, there is a unique 
$(\pmb{\psi},v) \in \productspacerobin$ satisfying 
  \begin{equation} \label{eq:theorem:duality_argument_normal_trace_phi_robin}
    \langle \pmb{\varphi} \cdot \pmb{n}, \pmb{\eta} \cdot \pmb{n} \rangle_\Gamma = b( (\pmb{\varphi}, u), (\pmb{\psi}, v) ) ~~~ \forall \, (\pmb{\varphi}, u) \in  \productspacerobin.
  \end{equation}
  To show the regularity assertions, we introduce the new quantities $z$, $\pmb{\mu}$, and $\sigma$ by 
  \begin{equation}\label{eq:actual_dual_problem_normal_trace_phi_robin}
    \begin{alignedat}{2}
      \nabla \cdot \pmb{\psi} + \gamma v   &= z            \qquad  &\text{in } \Omega, \\
      \nabla v + \pmb{\psi}                &= \pmb{\mu}    \qquad  &\text{in } \Omega, \\
      \pmb{\psi} \cdot \pmb{n} - \alpha v  &= \sigma       \qquad  &\text{on } \Gamma. 
    \end{alignedat}
  \end{equation}
In terms of these quantities, (\ref{eq:theorem:duality_argument_normal_trace_phi_robin}) reads 
  \begin{equation}\label{eq:axiliary_actual_dual_problem_normal_trace_phi_robin}
    \langle \pmb{\varphi} \cdot \pmb{n}, \pmb{\eta} \cdot \pmb{n} \rangle_\Gamma =
    (\nabla u + \pmb{\varphi}, \pmb{\mu})_\Omega +
    (\nabla \cdot \pmb{\varphi} + \gamma u, z)_\Omega +
    \langle \pmb{\varphi} \cdot \pmb{n} - \alpha u , \sigma \rangle_{\Gamma}
    ~~~ \forall \, (\pmb{\varphi}, u) \in  \productspacerobin.
  \end{equation}
  Selecting $u = 0$, we find after an integration by parts 
  \begin{equation*}
    \langle \pmb{\varphi} \cdot \pmb{n}, \pmb{\eta} \cdot \pmb{n} \rangle_\Gamma 
    = (\pmb{\varphi}, \pmb{\mu})_\Omega + (\nabla \cdot \pmb{\varphi}, z)_\Omega + \langle \pmb{\varphi} \cdot \pmb{n} , \sigma \rangle_{\Gamma}
    = (\pmb{\varphi}, \pmb{\mu} - \nabla z)_\Omega + \langle \pmb{\varphi} \cdot \pmb{n} , \sigma + z \rangle_{\Gamma} 
\qquad \forall \pmb{\varphi} \in \pmb{V}, 
  \end{equation*}
  which gives $\pmb{\mu} = \nabla z$ as well as $\sigma = \pmb{\eta} \cdot \pmb{n}-z$. Therefore we find with $\pmb{\varphi} = 0$ in 
  (\ref{eq:axiliary_actual_dual_problem_normal_trace_phi_robin})
  \begin{equation*}
    0 = (\nabla u, \nabla z)_\Omega + (\gamma u, z)_\Omega + \langle \alpha u, z  \rangle_{\Gamma} - \langle \alpha u, \pmb{\eta} \cdot \pmb{n}  \rangle_{\Gamma} ~~~ \forall \, u \in \HOne.
  \end{equation*}
  That is, $z$ satisfies, in strong form, 
  \begin{equation*}
    \begin{alignedat}{2}
      - \Delta z + \gamma z &= 0 \quad   &&\text{in } \Omega,\\
      \partial_n z + \alpha z &= \alpha \, \pmb{\eta} \cdot \pmb{n}   \quad     &&\text{on } \Gamma.
    \end{alignedat}
  \end{equation*}
  Assumption~\ref{assumption:smax_shift} provides 
  $z \in H^{\min(\smax+2, 3/2)}(\Omega)$ with the estimate $\norm{z}_{H^{\min(\smax+2, 3/2)}(\Omega)} \lesssim \norm{\pmb{\eta} \cdot \pmb{n}}_{L^2(\Gamma)}$.
  The equations satisfied by $v$ are easily seen to be: 
  \begin{equation*}
    \begin{alignedat}{2}
      - \Delta v + \gamma v &= (1-\gamma)z  \quad   &&\text{in } \Omega,\\
      \partial_n v + \alpha v &= (1-\alpha) (z - \pmb{\eta} \cdot \pmb{n} )   \quad &&\text{on } \Gamma.
    \end{alignedat}
  \end{equation*}
  Assumption~\ref{assumption:smax_shift} gives $v \in H^{\min(\smax+2, 3/2)}(\Omega)$ with the estimate 
  $ \norm{v}_{H^{\min(\smax+2, 3/2)}(\Omega)} \lesssim  \norm{\pmb{\eta} \cdot \pmb{n}}_{L^2(\Gamma)}$.
  Finally, we have $\pmb{\psi} = \nabla (z-v) \in \pmb{H}^{\min(\smax+1, 1/2)}(\Omega)$.
  The regularity of $\nabla \cdot \pmb{\psi}$ follows from 
  (\ref{eq:actual_dual_problem_normal_trace_phi_robin})$_{1}$ and that of 
$\pmb{\psi} \cdot \pmb{n}$ from (\ref{eq:actual_dual_problem_normal_trace_phi_robin})$_{3}$. 
\end{proof}

\begin{remark}\label{remark:reason_for_duality_argument_boundary}
  Usually a duality argument results in a dual solution with higher order Sobolev regularity.
  This is not the case in Theorem~\ref{theorem:duality_argument_normal_trace_robin}, where the regularity is \emph{not} improved, since $\pmb{\psi} \cdot \pmb{n}$ is still only in $L^2(\Gamma)$.
  The sole purpose of this duality argument is to introduce another Galerkin orthogonality that can be utilized in the error analysis to overcome the limiting regularity of the boundary data.
\eremk
\end{remark}

\section{Error analysis}\label{section:error_analysis_robin}
From here on we will only consider domains $\Omega$ satisfying Assumption~\ref{assumption:smax_shift} for some $\smax \geq 0$
such as domains with smooth boundary $\Gamma$ or convex domains. Non-convex polygonal/polyhedral domains would require a more 
careful analysis. 

  After recalling results about a commuting diagram operator in 
  Subsection~\ref{subsection:commuting_diagram_operator} we 
  proceed in Subsection~\ref{subsection:operator_IhGamma} with introducing and analyzing 
the operator $\IhGamma$, which features an orthogonality necessary
for our analysis. 
  Finally, we prove different error estimates in
  Subsection~\ref{subsection:main_error_estimates}
  via a bootstrap argument.
  We first prove suboptimal estimates for the errors 
  $\norm{e^u}_{L^2(\Omega)}$,
  $\norm{\pmb{e}^{\pmb{\varphi}}}_{L^2(\Omega)}$, and
  $\norm{\nabla e^u}_{L^2(\Omega)}$
  in
  Lemma~\ref{lemma:e_u_suboptimal_l2_error_estimate_robin},
  Theorem~\ref{theorem:e_phi_suboptimal_l2_error_estimate_robin}, and
  Lemma~\ref{theorem:grad_e_u_suboptimal_l2_error_estimate_robin}, 
  respectively, where $e^u = u - u_h$ and $\pmb{e}^{\pmb{\varphi}} = \pmb{\varphi} - \pmb{\varphi}_h$
are the FOSLS errors of the scalar and vectorial variable.
We then prove optimal estimates for
  $\norm{\pmb{e}^{\pmb{\varphi}} \cdot \pmb{n}}_{L^2(\Gamma)}$ and
  $\norm{\nabla e^u}_{L^2(\Omega)}$
  in
  Theorems~\ref{theorem:e_phi_normal_trace_l2_error_estimate_robin} and
  \ref{theorem:grad_e_u_optimal_l2_error_estimate_robin}.  
  Next, we derive in Theorem~\ref{theorem:e_phi_suboptimal_improved_l2_error_estimate_robin}
improved estimates for
  $\norm{\pmb{e}^{\pmb{\varphi}}}_{L^2(\Omega)}$ that are numerically seen to be still suboptimal. 
Finally, we conclude with an optimal estimate for
  $\norm{e^u}_{L^2(\Omega)}$
  in Theorem~\ref{theorem:e_u_optimal_l2_error_estimate_robin}.
\subsection{A commuting diagram operator}\label{subsection:commuting_diagram_operator}
In the analysis it is crucial to understand the approximation properties of the vector-valued finite element space in the classical $\HDiv$ norm as well as the $L^2(\Gamma)$ norm of the normal trace simultaneously.
We are therefore interested in quantifying
\begin{equation*}
\inf_{\tilde{\pmb{\psi}}_h \in \RTBDM}
  \| \pmb{\psi} - \tilde{\pmb{\psi}}_h \|_{\HDiv} + \| (\pmb{\psi} - \tilde{\pmb{\psi}}_h) \cdot \pmb{n} \|_{L^2(\Gamma)}
\end{equation*}
for $\pmb{\psi} \in \pmb{V}$. This infimum can be estimated by specific approximants. 
For the reader's convenience we briefly summarize some results of \cite{melenk-rojik18} concerning the $\HDiv$-conforming elementwise defined approximation operator 
$\Pi_{p_v}^{\mathrm{div}}: \pmb{H}^{1/2}(\div,\Omega) \rightarrow \RTBDM$ 
constructed therein. This operator is defined on the reference element with error estimates that are explicit in the polynomial degree $p_v$. 
A simple scaling argument gives the desired $h$ estimates of the global operator:

\begin{proposition}[Defs.~{2.3}, {2.6}, Thms.~{2.10}, {2.13}, \& Rem.~{2.9} in \cite{melenk-rojik18}]\label{proposition:melenk_rojik_operator}
  The global operator $\Pi_{p_v}^{\mathrm{div}}$ satisfies for every $\pmb{\varphi} \in \pmb{H}^{1/2}(\div,\Omega)$ and $\tilde{\pmb{\varphi}}_h \in \RTBDM$,
  \begin{enumerate}[label=(\roman*)]
    \item 
          \label{melenk_rojik_operator_prop_1}
$ (\nabla \cdot (\pmb{\varphi} - \Pi_{p_v}^{\mathrm{div}} \pmb{\varphi} ) , \nabla \cdot \tilde{\pmb{\varphi}}_h )_{\Omega} = 0 $ and consequently
          $ \| \nabla \cdot (\pmb{\varphi} - \Pi_{p_v}^{\mathrm{div}} \pmb{\varphi}) \|_{L^2(\Omega)} \leq \|  \nabla \cdot (\pmb{\varphi} - \tilde{\pmb{\varphi}}_h) \|_{L^2(\Omega)} $,
    \item 
\label{melenk_rojik_operator_prop_2}
$ \langle (\pmb{\varphi} - \Pi_{p_v}^{\mathrm{div}} \pmb{\varphi} ) \cdot \pmb{n}, \tilde{\pmb{\varphi}}_h \cdot \pmb{n} \rangle_{\Gamma} = 0 $ and consequently
          $ \| (\pmb{\psi} - \Pi_{p_v}^{\mathrm{div}}\pmb{\varphi}) \cdot \pmb{n} \|_{L^2(\Gamma)} \leq \| (\pmb{\varphi} - \tilde{\pmb{\varphi}}_h) \cdot \pmb{n} \|_{L^2(\Gamma)} $,      
    \item 
          \label{melenk_rojik_operator_prop_3}
$ \| \pmb{\varphi} - \Pi_{p_v}^{\mathrm{div}} \pmb{\varphi} \|_{\HDiv} \lesssim \left(\frac{h}{p_v}\right)^{1/2} \| \pmb{\varphi} - \tilde{\pmb{\varphi}}_h \|_{\pmb{H}^{1/2}(\mathrm{div}, \Omega)} $.
  \end{enumerate}
\end{proposition}

\begin{proof}
The operator $\Pi^{\div}_{p_v}$ is constructed in \cite{melenk-rojik18} and 
\cite{rojik20}. However, while \cite{melenk-rojik18} covered in detail 
the case of $\RTBDM = \pmb{\mathrm{RT}}_{p_v-1}$, the fact that 
the approximation properties and the commuting diagram properties are 
also valid for the choice $\RTBDM = \pmb{\mathrm{BDM}}_{p_v}$  is 
discussed in \cite[Sec.~{4.8}]{rojik20}.

  For the 2D case, \cite{melenk-rojik18} consider the operator $\operatorname{curl}$ instead of $\operatorname{div}$. However, 
in view of (\ref{eq:curl-div-2D}), the results of \cite{melenk-rojik18} can be reformulated in terms of the operator $\operatorname{div}$.

  Property~\ref{melenk_rojik_operator_prop_2} and~\ref{melenk_rojik_operator_prop_3} can be found in \cite{melenk-rojik18} for the case 
$\RTBDM = \pmb{\mathrm{RT}}_{p_v-1}$. 
Property~\ref{melenk_rojik_operator_prop_1} follows from the commuting diagram property of $\Pi_{p_V}^{\mathrm{div}}$: 
for any $\tilde{\pmb{\varphi}}_h \in \RTBDM$ we calculate with the 
$L^2$-projection $\Pi^{L^2}_{p_v}: L^2(\Omega) \rightarrow S^{-1}_{p_v}$ and using that $\nabla \cdot \tilde{\pmb{\varphi}}_h \in S^{-1}_{p_v}$: 
  \begin{equation*}
    (\nabla \cdot (\pmb{\varphi} - \Pi_{p_v}^{\mathrm{div}} \pmb{\varphi} ) , \nabla \cdot \tilde{\pmb{\varphi}}_h )_{\Omega}
    = (\nabla \cdot \pmb{\varphi} - \Pi_{p_v}^{L^2} \nabla \cdot \pmb{\varphi} ) , \nabla \cdot \tilde{\pmb{\varphi}}_h )_{\Omega}
    = 0.
  \end{equation*}
  This orthogonality then implies the last inequality.
\end{proof}

\subsection{The operator $\IhGamma$}\label{subsection:operator_IhGamma}
We will require an approximation operator with certain orthogonality properties, i.e., an operator similar to $\IhZero$ and $\Ih$ constructed in \cite[Sec.~4]{bernkopf-melenk22}.
Although the operator $\Ih$ of \cite{bernkopf-melenk22} is applicable to derive improved convergence results for the present case of Robin boundary conditions, they are 
only optimal in a pure $h$-version of the FOSLS method and suboptimal in a $p$-version context. 
This is due to the fact that the analysis requires approximation properties of $\Ih$ in the $L^2(\Gamma)$ norm for the normal trace, which can only be effected by relying on
inverse estimates.  Even though, \textsl{per se},  these inverse estimates are sharp one loses an order of $p$ when doing so.
For optimal $p$-estimate, it is therefore necessary to define the operator $\IhGamma$ such that the normal trace is appropriately involved.

\paragraph{Construction of $\IhGamma$:}

In the following we construct an operator which \textit{sees} the $L^2(\Gamma)$ normal trace.
We define $\IhGamma$ by a constrained minimization problem:

\begin{equation*}
  \IhGamma \pmb{\varphi} = \underset{\pmb{\varphi}_h \in \RTBDM }{\mathrm{argmin}}
  \frac{1}{2}\norm{ \pmb{\varphi} - \pmb{\varphi}_h }_{L^2(\Omega)}^2 +
  \frac{1}{2}\norm{ (\pmb{\varphi} - \pmb{\varphi}_h)\cdot \pmb{n} }_{L^2(\Gamma)}^2
\end{equation*}
\begin{equation*}
  \text{s.t.} \quad
  ( \nabla \cdot (\pmb{\varphi} - \IhGamma \pmb{\varphi}) , \nabla \cdot \pmb{\chi}_h )_{\Omega} = 0 , \quad \forall \pmb{\chi}_h \in \RTBDM.
\end{equation*}
To simplify the notation we introduce the scalar product $\triscalar{\cdot}{\cdot}$
and the induced norm $\trinorm{\cdot}$ on $\pmb{V} = \{ \pmb{\varphi} \in \HDiv \colon \pmb{\varphi} \cdot \pmb{n} \in L^2(\Gamma) \}$:
\begin{equation}
\label{eq:triscalar}
  \triscalar{\pmb{\varphi}}{\pmb{\psi}}
  \coloneqq ( \pmb{\varphi} , \pmb{\psi} )_{\Omega}
  + \langle \pmb{\varphi} \cdot \pmb{n} , \pmb{\psi} \cdot \pmb{n} \rangle_{\Gamma}.
\end{equation}
Therefore we can write the operator $\IhGamma$ as
\begin{equation*}
  \IhGamma \pmb{\varphi} = \underset{\pmb{\varphi}_h \in \RTBDM }{\mathrm{argmin}}
  \frac{1}{2}\trinorm{ \pmb{\varphi} - \pmb{\varphi}_h }^2
  \qquad \text{s.t.} \qquad
  ( \nabla \cdot (\pmb{\varphi} - \IhGamma \pmb{\varphi}) , \nabla \cdot \pmb{\chi}_h )_{\Omega} = 0 , \quad \forall \pmb{\chi}_h \in \RTBDM.
\end{equation*}
The variational formulation is now given by:
Find $(\pmb{\varphi}_h, \lambda_h) \in \RTBDM \times \nabla \cdot \RTBDM$ such that
\begin{align}
\label{eq:def-IhGamma-1}
    \triscalar{\pmb{\varphi}_h}{\pmb{\mu}_h} + ( \nabla \cdot \pmb{\mu}_h , \lambda_h )_{\Omega}  &= \triscalar{\pmb{\varphi}}{\pmb{\mu}_h}  \qquad \forall \pmb{\mu}_h    \in \RTBDM,\\
\label{eq:def-IhGamma-2}
    ( \nabla \cdot \pmb{\varphi}_h , \eta_h )_{\Omega}                                            &= ( \nabla \cdot \pmb{\varphi} , \eta_h )_{\Omega}  \qquad \forall \eta_h \in \nabla \cdot \RTBDM.
\end{align}
\\
\noindent
\textbf{Coercivity on the kernel}:
Let $\pmb{\mu} \in \left\{ \pmb{\psi} \in \pmb{V} \colon (\nabla \cdot \pmb{\psi}, \eta)_{\Omega} = 0, \forall \eta \in \nabla \cdot \pmb{V} \right\}$ be given.
Coercivity is trivial since by construction $ (\nabla \cdot \pmb{\mu}, \nabla \cdot \pmb{\mu} )_{\Omega} = 0 $ and thus 
\begin{equation*}
  \triscalar{\pmb{\mu}}{\pmb{\mu}} = \trinorm{\pmb{\mu}}^2 = \trinorm{\pmb{\mu}}^2 + \norm{\nabla \cdot \pmb{\mu}}_{L^2(\Omega)}^2 = \norm{\pmb{\mu}}_{\pmb{V}}^2.
\end{equation*}
\\
\noindent
\textbf{inf-sup condition}:
Let $\eta \in \nabla \cdot \pmb{V}$ be given.
Let $u \in H^1(\Omega)$ with zero average solve
\begin{equation*}
  \begin{alignedat}{2}
    - \Delta u   &= \eta                \quad   &&\text{in } \Omega,\\
    \partial_n u &= -\frac{1}{|\Gamma|} \int_\Omega \eta \, \mathrm{d}x \quad   &&\text{on } \Gamma.
  \end{alignedat}
\end{equation*}
By Assumption~\ref{assumption:smax_shift} and Remark~\ref{remark:regularity_shift_for_neumann_problem} we have $\norm{u}_{H^2(\Omega)} \lesssim \norm{\eta}_{L^2(\Omega)}$.
Consequently, with $\pmb{\mu} \coloneqq - \nabla u$
we have $\pmb{\mu} \in \pmb{V}$ and $ \norm{\pmb{\mu}}_{\pmb{V}} \lesssim \norm{\eta}_{L^2(\Omega)} $.
Finally we have
\begin{equation*}
  (\nabla \cdot \pmb{\mu}, \eta)_{\Omega} =
  (\eta, \eta)_{\Omega} =
  \norm{\eta}_{L^2(\Omega)} \norm{\eta}_{L^2(\Omega)} \gtrsim
  \norm{\eta}_{L^2(\Omega)} \norm{\pmb{\mu}}_{\pmb{V}},
\end{equation*}
which proves the inf-sup condition.
\\
\noindent
\textbf{Coercivity on the kernel - discrete}:
The coercivity follows by the same arguments as in the continuous case. 
\\
\noindent
\textbf{inf-sup condition - discrete}:
Let $\lambda_h \in \nabla \cdot \RTBDM$ be given.
As above in the continuous case we solve the Neumann problem
\begin{equation*}
  \begin{alignedat}{2}
    - \Delta u &= \lambda_h \quad   &&\text{in } \Omega,\\
    \partial_n u &= -\frac{1}{|\Gamma|} \int_\Omega \lambda_h \, \mathrm{d}x \quad   &&\text{on } \Gamma.
  \end{alignedat}
\end{equation*}
Consequently, with $\pmb{\Lambda} \coloneqq - \nabla u$, we again have $\norm{\pmb{\Lambda}}_{\pmb{V}} \lesssim \norm{\pmb{\Lambda}}_{\HDiv} + \norm{\pmb{\Lambda} \cdot \pmb{n}}_{L^2(\Gamma)} \leq \norm{u}_{H^2(\Omega)} \lesssim \norm{\lambda_h}_{L^2(\Omega)}$.
We next utilize the commuting diagram projection-based interpolation operator $\pmb{\Pi}^{\div}_{p_v}$ defined in \cite{melenk-rojik18},
see also Proposition~\ref{proposition:melenk_rojik_operator}.
We use this operator to project $\pmb{\Lambda}$ onto the conforming subspace:
With $\pmb{\Lambda}_h \coloneqq \pmb{\Pi}^{\div}_{p_v} \pmb{\Lambda}$ we find
\begin{equation*}
  \nabla \cdot \pmb{\Lambda}_h = \nabla \cdot \pmb{\Pi}^{\div}_{p_v} \pmb{\Lambda} = \pmb{\Pi}^{L^2}_{p_v} \nabla \cdot \pmb{\Lambda} =  \pmb{\Pi}^{L^2}_{p_v} \lambda_h = \lambda_h,
\end{equation*}
where $\pmb{\Pi}^{L^2}_{p_v}$ denotes the $L^2$ orthogonal projection on $\nabla \cdot \RTBDM$.
Using \cite[Thm.~{2.10} (vi)/Thm.~{2.13} (iv)]{melenk-rojik18} we can estimate
\begin{equation*}
  \| \pmb{\Lambda} - \pmb{\Pi}^{\div}_{p_v} \pmb{\Lambda} \|_{\HDiv}
  \lesssim \norm{\pmb{\Lambda}}_{H^1(\Omega)}
  \lesssim \norm{\lambda_h}_{L^2(\Omega)},
\end{equation*}
Additionally, since $\pmb{\Pi}^{\div}_{p_v}$ realizes the $L^2(\Gamma)$ orthogonal projection of the normal trace, we find
\begin{equation*}
  \| (\pmb{\Lambda} - \pmb{\Pi}^{\div}_{p_v} \pmb{\Lambda})\cdot \pmb{n} \|_{L^2(\Gamma)}
  \leq \| \pmb{\Lambda} \cdot \pmb{n} \|_{L^2(\Gamma)}
  \lesssim \norm{\lambda_h}_{L^2(\Omega)},
\end{equation*}
which finally leads to
\begin{equation*}
  \| \pmb{\Lambda}_h \|_{\pmb{V}}
  = \| \pmb{\Pi}^{\div}_{p_v} \pmb{\Lambda} \|_{\pmb{V}}
  \lesssim \| \pmb{\Lambda} - \pmb{\Pi}^{\div}_{p_v} \pmb{\Lambda} \|_{\pmb{V}} + \| \pmb{\Lambda} \|_{\pmb{V}}
  \lesssim  \norm{\lambda_h}_{L^2(\Omega)}.
\end{equation*}
For any $\lambda_h \in \nabla \cdot \RTBDM$ we estimate
\begin{equation*}
  \sup_{\pmb{\varphi}_h \in \RTBDM} \frac{(\nabla \cdot \pmb{\varphi}_h, \lambda_h)_{\Omega}}{\norm{\pmb{\varphi}_h}_{\pmb{V}} \norm{\lambda_h}_{L^2(\Omega)}}
  \geq \frac{(\nabla \cdot \pmb{\Lambda}_h, \lambda_h)_{\Omega}}{\norm{\pmb{\Lambda}_h}_{\pmb{V}} \norm{\lambda_h}_{L^2(\Omega)}}
  = \frac{ \norm{\lambda_h}_{L^2(\Omega)} }{\norm{\pmb{\Lambda}_h}_{\pmb{V}}}
  \gtrsim 1,
\end{equation*}
which proves the discrete inf-sup condition.
We have therefore proven
\begin{lemma}
  The operator $\IhGamma$ is well-defined.
\end{lemma}

\subsection{Helmholtz decompositions}
As a tool in the $L^2(\Omega)$ analysis of the operator $\IhGamma$ we need the following decomposition.
Compared to \cite[Sec.~{4}]{bernkopf-melenk22} we need a Helmholtz-like decomposition accounting for the regularity of the normal trace:

\begin{lemma}[Continuous and discrete Helmholtz-like decomposition - $L^2(\Gamma)$ normal trace]\label{lemma:helmholtz_decomp_l2_normal_trace}
  Let $\Omega$ satisfy Assumption~\ref{assumption:smax_shift} for some $\smax \geq 0$. Let $\MyHCurl \subset \HCurl$ be given by
  \begin{equation*}
    \MyHCurl \coloneqq \left\{ \pmb{\mu} \in \HCurl \colon (\nabla \times \pmb{\mu}) \cdot \pmb{n} \in L^2(\Gamma) \right\}.
  \end{equation*}
  The operators
  $\pmb{\Pi}^{\mathrm{curl},\Gamma} \colon \pmb{V} \to \nabla \times \MyHCurl$
  and
  $\pmb{\Pi}^{\mathrm{curl},\Gamma}_h \colon \pmb{V}\to \nabla \times \Nedelec$
  given by
  \begin{align}
\label{eq:lemma:helmholtz_decomp_l2_normal_trace-1}
    \triscalar{\pmb{\Pi}^{\mathrm{curl},\Gamma}   \pmb{\varphi}}{\nabla \times \pmb{\mu}}        & =  \triscalar{\pmb{\varphi}}{\nabla \times \pmb{\mu}} \quad \forall \pmb{\mu} \in  \MyHCurl,      \\ 
\label{eq:lemma:helmholtz_decomp_l2_normal_trace-2}
    \triscalar{\pmb{\Pi}^{\mathrm{curl},\Gamma}_h   \pmb{\varphi} }{\nabla \times \pmb{\mu}_h} & =  \triscalar{\pmb{\varphi}}{\nabla \times \pmb{\mu}_h} \quad \forall \pmb{\mu}_h \in \Nedelec, 
  \end{align}
  are well-defined.
  The remainder $\pmb{r}$ in the continuous decomposition $\pmb{\varphi} = \pmb{\Pi}^{\mathrm{curl},\Gamma} \pmb{\varphi} + \pmb{r}$ satisfies
  $\pmb{r} \in \pmb{H}^1(\Omega)$ with $\norm{\pmb{r}}_{H^1(\Omega)} \lesssim \norm{\nabla \cdot \pmb{\varphi}}_{L^2(\Omega)}$.
  Additionally the solution $R \in H^2(\Omega)$ of 
  \begin{equation*}
    \begin{alignedat}{2}
      - \Delta R       &= - \nabla \cdot \pmb{\varphi} \quad   &&\text{in } \Omega, \\
      \partial_n R + R &= 0                                    &&\text{on } \Gamma.
    \end{alignedat}
  \end{equation*}
satisfies $\pmb{r} = \nabla R$ together with 
  $\norm{R}_{H^2(\Omega)} \lesssim \norm{\pmb{r}}_{H^1(\Omega)} \lesssim \norm{\nabla \cdot \pmb{\varphi}}_{L^2(\Omega)}$.  
Furthermore $\pmb{r}$ satisfies
  \begin{align}
\label{eq:lemma:helmholtz_decomp_l2_normal_trace-6}
      \nabla \cdot  \pmb{r}  &= \nabla \cdot \pmb{\varphi} && \text{in } \Omega, \\
\label{eq:lemma:helmholtz_decomp_l2_normal_trace-7}
      \nabla \times \pmb{r}  &= 0                            && \text{in } \Omega, \\
\label{eq:lemma:helmholtz_decomp_l2_normal_trace-8}
      \pmb{r} \cdot \pmb{n}  &= -R                             && \text{on } \Gamma.
  \end{align}
\end{lemma}

\begin{proof}
  The unique solvability of (\ref{eq:lemma:helmholtz_decomp_l2_normal_trace-1}), (\ref{eq:lemma:helmholtz_decomp_l2_normal_trace-2})
on the discrete and continuous level follows immediately from the fact
  that the variational formulations are just the definition of the orthogonal projections onto $\nabla \times \MyHCurl$ and $\nabla \times \Nedelec$, respectively.
  For any $\pmb{\mu} \in \pmb{C}_0^\infty(\Omega)$ we find
  \begin{equation*}
    \triscalar{\pmb{r}}{\nabla \times \pmb{\mu}} = ( \pmb{r} , \nabla \times \pmb{\mu} )_{\Omega} = 0,
  \end{equation*}
  which gives $\nabla \times \pmb{r} = 0$. Since $\pmb{\Pi}^{\mathrm{curl},\Gamma} \pmb{\varphi} \in \nabla \times \MyHCurl$ we conclude $\nabla \cdot  \pmb{r} = \nabla \cdot \pmb{\varphi}$.
  The fact that $\nabla \times \pmb{r} = 0$ gives via the exact sequence property 
  \begin{equation*}
    \mathbb{R}       \stackrel{\mathrm{id}}{\longrightarrow}
    \HOne            \stackrel{\nabla}{\longrightarrow}
    \HCurl           \stackrel{\nabla \times}{\longrightarrow}
    \HDiv            \stackrel{\nabla \cdot}{\longrightarrow}
    L^2(\Omega)      \stackrel{0}{\longrightarrow}
    \left\{ 0 \right\}
  \end{equation*}
  the existence of a potential $R \in \HOne$ with $\pmb{r} = \nabla R$. The function $R$ is determined up to a constant that we will fix shortly.
  $\pmb{r} = \nabla R$ implies $ - \Delta R = - \nabla \cdot \nabla R = - \nabla \cdot \pmb{r} = - \nabla \cdot \pmb{\varphi}$.
  To analyze the boundary conditions satisfied by $R$ we insert $\pmb{r} = \nabla R$ into the variational formulation and integrate by parts to get 
  \begin{equation*}
    0 =
    \triscalar{\nabla R}{\nabla \times \pmb{\mu}} =
    ( \nabla R , \nabla \times \pmb{\mu} )_{\Omega} + \langle \partial_n R , (\nabla \times \pmb{\mu}) \cdot \pmb{n} \rangle_{\Gamma} =
    \langle R +  \partial_n R  , (\nabla \times \pmb{\mu}) \cdot \pmb{n} \rangle_{\Gamma}.
  \end{equation*}
  Since $(\nabla \times \pmb{\mu}) \cdot \pmb{n} = \nabla_{\Gamma} \cdot (\pmb{\mu} \times \pmb{n})$ and $\Gamma$ is connected, we conclude $\partial_n R + R = c$ for some $c \in \mathbb{R}$.
Since $R$ is fixed up to a constant, we select it such that $c = 0$. Hence, the function $R$ satisfies the boundary value problem of the statement of the lemma. 
By Assumption~\ref{assumption:smax_shift} and Remark~\ref{remark:regularity_shift_for_neumann_problem} we have $\|R\|_{H^2(\Omega)} \lesssim \|\nabla \cdot \pmb{\varphi}\|_{L^2(\Omega)}$. 
  This concludes the proof.
\end{proof}

\begin{lemma}\label{lemma:properties_of_IhGamma}
  Let $\Omega$ satisfy Assumption~\ref{assumption:smax_shift} for some $\smax \geq 0$. The operator $\IhGamma$ satisfies the following estimates
  \begin{align}
\label{eq:lemma:properties_of_IhGamma-10}
    \trinorm{\pmb{\varphi} - \IhGamma \pmb{\varphi}}
                                                                              & \lesssim \trinorm{\pmb{\varphi} - \tilde{\pmb{\varphi}}_h}
    + \frac{h}{p_v} \norm{\nabla \cdot (\pmb{\varphi} - \tilde{\pmb{\varphi}}_h) }_{L^2(\Omega)},                                                                 \\
\label{eq:lemma:properties_of_IhGamma-20}
    \| \nabla \cdot (\pmb{\varphi} - \IhGamma \pmb{\varphi}) \|_{L^2(\Omega)} & \leq \norm{\nabla \cdot (\pmb{\varphi} - \tilde{\pmb{\varphi}}_h) }_{L^2(\Omega)}
  \end{align}
  for any $\tilde{\pmb{\varphi}}_h \in \RTBDM$.
\end{lemma}

\begin{proof}
  The proof parallels the one of \cite[Lem.~{4.6}]{bernkopf-melenk22} by replacing 
  $\norm{\cdot}_{L^2(\Omega)}$ with $\trinorm{\cdot}$.
  Let $\tilde{\pmb{\varphi}}_h \in \RTBDM$ be arbitrary.
  The orthogonality relations enforced in the construction of the operator $\IhGamma$ readily imply the estimate (\ref{eq:lemma:properties_of_IhGamma-20}).  
  We have with $\pmb{e} = \pmb{\varphi} - \IhGamma \pmb{\varphi}$
  \begin{equation*}
    \trinorm{\pmb{e}}^2 = \triscalar{\pmb{e}}{\pmb{\varphi} - \tilde{\pmb{\varphi}}_h} + \triscalar{\pmb{e}}{\tilde{\pmb{\varphi}}_h - \IhGamma \pmb{\varphi}}.
  \end{equation*}
  Lemma~\ref{lemma:helmholtz_decomp_l2_normal_trace} allows us to decompose the discrete object $\tilde{\pmb{\varphi}}_h - \IhGamma \pmb{\varphi} \in \RTBDM$ 
on a discrete as well as a continuous level:
  \begin{align*}
    \tilde{\pmb{\varphi}}_h - \IhGamma \pmb{\varphi} & = \nabla \times \pmb{\mu} + \pmb{r},    \\
    \tilde{\pmb{\varphi}}_h - \IhGamma \pmb{\varphi} & = \nabla \times \pmb{\mu}_h + \pmb{r}_h
  \end{align*}
  for certain $\pmb{\mu} \in \pmb{Y}$, $\pmb{r} \in \pmb{V}$, $\pmb{\mu}_h \in \Nedelec$, and $\pmb{r}_h \in \RTBDM$.
  Since $\nabla \cdot \nabla \times = 0$, property (\ref{eq:def-IhGamma-1}) of $\IhGamma$ immediately gives 
  \begin{equation*}
    \triscalar{\pmb{\varphi} - \IhGamma \pmb{\varphi}}{\nabla \times \pmb{\mu}_h} = 0.
  \end{equation*}
  We therefore have
  \begin{equation*}
    \triscalar{\pmb{e}}{\tilde{\pmb{\varphi}}_h - \IhGamma \pmb{\varphi}}
    = \triscalar{\pmb{e}}{\nabla \times \pmb{\mu}_h + \pmb{r}_h}
    = \triscalar{\pmb{e}}{ \pmb{r}_h}
    = \triscalar{\pmb{e}}{ \pmb{r}_h - \pmb{r}} + \triscalar{\pmb{e}}{\pmb{r}}
    \coloneqq T_1 + T_2. 
  \end{equation*}
  \noindent
Before continuing with the treatment of the terms $T_1$ and $T_2$, we collect that Lemma~\ref{lemma:helmholtz_decomp_l2_normal_trace}
states that $\pmb{r} = \nabla R$ with 
\begin{align}
\label{eq:properties-of-R-10}
\|R\|_{H^2(\Omega)} &\lesssim \|\nabla \cdot (\pmb{\varphi}_h - \IhGamma \pmb{\varphi})\|_{L^2(\Omega)}, \\
\label{eq:properties-of-R-20}
-\Delta R & = - \nabla \cdot (\pmb{\varphi}_h - \IhGamma \pmb{\varphi}) \mbox{ in $\Omega$}, & \partial_n R + R & = 0 \quad \mbox{ on $\Gamma$}, \\
\label{eq:properties-of-R-30}
\nabla R & = \pmb{r}, & \pmb{r} \cdot \pmb{n} & = -R \quad \mbox{ on $\Gamma$}. 
\end{align}
  \textbf{Treatment of $T_1$}:
  See \cite[Lem.~{4.6}]{bernkopf-melenk22} for completely analogous arguments and more details.
  Since $\nabla \cdot \pmb{r} = \nabla \cdot \pmb{r}_h \in \nabla \cdot \RTBDM$ we find using the commuting diagram as well as the projection property of the operator $\pmb{\Pi}^{\div}_{p_v}$
  \begin{equation*}
    \nabla \cdot (\pmb{\Pi}^{\div}_{p_v} \pmb{r} - \pmb{r}_h) = \pmb{\Pi}^{L^2}_{p_v} (\nabla \cdot \pmb{r}) - \nabla \cdot \pmb{r}_h = \nabla \cdot \pmb{r} - \nabla \cdot \pmb{r}_h = 0.
  \end{equation*}
  By the exact sequence property we thus have $\pmb{\Pi}^{\div}_{p_v} \pmb{r} - \pmb{r}_h \in \nabla \times \Nedelec$.
  The definition of $\pmb{r}$ and $\pmb{r}_h$ in Lemma~\ref{lemma:helmholtz_decomp_l2_normal_trace} provides the orthogonality
  \begin{equation}
\label{eq:r-rh:rot}
    \triscalar{\pmb{r} - \pmb{r}_h}{\nabla \times \tilde{\pmb{\mu}}_h} = 0 \qquad \forall \tilde{\pmb{\mu}}_h \in \Nedelec.
  \end{equation}
Using (\ref{eq:r-rh:rot}) and $\pmb{\Pi}^{\div}_{p_v} \pmb{r} - \pmb{r}_h \in \nabla \times \Nedelec$ yields 
  \begin{equation*}
    \trinorm{ \pmb{r} - \pmb{r}_h }^2
    = \triscalar{\pmb{r} - \pmb{r}_h}{\pmb{r} - \pmb{\Pi}^{\div}_{p_v} \pmb{r}} + \triscalar{\pmb{r} - \pmb{r}_h}{\pmb{\Pi}^{\div}_{p_v} \pmb{r} - \pmb{r}_h }
    = \triscalar{\pmb{r} - \pmb{r}_h}{\pmb{r} - \pmb{\Pi}^{\div}_{p_v} \pmb{r}}.
  \end{equation*}
Hence, by Cauchy-Schwarz and the definition of $\trinorm{\cdot}$ we find 
  \begin{equation*}
    \trinorm{ \pmb{r} - \pmb{r}_h }
    \leq \trinorm{ \pmb{r} - \pmb{\Pi}^{\div}_{p_v} \pmb{r} }
    \lesssim  \| \pmb{r} - \pmb{\Pi}^{\div}_{p_v} \pmb{r} \|_{L^2(\Omega)} + \| (\pmb{r} - \pmb{\Pi}^{\div}_{p_v} \pmb{r} ) \cdot \pmb{n} \|_{L^2(\Gamma)}.
  \end{equation*}
  In order to treat the volume term we invoke \cite[Thm.~{2.10} (vi)/Thm.~{2.13} (iv)]{melenk-rojik18}, which is applicable since $\nabla \cdot \pmb{r} = \nabla \cdot \pmb{r}_h$ is discrete.
  We therefore have using the estimate of Lemma~\ref{lemma:helmholtz_decomp_l2_normal_trace}
  \begin{equation*}
    \| \pmb{r} - \pmb{\Pi}^{\div}_{p_v} \pmb{r} \|_{L^2(\Omega)}
    \lesssim \frac{h}{p_v} \norm{\pmb{r}}_{H^1(\Omega)}  \lesssim \frac{h}{p_v} \| \nabla \cdot (\tilde{\pmb{\varphi}}_h - \IhGamma \pmb{\varphi}) \|_{L^2(\Omega)}.
  \end{equation*}
  To estimate the boundary term we apply Proposition~\ref{proposition:melenk_rojik_operator} to conclude
  \begin{align*}
    \| (\pmb{r} - \pmb{\Pi}^{\div}_{p_v} \pmb{r} ) \cdot \pmb{n} \|_{L^2(\Gamma)}
     & = \| \pmb{r}\cdot \pmb{n} - \pmb{\Pi}^{L^2(\Gamma)}_{p_v} \pmb{r} \cdot \pmb{n} \|_{L^2(\Gamma)}
    \lesssim \frac{h}{p_v} \norm{\pmb{r}\cdot \pmb{n}}_{H^1(\Gamma)}
    \stackrel{(\ref{eq:properties-of-R-30})}{\lesssim} \frac{h}{p_v} \norm{R}_{H^1(\Gamma)}                                                       \\
     & \lesssim \frac{h}{p_v} \norm{R}_{H^2(\Omega)}
    \stackrel{(\ref{eq:properties-of-R-10})}{\lesssim} \frac{h}{p_v}  \| \nabla \cdot (\tilde{\pmb{\varphi}}_h - \IhGamma \pmb{\varphi}) \|_{L^2(\Omega)}.
  \end{align*}
  Summarizing the above development, we have
  \begin{equation}\label{eq:estimate_for_r_minus_r_h}
    \trinorm{\pmb{r} - \pmb{r}_h}
    \lesssim \frac{h}{p_v} \| \nabla \cdot (\tilde{\pmb{\varphi}}_h - \IhGamma \pmb{\varphi}) \|_{L^2(\Omega)}. 
  \end{equation}
  Adding and subtracting $\pmb{\varphi}$ and using (\ref{eq:lemma:properties_of_IhGamma-20}) we get 
  \begin{equation*}
    T_1
    \leq \trinorm{\pmb{e}} \cdot \trinorm{\pmb{r} - \pmb{r}_h}
    \lesssim \frac{h}{p_v} \trinorm{\pmb{e}} \cdot \| \nabla \cdot (\tilde{\pmb{\varphi}}_h - \IhGamma \pmb{\varphi}) \|_{L^2(\Omega)}
    \lesssim \frac{h}{p_v} \trinorm{\pmb{e}} \cdot \| \nabla \cdot (\pmb{\varphi} - \tilde{\pmb{\varphi}}_h) \|_{L^2(\Omega)}.
  \end{equation*}
  \\
  \noindent
  \textbf{Treatment of $T_2$}:
  The term $T_2$ is estimated with a duality argument.
  We seek a $\pmb{\psi} \in \HDiv$ such that
  \begin{equation}
\label{eq:some-psi}
    \triscalar{\pmb{v}}{\pmb{r}} = (\nabla \cdot \pmb{v}, \nabla \cdot \pmb{\psi})_{\Omega} \qquad \forall \pmb{v} \in \pmb{V}.
  \end{equation}
  As $\pmb{r} = \nabla R$ for some $R \in H^2(\Omega)$, see (\ref{eq:properties-of-R-30}), 
  \begin{align*}
    (\nabla \cdot \pmb{v}, \nabla \cdot \pmb{\psi})_{\Omega}
     & = \triscalar{\pmb{v}}{\pmb{r}}
    = \triscalar{\pmb{v}}{\nabla R}
    = (\pmb{v}, \nabla R)_{\Omega} + \langle \pmb{v} \cdot \pmb{n}, \partial_n R \rangle_{\Omega}                                                        \\
     & = - (\nabla \cdot \pmb{v}, R)_{\Omega} + \langle \pmb{v} \cdot \pmb{n}, \partial_n R + R \rangle_{\Gamma} \stackrel{(\ref{eq:properties-of-R-20})}{=}  - (\nabla \cdot \pmb{v}, R)_{\Omega}. 
  \end{align*}
  Upon solving the problem
  \begin{equation*}
    \begin{alignedat}{2}
      - \Delta w &= R  \quad   &&\text{in } \Omega,\\
      w &= 0          &&\text{on } \Gamma, \\
    \end{alignedat}
  \end{equation*}
  and setting $\pmb{\psi} = \nabla w$ we have found a $\pmb{\psi}$ satisfying (\ref{eq:some-psi}). 
  Furthermore, the following estimate holds
  \begin{equation}
\label{eq:div-psi-H1}
    \norm{\nabla \cdot \pmb{\psi}}_{H^1(\Omega)}
    = \norm{R}_{H^1(\Omega)}
    \stackrel{(\ref{eq:properties-of-R-10})}{\lesssim} \| \nabla \cdot (\tilde{\pmb{\varphi}}_h - \IhGamma \pmb{\varphi}) \|_{(\HOne)^\prime}.
  \end{equation}
  Hence, we have for any $\tilde {\pmb{\psi}}_h \in \RTBDM$
  \begin{equation*}
    T_2
    = \triscalar{\pmb{e}}{\pmb{r}}
    = (\nabla \cdot \pmb{e}, \nabla \cdot \pmb{\psi})_{\Omega}
    = (\nabla \cdot \pmb{e}, \nabla \cdot (\pmb{\psi} - \tilde {\pmb{\psi}}_h))_{\Omega}
    \leq \norm{\nabla \cdot \pmb{e}}_{L^2(\Omega)} \norm{\nabla \cdot (\pmb{\psi} - \tilde{\pmb{\psi}}_h)}_{L^2(\Omega)},
  \end{equation*}
  where we used the definition of $T_2$, employed the duality argument elaborated above, and exploited the orthogonality relation of $\IhGamma$ to insert any $\tilde{\pmb{\psi}}_h \in \RTBDM$. 
  Finally making use of the \textsl{a priori} estimate (\ref{eq:div-psi-H1}) of $\pmb{\psi}$ we find
  \begin{equation*}
    T_2
    \leq
    \norm{\nabla \cdot \pmb{e}}_{L^2(\Omega)} \cdot \inf_{\tilde{\pmb{\psi}}_h \in \RTBDM} \norm{\nabla \cdot (\pmb{\psi} - \tilde{\pmb{\psi}}_h)}_{L^2(\Omega)}
    \stackrel{(\ref{eq:div-psi-H1})}{\lesssim} \frac{h}{p_v} \norm{\nabla \cdot \pmb{e}}_{L^2(\Omega)} \| \nabla \cdot (\tilde{\pmb{\varphi}}_h - \IhGamma \pmb{\varphi}) \|_{(\HOne)^\prime}.
  \end{equation*}
  Using integration by parts, we get 
  \begin{align*}
    \| \nabla \cdot (\tilde{\pmb{\varphi}}_h - \IhGamma \pmb{\varphi}) \|_{(\HOne)^\prime}
     & = \sup_{f \in \HOne} \frac{(\nabla \cdot (\tilde{\pmb{\varphi}}_h - \IhGamma \pmb{\varphi}), f)_{\Omega}}{\norm{f}_{\HOne}}                                                                                            \\
     & = \sup_{f \in \HOne} \frac{-(\tilde{\pmb{\varphi}}_h - \IhGamma \pmb{\varphi} , \nabla f)_{\Omega} + \langle (\tilde{\pmb{\varphi}}_h - \IhGamma \pmb{\varphi}) \cdot \pmb{n} , f \rangle_{\Gamma} }{\norm{f}_{\HOne}} \\
     & \lesssim \trinorm{\tilde{\pmb{\varphi}}_h - \IhGamma \pmb{\varphi}}. 
  \end{align*}
  Finally we have for any $\tilde{\pmb{\varphi}}_h$ 
  \begin{align*}
    \trinorm{\pmb{e}}^2
     & = \triscalar{\pmb{e}}{\pmb{\varphi} - \tilde{\pmb{\varphi}}_h} + \triscalar{\pmb{e}}{\tilde{\pmb{\varphi}}_h - \IhGamma \pmb{\varphi}} \\
     & = \triscalar{\pmb{e}}{\pmb{\varphi} - \tilde{\pmb{\varphi}}_h}  + T_1 + T_2                                                            \\
     & \lesssim \trinorm{\pmb{e}} \cdot \trinorm{\pmb{\varphi} - \tilde{\pmb{\varphi}}_h}_{L^2(\Omega)}
    + \frac{h}{p_v} \trinorm{\pmb{e}} \cdot \| \nabla \cdot (\pmb{\varphi} - \tilde{\pmb{\varphi}}_h) \|_{L^2(\Omega)}
    + \frac{h}{p_v} \norm{\nabla \cdot \pmb{e}}_{L^2(\Omega)}  \cdot \trinorm{\tilde{\pmb{\varphi}}_h - \IhGamma \pmb{\varphi}}.
  \end{align*}
  Adding and subtracting $\pmb{\varphi}$ in the last term and applying estimate (\ref{eq:lemma:properties_of_IhGamma-20}) 
together with the Young inequality yields the result.
\end{proof}

\subsection{Error estimates}\label{subsection:main_error_estimates}

\begin{lemma}[Suboptimal estimate for $\norm{e^u}_{L^2(\Omega)}$ -
  Robin version of {\cite[Lem.~{4.1}]{bernkopf-melenk22}}]\label{lemma:e_u_suboptimal_l2_error_estimate_robin}
  Let Assumption~\ref{assumption:smax_shift} be valid for some $\smax \geq 0$.  Let $( \pmb{\varphi}_h , u_h )$ be the FOSLS approximation of $( \pmb{\varphi} , u )$. 
Set  $e^u = u-u_h$ and $\pmb{e}^{\pmb{\varphi}} = \pmb{\varphi}-\pmb{\varphi}_h$. Then,
for any $\tilde{u}_h \in \Sp$, $\tilde{\pmb{\varphi}}_h \in \RTBDM$, there holds 
\begin{align*}
  \norm{e^u}_{L^2(\Omega)}
   & \lesssim  \frac{h}{p}  \| ( \pmb{e}^{\pmb{\varphi}} , e^u ) \|_b \\
   & \lesssim \frac{h}{p}  \norm{u - \tilde{u}_h}_{H^1(\Omega)}
  + \frac{h}{p}  \norm{\pmb{\varphi} - \tilde{\pmb{\varphi}}_h}_{L^2(\Omega)}
  + \frac{h}{p}  \| (\pmb{\psi} - \tilde{\pmb{\psi}}_h) \cdot \pmb{n} \|_{L^2(\Gamma)}
  + \frac{h}{p}  \norm{\nabla \cdot (\pmb{\varphi} - \tilde{\pmb{\varphi}}_h)}_{L^2(\Omega)}. 
\end{align*}
\end{lemma}

\begin{proof}
We employ the duality argument of Theorem~\ref{theorem:duality_argument_robin} with $w = e^u$.
As in \cite[Lem.~{4.1}]{bernkopf-melenk22} we find
by Galerkin orthogonality and the Cauchy-Schwarz inequality
for any $\tilde{\pmb{\psi}}_h \in \RTBDM$ and $\tilde{v}_h \in \Sp$ 
\begin{equation*}
  \norm{e^u}_{L^2(\Omega)}^2 \leq \| ( \pmb{e}^{\pmb{\varphi}} , e^u ) \|_b \| ( \pmb{\psi} - \tilde{\pmb{\psi}}_h , v - \tilde{v}_h ) \|_b. 
\end{equation*}
The norm equivalence in Theorem~\ref{theorem:norm_equivalence_robin} gives
\begin{equation*}
  \| ( \pmb{\psi} - \tilde{\pmb{\psi}}_h , v - \tilde{v}_h ) \|_b
  \lesssim \| v - \tilde{v}_h \|_{\HOne}
  + \| \pmb{\psi} - \tilde{\pmb{\psi}}_h \|_{\HDiv}
  + \| (\pmb{\psi} - \tilde{\pmb{\psi}}_h) \cdot \pmb{n} \|_{L^2(\Gamma)}.
\end{equation*}
Using Proposition~\ref{proposition:melenk_rojik_operator} and exploiting the regularity estimates given by Theorem~\ref{theorem:duality_argument_robin} yields the result.
\end{proof}

\begin{theorem}[Suboptimal estimate for $\norm{\pmb{e}^{\pmb{\varphi}}}_{L^2(\Omega)}$ ---
    Robin version of {\cite[Thm.~{4.8}]{bernkopf-melenk22}}]\label{theorem:e_phi_suboptimal_l2_error_estimate_robin}
    Let Assumption~\ref{assumption:smax_shift} be valid for some $\smax \geq 0$. Let $( \pmb{\varphi}_h , u_h )$ be the FOSLS approximation of $( \pmb{\varphi} , u )$. 
Set  $e^u = u-u_h$ and $\pmb{e}^{\pmb{\varphi}} = \pmb{\varphi}-\pmb{\varphi}_h$. Then,
  for any $\tilde{u}_h \in \Sp$, $\tilde{\pmb{\varphi}}_h \in \RTBDM$,
  \begin{equation*}
    \norm{\pmb{e}^{\pmb{\varphi}}}_{L^2(\Omega)}
    \lesssim \left(\frac{h}{p}\right)^{1/2} \norm{u - \tilde{u}_h}_{H^1(\Omega)}
    + \norm{\pmb{\varphi} - \tilde{\pmb{\varphi}}_h}_{L^2(\Omega)}
    + \| (\pmb{\varphi} - \tilde{\pmb{\varphi}}_h) \cdot \pmb{n} \|_{L^2(\Gamma)}
    + \left(\frac{h}{p}\right)^{1/2}  \norm{\nabla \cdot (\pmb{\varphi} - \tilde{\pmb{\varphi}}_h)}_{L^2(\Omega)}. 
  \end{equation*}
\end{theorem}

\begin{proof}
  Let $(\pmb{\psi}, v) \in \productspacerobin$ denote the dual solution given by Theorem~\ref{theorem:duality_argument_phi_robin} applied to $\pmb{\eta} = \pmb{e}^{\pmb{\varphi}}$.
  Theorem~\ref{theorem:duality_argument_phi_robin} gives $\pmb{\psi} \in \pmb{L}^2(\Omega)$, $\nabla \cdot \pmb{\psi} \in H^1(\Omega)$, $\pmb{\psi} \cdot \pmb{n} \in H^{1/2}(\Gamma)$, 
and $v \in H^2(\Omega)$, which are controlled by $\|\pmb{e}^{\pmb{\varphi}}\|_{L^2(\Omega)}$. 
  By the Galerkin orthogonality we have for any $( \tilde{\pmb{\psi}}_h , \tilde{v}_h  )$
  \begin{equation*}
    \norm{\pmb{e}^{\pmb{\varphi}}}_{L^2(\Omega)}^2
    = b( ( \pmb{e}^{\pmb{\varphi}} , e^u ) , ( \pmb{\psi} , v ) )
    = b( ( \pmb{e}^{\pmb{\varphi}} , e^u ) , ( \pmb{\psi} - \tilde{\pmb{\psi}}_h , v - \tilde{v}_h ) ).
  \end{equation*}
  Let us first estimate all terms in the above except for $\triscalar{\pmb{e}^{\pmb{\varphi}}}{\pmb{\psi} - \tilde{\pmb{\psi}}_h} = (\pmb{e}^{\pmb{\varphi}} , \pmb{\psi} - \tilde{\pmb{\psi}}_h )_{L^2(\Omega)} + \langle \pmb{e}^{\pmb{\varphi}} \cdot \pmb{n} , (\pmb{\psi} - \tilde{\pmb{\psi}}_h) \cdot \pmb{n} \rangle_{L^2(\Gamma)}$: 
  \begin{equation}\label{eq:estimates_for_optimal_e_phi_error_estimate_robin}
    \begin{alignedat}{2}
      ( \nabla e^u + \pmb{e}^{\pmb{\varphi}} , \nabla (v - \tilde{v}_h) )_{\Omega}
      &\leq \| ( \pmb{e}^{\pmb{\varphi}} , e^u ) \|_b  \| \nabla (v - \tilde{v}_h) \|_{L^2(\Omega)}, \\
      \langle -\alpha e^u , (\pmb{\psi} - \tilde{\pmb{\psi}}_h) \cdot \pmb{n} \rangle_{\Gamma}
      &\leq \norm{e^u}_{L^2(\Gamma)}  \| (\pmb{\psi} - \tilde{\pmb{\psi}}_h) \cdot \pmb{n} \|_{L^2(\Gamma)} \\
      &\leq (h/p)^{1/2} \| ( \pmb{e}^{\pmb{\varphi}} , e^u ) \|_b \| (\pmb{\psi} - \tilde{\pmb{\psi}}_h) \cdot \pmb{n} \|_{L^2(\Gamma)}, \\
      ( \nabla \cdot \pmb{e}^{\pmb{\varphi}} + \gamma e^u , \nabla \cdot (\pmb{\psi} - \tilde{\pmb{\psi}}_h) + \gamma (v - \tilde{v}_h) )_{\Omega}
      &\lesssim \| ( \pmb{e}^{\pmb{\varphi}} , e^u ) \|_b  \left[ \| \nabla \cdot ( \pmb{\psi} - \tilde{\pmb{\psi}}_h ) ) \|_{L^2(\Omega)} + \norm{ v - \tilde{v}_h }_{L^2(\Omega)} \right], \\
      ( \nabla e^u , \pmb{\psi} - \tilde{\pmb{\psi}}_h )_{\Omega}
      &= - ( e^u , \nabla \cdot ( \pmb{\psi} - \tilde{\pmb{\psi}}_h ) )_{\Omega} + \langle e^u , (\pmb{\psi} - \tilde{\pmb{\psi}}_h) \cdot \pmb{n} \rangle_{\Gamma} \\
      &\lesssim \| ( \pmb{e}^{\pmb{\varphi}} , e^u ) \|_b  \left[ \| \nabla \cdot ( \pmb{\psi} - \tilde{\pmb{\psi}}_h ) ) \|_{L^2(\Omega)} + (h/p)^{1/2} \| (\pmb{\psi} - \tilde{\pmb{\psi}}_h) \cdot \pmb{n} \|_{L^2(\Gamma)} \right], \\
      \langle \pmb{e}^{\pmb{\varphi}} \cdot \pmb{n} - \alpha e^u , -\alpha (v - \tilde{v}_h) \rangle_{\Gamma}
      &\lesssim \| ( \pmb{e}^{\pmb{\varphi}} , e^u ) \|_b \| v - \tilde{v}_h \|_{H^1(\Omega)}.
    \end{alignedat}
  \end{equation}
  Hence, we arrive at 
  \begin{equation}\label{eq:e_phi_intermediate_estimate_1_robin}
    \norm{\pmb{e}^{\pmb{\varphi}}}_{L^2(\Omega)}^2
    \lesssim
    \| ( \pmb{e}^{\pmb{\varphi}} , e^u ) \|_b
    \left[
    \| \nabla \cdot ( \pmb{\psi} - \tilde{\pmb{\psi}}_h ) ) \|_{L^2(\Omega)} +
    (h/p)^{1/2} \| (\pmb{\psi} - \tilde{\pmb{\psi}}_h) \cdot \pmb{n} \|_{L^2(\Gamma)} + \norm{ v - \tilde{v}_h }_{H^1(\Omega)} \right] +
    \triscalar{\pmb{e}^{\pmb{\varphi}}}{\pmb{\psi} - \tilde{\pmb{\psi}}_h}.
  \end{equation}
  To analyze the term $\triscalar{\pmb{e}^{\pmb{\varphi}}}{\pmb{\psi} - \tilde{\pmb{\psi}}_h}$ we follow a similar procedure as in \cite[Thm.~{4.8}]{bernkopf-melenk22} and 
  first perform a Helmholtz decomposition of the vector field $\pmb{\psi}$.
  Since $\pmb{\psi} \in \HDiv$ with $\nabla \cdot \pmb{\psi} \in \HOne$ and $\pmb{\psi} \cdot \pmb{n} \in H^{1/2}(\Gamma)$
  we find $\pmb{\rho} \in \HZeroCurl$ and $z \in H^2(\Omega)$ such that $\pmb{\psi} = \nabla \times \pmb{\rho} + \nabla z$ in the following way: 
  let $z \in \HOne$ with zero average solve
  \begin{equation}
    \label{eq:theorem:e_phi_suboptimal_l2_error_estimate_robin-10}
    \begin{alignedat}{2}
      - \Delta z     &= - \nabla \cdot \pmb{\psi} \quad  &&\text{in } \Omega, \\
      \partial_n z &= \pmb{\psi} \cdot \pmb{n}         &&\text{on } \Gamma.
    \end{alignedat}
  \end{equation}
  As $\nabla \cdot (\pmb{\psi} - \nabla z) = 0$ and $(\pmb{\psi} - \nabla z) \cdot \pmb{n} = 0$ by construction,
  the exact sequence property ensures the existence of $\pmb{\rho} \in \HZeroCurl$ such that $\pmb{\psi} - \nabla z = \nabla \times \pmb{\rho}$.
  By Assumption~\ref{assumption:smax_shift} and Remark~\ref{remark:regularity_shift_for_neumann_problem}\ref{item:remark:regularity_shift_for_neumann_problem-ii}
we have $z \in H^2(\Omega)$ together with the estimate
  \begin{equation*}
    \norm{z}_{H^2(\Omega)} \lesssim \norm{\nabla \cdot \pmb{\psi}}_{L^2(\Omega)} + \norm{\pmb{\psi} \cdot \pmb{n}}_{H^{1/2}(\Gamma)}.
  \end{equation*}
The weak formulation of 
(\ref{eq:theorem:e_phi_suboptimal_l2_error_estimate_robin-10}) is given by
  \begin{equation*}
\forall w \in H^1(\Omega) \quad \colon \quad 
    (\nabla z, \nabla w)_{\Omega} = (- \nabla \cdot \pmb{\psi}, w)_{\Omega} + \langle \pmb{\psi} \cdot \pmb{n}, w\rangle_{\Gamma} = ( \pmb{\psi}, \nabla w)_{\Omega},
  \end{equation*}
  due to partial integration.
  Hence, Lax-Milgram provides $\norm{z}_{H^1(\Omega)} \lesssim \norm{\pmb{\psi}}_{L^2(\Omega)}$.
  Finally we have $\norm{\nabla \times \pmb{\rho}}_{L^2(\Omega)} \leq \norm{\pmb{\psi}}_{L^2(\Omega)} + \norm{\nabla z}_{L^2(\Omega)} \lesssim \norm{\pmb{\psi}}_{L^2(\Omega)}$.
  We now continue estimating (\ref{eq:e_phi_intermediate_estimate_1_robin}) by applying the Helmholtz decomposition.
  In essence this is again the procedure of \cite[Thm.~{4.8}]{bernkopf-melenk22} after replacing $\norm{\cdot}_{L^2(\Omega)}$ with $\trinorm{\cdot}$.
  For the reader's convenience we recall the important steps.
  For any $\tilde{\pmb{\psi}}_h^c, \tilde{\pmb{\psi}}_h^g \in \RTBDM$ we have with $\tilde{\pmb{\psi}}_h = \tilde{\pmb{\psi}}_h^c + \tilde{\pmb{\psi}}_h^g$
  \begin{equation*}
    \triscalar{\pmb{e}^{\pmb{\varphi}}}{\pmb{\psi} - \tilde{\pmb{\psi}}_h}
    = \triscalar{\pmb{e}^{\pmb{\varphi}}}{\nabla \times \pmb{\rho} - \tilde{\pmb{\psi}}_h^c} + \triscalar{\pmb{e}^{\pmb{\varphi}}}{\nabla z - \tilde{\pmb{\psi}}_h^g}
    \eqqcolon T^c + T^g. 
  \end{equation*}
  \noindent
  \textbf{Treatment of $T^g$}:
  The Cauchy-Schwarz inequality gives 
  \begin{equation*}
    T^g = \triscalar{\pmb{e}^{\pmb{\varphi}}}{\nabla z - \tilde{\pmb{\psi}}_h^g}
    \leq \trinorm{\pmb{e}^{\pmb{\varphi}}} \cdot \trinorm{\nabla z - \tilde{\pmb{\psi}}_h^g}.
  \end{equation*}
  \textbf{Treatment of $T^c$}:
  For any $\tilde{\pmb{\varphi}}_h \in \RTBDM$ we have
  \begin{align*}
    T^c
     & =  \triscalar{\pmb{e}^{\pmb{\varphi}}}{\nabla \times \pmb{\rho} - \tilde{\pmb{\psi}}_h^c}                \\
     & = \triscalar{\pmb{\varphi} - \tilde{\pmb{\varphi}}_h}{\nabla \times \pmb{\rho} - \tilde{\pmb{\psi}}_h^c}
    + \triscalar{\tilde{\pmb{\varphi}}_h - \pmb{\varphi}_h }{\nabla \times \pmb{\rho} - \tilde{\pmb{\psi}}_h^c} \eqqcolon T^c_1 + T^c_2. 
  \end{align*}
  \textbf{Treatment of $T^c_1$}:
  By Cauchy-Schwarz inequality we have
  \begin{equation*}
    T^c_1 = \triscalar{\pmb{\varphi} - \tilde{\pmb{\varphi}}_h}{\nabla \times \pmb{\rho} - \tilde{\pmb{\psi}}_h^c}
    \leq \trinorm{ \pmb{\varphi} - \tilde{\pmb{\varphi}}_h } \cdot \trinorm{ \nabla \times \pmb{\rho} - \tilde{\pmb{\psi}}_h^c }.
  \end{equation*}
  \textbf{Treatment of $T^c_2$}:
  To treat $T^c_2$ we 
decompose the finite element function $\tilde{\pmb{\varphi}}_h - \pmb{\varphi}_h \in \RTBDM$ 
on the discrete and the continuous level with the aid of Lemma~\ref{lemma:helmholtz_decomp_l2_normal_trace} as 
  \begin{align*}
    \tilde{\pmb{\varphi}}_h - \pmb{\varphi}_h & = \nabla \times \pmb{\mu} + \pmb{r},    \\
    \tilde{\pmb{\varphi}}_h - \pmb{\varphi}_h & = \nabla \times \pmb{\mu}_h + \pmb{r}_h, 
  \end{align*}
  where $\pmb{\mu} \in \pmb{Y}$, $\pmb{r} \in \pmb{V}$, $\pmb{\mu}_h \in \Nedelec$ and $\pmb{r}_h \in \RTBDM$.
  We next select $\tilde{\pmb{\psi}}_h^c = \pmb{\Pi}^{\mathrm{curl},\Gamma}_h \nabla \times \pmb{\rho} $ given by Lemma~\ref{lemma:helmholtz_decomp_l2_normal_trace}.
  In view of the definition of $\pmb{\Pi}^{\mathrm{curl}, \Gamma}_h$ we find
  \begin{align*}
    T^c_2
     & = \triscalar{\tilde{\pmb{\varphi}}_h - \pmb{\varphi}_h }{\nabla \times \pmb{\rho} - \tilde{\pmb{\psi}}_h^c}                                                                                                                             \\
     & = \underbrace{\triscalar{\nabla \times \pmb{\mu}_h}{\nabla \times \pmb{\rho} - \pmb{\Pi}^{\mathrm{curl},\Gamma}_h \nabla \times \pmb{\rho}}}_{=0}
    + \triscalar{\pmb{r}_h}{\nabla \times \pmb{\rho} - \pmb{\Pi}^{\mathrm{curl},\Gamma}_h \nabla \times \pmb{\rho}}                                                                                                                            \\
     & = \triscalar{\pmb{r}_h - \pmb{r}}{\nabla \times \pmb{\rho} - \pmb{\Pi}^{\mathrm{curl},\Gamma}_h \nabla \times \pmb{\rho}} + \triscalar{\pmb{r}}{\nabla \times \pmb{\rho} - \pmb{\Pi}^{\mathrm{curl},\Gamma}_h \nabla \times \pmb{\rho}} \\
     & \eqqcolon T_1 + T_2. 
  \end{align*}
  \textbf{Treatment of $T_1$}:
  As in the estimate (\ref{eq:estimate_for_r_minus_r_h}) we have
  \begin{equation*}
    \trinorm{\pmb{r} - \pmb{r}_h}
    \lesssim \frac{h}{p_v} \| \nabla \cdot (\tilde{\pmb{\varphi}}_h - \pmb{\varphi}_h) \|_{L^2(\Omega)},
  \end{equation*}
  which gives with Cauchy-Schwarz 
  \begin{equation*}
    T_1
    \lesssim \frac{h}{p_v} \norm{\nabla \cdot (\tilde{\pmb{\varphi}}_h - \pmb{\varphi}_h) }_{L^2(\Omega)} \trinorm{ \nabla \times \pmb{\rho} - \pmb{\Pi}^{\mathrm{curl},\Gamma}_h \nabla \times \pmb{\rho} }
    \lesssim \frac{h}{p_v} \norm{\nabla \cdot (\tilde{\pmb{\varphi}}_h - \pmb{\varphi}_h) }_{L^2(\Omega)} \trinorm{ \nabla \times \pmb{\rho} };
  \end{equation*}
  here, the last estimate follows from the observation that $\pmb{\Pi}^{\mathrm{curl},\Gamma}_h$ is an orthogonal projection with respect to $\trinorm{\cdot}$.
  Inserting $\pmb{\varphi}$ and estimating $\norm{\nabla \cdot (\pmb{\varphi} - \pmb{\varphi}_h) }_{L^2(\Omega)}$ 
by $\| (\pmb{e}^{\pmb{\varphi}}, e^u) \|_b$ and using that $\pmb{\rho} \in \pmb{H}_0(\curl,\Omega)$ so that $\trinorm{\nabla \times \pmb{\rho}} = \|\nabla \times \pmb{\rho}\|_{L^2(\Omega)}$ we find
  \begin{equation*}
    T_1
    \lesssim \frac{h}{p_v} \norm{\nabla \cdot (\pmb{\varphi} - \tilde{\pmb{\varphi}}_h) }_{L^2(\Omega)} \| \nabla \times \pmb{\rho} \|_{L^2(\Omega)} \lesssim \frac{h}{p_v} \| ( \pmb{e}^{\pmb{\varphi}} , e^u ) \|_b \| \nabla \times \pmb{\rho} \|_{L^2(\Omega)}.
  \end{equation*}
  \\
  \noindent
  \textbf{Treatment of $T_2$}: Recall that $\pmb{\rho} \in \HZeroCurl$ and that $\pmb{\Pi}^{\mathrm{curl}, \Gamma}_h$ maps into $\nabla \times \Nedelec$.
  Therefore we can write $\nabla \times \pmb{\rho} - \pmb{\Pi}^{\mathrm{curl},\Gamma}_h \nabla \times \pmb{\rho} = \nabla \times \widehat{\pmb{\rho}}$ for some $\widehat{\pmb{\rho}} \in \HCurl$.
  In fact $\widehat{\pmb{\rho}} \in \pmb{Y}$ since $(\nabla \times \widehat{\pmb{\rho}}) \cdot \pmb{n} = (\nabla \times \pmb{\rho} - \pmb{\Pi}^{\mathrm{curl},\Gamma}_h \nabla \times \pmb{\rho}) \cdot \pmb{n} = (\pmb{\Pi}^{\mathrm{curl},\Gamma}_h \nabla \times \pmb{\rho}) \cdot \pmb{n} \in L^2(\Gamma)$.
  Consequently, the definition of the remainder $\pmb{r}$ gives $T_2 = \triscalar{\pmb{r}}{\nabla \times \widehat{\pmb{\rho}}} = 0$, see Lemma~\ref{lemma:helmholtz_decomp_l2_normal_trace}.
  \\
  \noindent
  \textbf{Collecting all the terms}:
  Since $\pmb{\rho} \in \HZeroCurl$ and consequently $\nabla \times \pmb{\rho} \in \HZeroDiv$
  we can estimate 
\begin{equation}
\label{eq:bound-rho}
\trinorm{\nabla \times \pmb{\rho}} = \norm{\nabla \times \pmb{\rho}}_{L^2(\Omega)} \lesssim \norm{\pmb{\psi}}_{L^2(\Omega)} \lesssim \norm{\pmb{e}^{\pmb{\varphi}}}_{L^2(\Omega)},  
\end{equation}
  where we used the estimates of the Helmholtz decomposition as well as the regularity estimates of Lemma~\ref{theorem:duality_argument_phi_robin}.
  We can now summarize
  \begin{equation}\label{eq:e_phi_intermediate_estimate_2_robin}
    \triscalar{\pmb{e}^{\pmb{\varphi}}}{\pmb{\psi} - \tilde{\pmb{\psi}}_h}
    \lesssim
    \trinorm{\nabla z - \tilde{\pmb{\psi}}_h^g} \cdot \trinorm{\pmb{e}^{\pmb{\varphi}}} +
    \left[
    \trinorm{\pmb{\varphi} - \tilde{\pmb{\varphi}}_h}
    + \frac{h}{p_v} \norm{\nabla \cdot (\pmb{\varphi} - \tilde{\pmb{\varphi}}_h) }_{L^2(\Omega)}
    + \frac{h}{p_v} \| ( \pmb{e}^{\pmb{\varphi}} , e^u ) \|_b \right] \norm{\pmb{e}^{\pmb{\varphi}}}_{L^2(\Omega)}.
  \end{equation}
  To conclude the proof we estimate the quantities arising in the estimates (\ref{eq:e_phi_intermediate_estimate_1_robin}) and (\ref{eq:e_phi_intermediate_estimate_2_robin}).
  To that end note that $\nabla z \in \pmb{H}^1(\div,\Omega)$.
  Using the estimates of the Helmholtz decomposition, the equation satisfied by $z$, and the regularity estimates given by Theorem~\ref{theorem:duality_argument_phi_robin}, we find
  \begin{align*}
    \norm{\nabla z}_{\pmb{H}^1(\div,\Omega)}
     & \lesssim \norm{z}_{H^2(\Omega)} + \| \underbrace{\Delta z}_{= \nabla \cdot \pmb{\psi}} \|_{H^1(\Omega)}
    \lesssim \norm{ \pmb{e}^{\pmb{\varphi}} }_{L^2(\Omega)},                                                   \\
    \norm{\nabla z \cdot \pmb{n}}_{H^{1/2}(\Gamma)}
     & = \norm{\pmb{\psi} \cdot \pmb{n}}_{H^{1/2}(\Gamma)}
    \lesssim \norm{ \pmb{e}^{\pmb{\varphi}} }_{L^2(\Omega)}.
  \end{align*}
  Exploiting these regularity estimates and employing the operator of Proposition~\ref{proposition:melenk_rojik_operator} we may find $\tilde{\pmb{\psi}}_h^g \in \RTBDM$ with 
  \begin{align*}
    \| \nabla z - \tilde{\pmb{\psi}}_h^g  \|_{\HDiv}
     & \lesssim
    h/p_v \norm{\nabla z}_{\pmb{H}^1(\div,\Omega)}
    \lesssim h/p_v \norm{\pmb{e}^{\pmb{\varphi}}}_{L^2(\Omega)} ,          \\
    \| (\nabla z - \tilde{\pmb{\psi}}_h^g) \cdot \pmb{n}  \|_{L^2(\Gamma)}
     & \lesssim
    ( h/p_v )^{1/2} \norm{\nabla z \cdot \pmb{n}}_{H^{1/2}(\Gamma)}
    \stackrel{\text{Prop.~\ref{proposition:melenk_rojik_operator}\ref{melenk_rojik_operator_prop_2}}}{\lesssim} ( h/p_v )^{1/2} \norm{\pmb{e}^{\pmb{\varphi}}}_{L^2(\Omega)}, \\
    \trinorm{ \nabla z - \tilde{\pmb{\psi}}_h^g}
     & \lesssim
    ( h/p_v )^{1/2} \norm{\pmb{e}^{\pmb{\varphi}}}_{L^2(\Omega)},
  \end{align*}
  where the last one is just a combination of the previous ones.
  These estimates in turn give (note that $\nabla \cdot \tilde{ \pmb{\psi}}_h^c = \nabla \cdot \pmb{\Pi}^{\mathrm{curl},\Gamma}_h \nabla \times \pmb{\rho} = 0$)
  \begin{align*}
    \| \nabla \cdot ( \pmb{\psi} - \tilde{\pmb{\psi}}_h ) ) \|_{L^2(\Omega)}
     & = \| \nabla \cdot (\nabla z - \tilde{\pmb{\psi}}_h^g ) \|_{L^2(\Omega)}
    \lesssim h/p_v \norm{\pmb{e}^{\pmb{\varphi}}}_{L^2(\Omega)},                                                                                                                                                 \\ 
    \| (\pmb{\psi} - \tilde{\pmb{\psi}}_h) \cdot \pmb{n} \|_{L^2(\Gamma)}
     & \leq  \| (\nabla \times \pmb{\rho} - \pmb{\Pi}^{\mathrm{curl},\Gamma}_h \nabla \times \pmb{\rho}) \cdot \pmb{n} \|_{L^2(\Gamma)} + \| (\nabla z - \tilde{\pmb{\psi}}_h^g) \cdot \pmb{n} \|_{L^2(\Gamma)} 
    \stackrel{(\ref{eq:bound-rho})}{\lesssim} \norm{\pmb{e}^{\pmb{\varphi}}}_{L^2(\Omega)}.                                                                                                                                                       \\ 
  \end{align*}
  Furthermore there exists $\tilde{v}_h \in \Sp$ with $\norm{v - \tilde{v}_h}_{H^1(\Omega)} \lesssim h/p_s \norm{v}_{H^2(\Omega)} \lesssim h/p_s \norm{\pmb{e}^{\pmb{\varphi}}}_{L^2(\Omega)} $.
  We then combine the estimates (\ref{eq:e_phi_intermediate_estimate_1_robin}) and (\ref{eq:e_phi_intermediate_estimate_2_robin}) to find
  \begin{align*}
    \norm{\pmb{e}^{\pmb{\varphi}}}_{L^2(\Omega)}^2
     & \lesssim
    (h/p)^{1/2} \| ( \pmb{e}^{\pmb{\varphi}} , e^u ) \|_b \norm{\pmb{e}^{\pmb{\varphi}}}_{L^2(\Omega)} +
    (h/p)^{1/2} \cdot \trinorm{\pmb{e}^{\pmb{\varphi}}} \norm{\pmb{e}^{\pmb{\varphi}}}_{L^2(\Omega)} \\
     & \phantom{\lesssim}+
    \left[
    \trinorm{\pmb{\varphi} - \tilde{\pmb{\varphi}}_h}
    + h/p \norm{\nabla \cdot (\pmb{\varphi} - \tilde{\pmb{\varphi}}_h) }_{L^2(\Omega)}
    + h/p \| ( \pmb{e}^{\pmb{\varphi}} , e^u ) \|_b \right] \norm{\pmb{e}^{\pmb{\varphi}}}_{L^2(\Omega)}.
  \end{align*}
  Canceling one power of $\norm{\pmb{e}^{\pmb{\varphi}}}_{L^2(\Omega)}$ on both sides, estimating $\trinorm{\pmb{e}^{\pmb{\varphi}}}$ 
by $\| ( \pmb{e}^{\pmb{\varphi}} , e^u ) \|_b$, and collecting the terms, we find
  \begin{equation*}
    \norm{\pmb{e}^{\pmb{\varphi}}}_{L^2(\Omega)}
    \lesssim
    (h/p)^{1/2} \| ( \pmb{e}^{\pmb{\varphi}} , e^u ) \|_b +
    \trinorm{\pmb{\varphi} - \tilde{\pmb{\varphi}}_h}
    + h/p \norm{\nabla \cdot (\pmb{\varphi} - \tilde{\pmb{\varphi}}_h) }_{L^2(\Omega)}.
  \end{equation*}
  The result follows by from the observation that the FOSLS approximation is the orthogonal projection with respect to the $b$ scalar product,  and the norm equivalence of Theorem~\ref{theorem:norm_equivalence_robin}, and collecting terms.
\end{proof}

\begin{remark}\label{remark:suboptimal_l2_error_estimate_robin}
  Theorem~\ref{theorem:e_phi_suboptimal_l2_error_estimate_robin} seems suboptimal in the following sense:
  Given $f \in L^2(\Omega)$ and $g \in H^{1/2}(\Gamma)$ the shift theorem gives $u \in H^2(\Omega)$ and consequently $\pmb{\varphi} \in \pmb{H}^1(\Omega)$.
  Theorem~\ref{theorem:e_phi_suboptimal_l2_error_estimate_robin} gives
  \begin{equation*}
    \norm{\pmb{e}^{\pmb{\varphi}}}_{L^2(\Omega)}
    \lesssim  h^{3/2} \norm{u}_{H^2(\Omega)}
    + h \norm{\pmb{\varphi}}_{H^1(\Omega)}
    + h^{1/2} \| \pmb{\varphi} \cdot \pmb{n} \|_{H^{1/2}(\Gamma)}
    + h^{1/2}  \norm{\nabla \cdot \pmb{\varphi} }_{L^2(\Omega)}
    \lesssim h^{1/2} (\norm{f}_{L^2(\Omega)} + \norm{g}_{H^{1/2}(\Gamma)}),
  \end{equation*}
  whereas from a best approximation viewpoint we could hope for $\mathcal{O}(h)$.
\eremk
\end{remark}

\begin{lemma}[Convergence of dual solution for $\nabla e^u$ -
    Robin version of {\cite[Lem.~{4.9}]{bernkopf-melenk22}}]\label{lemma:convergence_of_dual_solution_grad_u_robin}
    Let  Assumption~\ref{assumption:smax_shift} be valid for some $\smax \geq 0$. Let $( \pmb{\varphi}_h , u_h )$ be the FOSLS approximation of $( \pmb{\varphi} , u )$.
  Set $e^u = u-u_h$ and $\pmb{e}^{\pmb{\varphi}} = \pmb{\varphi}-\pmb{\varphi}_h$.
  Let $(\pmb{\psi}, v) \in \productspacerobin$ be the dual solution given by Theorem~\ref{theorem:duality_argument_grad_u_robin} with $w = e^u$.
  Let $( \pmb{\psi}_h , v_h )$ be the FOSLS approximation of $( \pmb{\psi} , v )$ and
  abbreviate $e^v = v-v_h$ and $\pmb{e}^{\pmb{\psi}} = \pmb{\psi}-\pmb{\psi}_h$.
  Then,
  \begin{align*}
    \| ( \pmb{e}^{\pmb{\psi}} , e^v ) \|_b   & \lesssim \norm{\nabla e^u}_{L^2(\Omega)},                                \\
    \norm{e^v}_{L^2(\Omega)}                 & \lesssim \frac{h}{p} \norm{\nabla e^u}_{L^2(\Omega)},                    \\
    \norm{e^v}_{L^2(\Gamma)}                 & \lesssim \left(\frac{h}{p}\right)^{1/2} \norm{\nabla e^u}_{L^2(\Omega)}, \\
    \| \pmb{e}^{\pmb{\psi}} \|_{L^2(\Omega)} & \lesssim \left(\frac{h}{p}\right)^{1/2} \norm{\nabla e^u}_{L^2(\Omega)}.
  \end{align*}
\end{lemma}

\begin{proof}
  Theorem~\ref{theorem:duality_argument_grad_u_robin} gives $\pmb{\psi} \in \pmb{H}^1(\Omega)$, $\nabla \cdot \pmb{\psi} \in H^1(\Omega)$, and $v \in H^1(\Omega)$ 
which are controlled by $\|\nabla e^u\|_{L^2(\Omega)}$. Stability gives 
  \begin{equation*}
    \| ( \pmb{e}^{\pmb{\psi}} , e^v ) \|_b \leq \|(\pmb{\psi},v)\|_b
    \lesssim \norm{\nabla e^u}_{L^2(\Omega)}.
  \end{equation*}
  Lemma~\ref{lemma:e_u_suboptimal_l2_error_estimate_robin} gives 
  \begin{equation*}
    \norm{e^v}_{L^2(\Omega)} \lesssim h/p \| ( \pmb{e}^{\pmb{\psi}} , e^v ) \|_b,
  \end{equation*}
  which together with the above stability result proves the second estimate.
  The third one follows by a multiplicative trace inequality together with the second estimate and the norm equivalence theorem in conjunction with the first estimate of the present lemma:
  \begin{equation*}
    \norm{e^v}_{L^2(\Gamma)} \lesssim \norm{e^v}_{L^2(\Omega)}^{1/2} \norm{e^v}_{H^1(\Omega)}^{1/2} \lesssim (h/p)^{1/2} \| ( \pmb{e}^{\pmb{\psi}} , e^v ) \|_b \lesssim (h/p)^{1/2} \norm{\nabla e^u}_{L^2(\Omega)}.
  \end{equation*}
  By Theorem~\ref{theorem:e_phi_suboptimal_l2_error_estimate_robin} we have, for any $\tilde{v}_h \in \Sp$, $\tilde{\pmb{\psi}}_h \in \RTBDM$,
  \begin{equation*}
    \| \pmb{e}^{\pmb{\psi}} \|_{L^2(\Omega)}
    \lesssim \left(\frac{h}{p}\right)^{1/2} \norm{v - \tilde{v}_h}_{H^1(\Omega)}
    + \| \pmb{\psi} - \tilde{\pmb{\psi}}_h \|_{L^2(\Omega)}
    + \| (\pmb{\psi} - \tilde{\pmb{\psi}}_h) \cdot \pmb{n} \|_{L^2(\Gamma)}
    + \left(\frac{h}{p}\right)^{1/2}  \| \nabla \cdot (\pmb{\psi} - \tilde{\pmb{\psi}}_h) \|_{L^2(\Omega)}. 
  \end{equation*}
The regularity of the dual solution together the approximation properties of the pertinent spaces then implies the result. 
\end{proof}

\begin{theorem}[Suboptimal estimate for $\norm{\nabla e^u}_{L^2(\Omega)}$ ---
    Robin version of {\cite[Thm.~4.10]{bernkopf-melenk22}}]\label{theorem:grad_e_u_suboptimal_l2_error_estimate_robin}
    Let Assumption~\ref{assumption:smax_shift} be valid for some $\smax \geq 0$. Let $( \pmb{\varphi}_h , u_h )$ be the FOSLS approximation of $( \pmb{\varphi} , u )$.
  Set $e^u = u-u_h$. Then, for any $\tilde{\pmb{\varphi}}_h \in \RTBDM$, $\tilde{u}_h \in \Sp$ there holds 
  \begin{equation*}
    \norm{\nabla e^u}_{L^2(\Omega)}
    \lesssim
    \norm{ u - \tilde{u}_h }_{H^1(\Omega)} +
    \| \pmb{\varphi} - \tilde{\pmb{\varphi}}_h \|_{L^2(\Omega)} +
    \| (\pmb{\varphi} - \tilde{\pmb{\varphi}}_h) \cdot \pmb{n} \|_{L^2(\Gamma)} +
    \frac{h}{p} \| \nabla \cdot (\pmb{\varphi} - \tilde{\pmb{\varphi}}_h ) \|_{L^2(\Omega)}. 
  \end{equation*}
\end{theorem}

\begin{proof}
  We proceed as in \cite[Thm.~{4.10}]{bernkopf-melenk22} and denote by $( \pmb{e}^{\pmb{\psi}} , e^v )$ the FOSLS approximation of the dual solution given by Theorem~\ref{theorem:duality_argument_grad_u_robin} (duality argument for the gradient of the scalar variable) applied to the right-hand side $w = e^u$. 
  We note that $\|v\|_{H^1(\Omega)}$, $\|\pmb{\psi}\|_{\pmb{H}^1(\Omega)}$, and 
  $\|\nabla \cdot \pmb{\psi}\|_{H^1(\Omega)}$ are controlled by $\|\nabla e^{u}\|_{L^2(\Omega)}$.
  By Galerkin orthogonality, we have for any $\tilde{\pmb{\varphi}}_h \in \RTBDM$, $\tilde{u}_h \in \Sp$
  \begin{equation}
    \label{eq:theorem:grad_e_u_suboptimal_l2_error_estimate_robin-10}
    \norm{e^u}_{L^2(\Omega)}^2
    = b( ( \pmb{\varphi} - \tilde{\pmb{\varphi}}_h , u - \tilde{u}_h ), ( \pmb{e}^{\pmb{\psi}} , e^v ) ).
  \end{equation}
  We specifically choose $\tilde{\pmb{\varphi}}_h = \IhGamma \pmb{\varphi}$.
  In what follows, we repeatedly use properties of the operator $\IhGamma$ collected in Lemma~\ref{lemma:properties_of_IhGamma}.
  Making use of the regularity properties of the dual solution spelled out in Theorem~\ref{theorem:duality_argument_grad_u_robin} and using Lemma~\ref{lemma:convergence_of_dual_solution_grad_u_robin} we get: 
  \begin{equation*}
    \begin{alignedat}{2}
      ( \gamma (u - \tilde{u}_h) , \nabla \cdot \pmb{e}^{\pmb{\psi}} + \gamma e^v )_{\Omega}
      &\lesssim \norm{ u - \tilde{u}_h }_{L^2(\Omega)} \| ( \pmb{e}^{\pmb{\psi}}, e^v ) \|_b \\
      &\lesssim \norm{ u - \tilde{u}_h }_{H^1(\Omega)} \norm{\nabla e^u}_{L^2(\Omega)} , \\
      ( \nabla (u - \tilde{u}_h) , \nabla e^v + \pmb{e}^{\pmb{\psi}} )_{\Omega}
      &\lesssim \norm{\nabla (u - \tilde{u}_h) }_{L^2(\Omega)} \| ( \pmb{e}^{\pmb{\psi}}, e^v ) \|_b \\
      &\lesssim \norm{ u - \tilde{u}_h }_{H^1(\Omega)} \norm{\nabla e^u}_{L^2(\Omega)}, \\
      \langle - \alpha (u - \tilde{u}_h) , \pmb{e}^{\pmb{\psi}} \cdot \pmb{n} - \alpha e^v \rangle_{\Gamma}
      &\lesssim \norm{ u - \tilde{u}_h }_{L^2(\Gamma)} \| ( \pmb{e}^{\pmb{\psi}}, e^v ) \|_b \\
      &\lesssim \norm{ u - \tilde{u}_h }_{H^1(\Omega)} \norm{\nabla e^u}_{L^2(\Omega)}, \\
      ( \pmb{\varphi} - \IhGamma \pmb{\varphi} , \nabla e^v )_{\Omega}
      &= - ( \nabla \cdot ( \pmb{\varphi} - \IhGamma \pmb{\varphi} ) , e^v )_{\Omega} + \langle (\pmb{\varphi} - \IhGamma \pmb{\varphi}) \cdot \pmb{n} , e^v \rangle_{\Gamma}  \\
      &\leq \| \nabla \cdot ( \pmb{\varphi} - \IhGamma \pmb{\varphi}) \|_{L^2(\Omega)} \norm{e^v}_{L^2(\Omega)} +  \| (\pmb{\varphi} - \IhGamma \pmb{\varphi}) \cdot \pmb{n} \|_{L^2(\Gamma)} \norm{e^v}_{L^2(\Gamma)} \\
      &\lesssim \left[ h/p \| \nabla \cdot ( \pmb{\varphi} - \IhGamma \pmb{\varphi}) \|_{L^2(\Omega)} + (h/p)^{1/2} \trinorm{ \pmb{\varphi} - \IhGamma \pmb{\varphi}} \right] \norm{\nabla e^u}_{L^2(\Omega)}, \\
      ( \nabla \cdot (\pmb{\varphi} - \IhGamma \pmb{\varphi}) , \gamma e^v )_{\Omega}
      &\leq \| \nabla \cdot ( \pmb{\varphi} - \IhGamma \pmb{\varphi}) \|_{L^2(\Omega)} \norm{e^v}_{L^2(\Omega)} \\
      &\lesssim h/p \|\nabla \cdot ( \pmb{\varphi} - \IhGamma \pmb{\varphi})\|_{L^2(\Omega)} \norm{\nabla e^u}_{L^2(\Omega)}, \\
      \langle (\pmb{\varphi} - \IhGamma \pmb{\varphi}) \cdot \pmb{n} , - \alpha e^v \rangle_{\Gamma}
      &\leq \| (\pmb{\varphi} - \IhGamma \pmb{\varphi}) \cdot \pmb{n} \|_{L^2(\Gamma)} \norm{e^v}_{L^2(\Gamma)} \\
      &\lesssim (h/p)^{1/2} \trinorm{ \pmb{\varphi} - \IhGamma \pmb{\varphi}} \norm{\nabla e^u}_{L^2(\Omega)}, \\
      ( \pmb{\varphi} - \IhGamma \pmb{\varphi} , \pmb{e}^{\pmb{\psi}} )_{\Omega}
      &\lesssim \| \pmb{\varphi} - \IhGamma \pmb{\varphi} \|_{L^2(\Omega)} \| \pmb{e}^{\pmb{\psi}} \|_{L^2(\Omega)} \\
      &\lesssim (h/p)^{1/2} \trinorm{ \pmb{\varphi} - \IhGamma \pmb{\varphi}} \norm{\nabla e^u}_{L^2(\Omega)}, \\
      ( \nabla \cdot (\pmb{\varphi} - \IhGamma \pmb{\varphi}) , \nabla \cdot \pmb{e}^{\pmb{\psi}} )_{\Omega}
      &= ( \nabla \cdot (\pmb{\varphi} - \IhGamma \pmb{\varphi}) , \nabla \cdot (\pmb{\psi} - \tilde{\pmb{\psi}}_h) )_{\Omega} \\
      &\leq \| \nabla \cdot (\pmb{\varphi} - \IhGamma \pmb{\varphi}) \|_{L^2(\Omega)} \| \nabla \cdot (\pmb{\psi} - \tilde{\pmb{\psi}}_h) \|_{L^2(\Omega)} \\
      &\lesssim h/p \| \nabla \cdot (\pmb{\varphi} - \IhGamma \pmb{\varphi}) \|_{L^2(\Omega)} \norm{\nabla e^u}_{L^2(\Omega)}, \\
      \langle (\pmb{\varphi} - \IhGamma \pmb{\varphi}) \cdot \pmb{n} , \pmb{e}^{\pmb{\psi}} \cdot \pmb{n} \rangle_{\Gamma}
      &\leq \| (\pmb{\varphi} - \IhGamma \pmb{\varphi}) \cdot \pmb{n} \|_{L^2(\Gamma)} \| ( \pmb{e}^{\pmb{\psi}}, e^v ) \|_b  \\
      &\lesssim \trinorm{ \pmb{\varphi} - \IhGamma \pmb{\varphi}} \norm{\nabla \pmb{e}^{\pmb{\varphi}}}_{L^2(\Omega)}. \\
    \end{alignedat}
  \end{equation*}
Inserting these bounds in (\ref{eq:theorem:grad_e_u_suboptimal_l2_error_estimate_robin-10}),  
  canceling one power of $\norm{\nabla e^u}_{L^2(\Omega)}$ on both sides, and collecting the terms yields
  \begin{equation*}
    \norm{\nabla e^u}_{L^2(\Omega)}
    \lesssim
    \norm{ u - \tilde{u}_h }_{H^1(\Omega)} +
    \trinorm{ \pmb{\varphi} - \IhGamma \pmb{\varphi}} +
    \frac{h}{p} \| \nabla \cdot ( \pmb{\varphi} - \IhGamma \pmb{\varphi}) \|_{L^2(\Omega)}.
  \end{equation*}
  Finally exploiting the estimates of the operator $\IhGamma$ we obtain at the asserted estimate.
\end{proof}

\begin{remark}
  Theorem~\ref{theorem:grad_e_u_suboptimal_l2_error_estimate_robin} seems again suboptimal:
  Given $f \in L^2(\Omega)$ and $g \in H^{1/2}(\Gamma)$ the shift theorem gives $u \in H^2(\Omega)$ and consequently $\pmb{\varphi} \in \pmb{H}^1(\Omega)$.
  Theorem~\ref{theorem:grad_e_u_suboptimal_l2_error_estimate_robin} gives
  \begin{equation*}
    \norm{ \nabla e^u }_{L^2(\Omega)}
    \lesssim  h \norm{u}_{H^2(\Omega)}
    + h \norm{\pmb{\varphi}}_{H^1(\Omega)}
    + h^{1/2} \| \pmb{\varphi} \cdot \pmb{n} \|_{H^{1/2}(\Gamma)}
    + h \norm{\nabla \cdot \pmb{\varphi} }_{L^2(\Omega)}
    \lesssim h^{1/2} (\norm{f}_{L^2(\Omega)} + \norm{g}_{H^{1/2}(\Gamma)}),
  \end{equation*}
  whereas from a best approximation viewpoint we could hope for $\mathcal{O}(h)$.\eremk
\end{remark}

\begin{theorem}[Optimal estimate for $\norm{\pmb{e}^{\pmb{\varphi}} \cdot \pmb{n}}_{L^2(\Gamma)}$]\label{theorem:e_phi_normal_trace_l2_error_estimate_robin}
  Let Assumption~\ref{assumption:smax_shift} be valid for some $\smax \geq 0$.  Let $( \pmb{\varphi}_h , u_h )$ be the FOSLS approximation of $( \pmb{\varphi} , u )$. 
  Set  $e^u = u-u_h$ and $\pmb{e}^{\pmb{\varphi}} = \pmb{\varphi}-\pmb{\varphi}_h$. Then,
  for any $\tilde{u}_h \in \Sp$, $\tilde{\pmb{\varphi}}_h \in \RTBDM$, there holds 
  \begin{equation*}
    \norm{\pmb{e}^{\pmb{\varphi}} \cdot \pmb{n}}_{L^2(\Gamma)}
    \lesssim \left(\frac{h}{p}\right)^{1/2} \norm{u - \tilde{u}_h}_{H^1(\Omega)}
    + \left(\frac{h}{p}\right)^{1/2} \norm{\pmb{\varphi} - \tilde{\pmb{\varphi}}_h}_{L^2(\Omega)}
    + \| (\pmb{\varphi} - \tilde{\pmb{\varphi}}_h) \cdot \pmb{n} \|_{L^2(\Gamma)}
    + \frac{h}{p} \norm{\nabla \cdot (\pmb{\varphi} - \tilde{\pmb{\varphi}}_h)}_{L^2(\Omega)}. 
  \end{equation*}
\end{theorem}

\begin{proof}
  Let $(\pmb{\psi}, v) \in \productspacerobin$ denote the dual solution given by Theorem~\ref{theorem:duality_argument_normal_trace_robin} with $\pmb{\eta} = \pmb{e}^{\pmb{\varphi}}$.
  Theorem~\ref{theorem:duality_argument_normal_trace_robin} asserts 
$\pmb{\psi} \in \pmb{H}^{1/2}(\Omega)$, $\nabla \cdot \pmb{\psi} \in H^{3/2}(\Omega)$, $\pmb{\psi} \cdot \pmb{n} \in L^2(\Gamma)$, and $v \in H^{3/2}(\Omega)$, which are 
controlled by $\|\pmb{e}^{\pmb{\varphi}} \cdot \pmb{n}\|_{L^2(\Gamma)}$. 
  For the analysis we employ the operator $\Pi_{p_v}^{\mathrm{div}}$ from \cite{melenk-rojik18} and summarized in Proposition~\ref{proposition:melenk_rojik_operator}.
  The main features exploited in the proof are that $\Pi_{p_v}^{\mathrm{div}}$ realizes the $L^2$ orthogonal projections of the divergence as well as the normal trace.
  By Galerkin orthogonality we have for any $( \tilde{\pmb{\psi}}_h , \tilde{v}_h  )$
  \begin{equation*}
    \norm{\pmb{e}^{\pmb{\varphi}} \cdot \pmb{n}}_{L^2(\Gamma)}^2
    = b( ( \pmb{e}^{\pmb{\varphi}} , e^u ) , ( \pmb{\psi} , v ) )
    = b( ( \pmb{e}^{\pmb{\varphi}} , e^u ) , ( \pmb{\psi} - \tilde{\pmb{\psi}}_h , v - \tilde{v}_h ) ).
  \end{equation*}
  Choosing $\tilde{\pmb{\psi}}_h = \Pi_{p_v}^{\mathrm{div}}\pmb{\psi}$, we estimate 
  \begin{equation*}
    \begin{alignedat}{2}
      ( \nabla \cdot \pmb{e}^{\pmb{\varphi}} + \gamma e^u , \nabla \cdot (\pmb{\psi} - \Pi_{p_v}^{\mathrm{div}}\pmb{\psi}) + \gamma (v - \tilde{v}_h) )_{\Omega}
      &\lesssim \| ( \pmb{e}^{\pmb{\varphi}}, e^u ) \|_b \left[ \| \nabla \cdot (\pmb{\psi} - \Pi_{p_v}^{\mathrm{div}}\pmb{\psi}) \|_{L^2(\Omega)} + \norm{ v - \tilde{v}_h }_{L^2(\Omega)} \right], \\
      ( \nabla e^u + \pmb{e}^{\pmb{\varphi}} , \nabla (v - \tilde{v}_h) + \pmb{\psi} - \Pi_{p_v}^{\mathrm{div}}\pmb{\psi} )_{\Omega}
      &\lesssim \left[ \|  \nabla e^u \|_{L^2(\Omega)} + \| \pmb{e}^{\pmb{\varphi}} \|_{L^2(\Omega)} \right] \left[ \| v - \tilde{v}_h \|_{H^1(\Omega)} + \| \pmb{\psi} - \Pi_{p_v}^{\mathrm{div}}\pmb{\psi} \|_{L^2(\Omega)} \right], \\
      \langle -\alpha e^u , -\alpha (v - \tilde{v}_h) \rangle_{\Gamma}
      &\lesssim \| ( \pmb{e}^{\pmb{\varphi}}, e^u ) \|_b \| v - \tilde{v}_h \|_{L^2(\Gamma)}, \\
      \langle \pmb{e}^{\pmb{\varphi}} \cdot \pmb{n} , (\pmb{\psi} - \tilde{\pmb{\psi}}_h) \cdot \pmb{n} \rangle_{\Gamma}
      &\stackrel{\text{orth. of $\Pi^{\div}_{p_v}$}}{=} \langle (\pmb{\varphi} - \tilde{\pmb{\varphi}}_h) \cdot \pmb{n} , (\pmb{\psi} - \Pi_{p_v}^{\mathrm{div}}\pmb{\psi}) \cdot \pmb{n} \rangle_{\Gamma} \\
      &\lesssim  \| (\pmb{\varphi} - \tilde{\pmb{\varphi}}_h) \cdot \pmb{n} \|_{L^2(\Gamma)} \| (\pmb{\psi} - \Pi_{p_v}^{\mathrm{div}}\pmb{\psi}) \cdot \pmb{n} \|_{L^2(\Gamma)}, \\
    \end{alignedat}
  \end{equation*}
  The two missing boundary terms, i.e.,
  $\langle \pmb{e}^{\pmb{\varphi}} \cdot \pmb{n} , - \alpha (v - \tilde{v}_h) \rangle_{\Gamma}$ and
  $\langle -\alpha e^u , (\pmb{\psi} - \tilde{\pmb{\psi}}_h) \cdot \pmb{n} \rangle_{\Gamma}$,
  can be written as volume terms by means of partial integration
  \begin{align*}
    \langle \pmb{e}^{\pmb{\varphi}} \cdot \pmb{n} , - \alpha (v - \tilde{v}_h) \rangle_{\Gamma} 
    &= ( \nabla \cdot \pmb{e}^{\pmb{\varphi}} , - \alpha (v - \tilde{v}_h) )_{\Omega} 
    + ( \pmb{e}^{\pmb{\varphi}} , - \alpha \nabla (v - \tilde{v}_h) )_{\Omega}, \\
    \langle -\alpha e^u , (\pmb{\psi} - \tilde{\pmb{\psi}}_h) \cdot \pmb{n} \rangle_{\Gamma}
    &= ( -\alpha \nabla e^u , (\pmb{\psi} - \tilde{\pmb{\psi}}_h) )_{\Omega} 
    + ( -\alpha e^u , \nabla \cdot (\pmb{\psi} - \tilde{\pmb{\psi}}_h) )_{\Omega}
  \end{align*}
  and can therefore be controlled by the right-hand sides of the first two estimates.
  We now exploit the regularity estimates given in Theorem~\ref{theorem:duality_argument_normal_trace_robin}, the properties of $\Pi_{p_v}^{\mathrm{div}}$ given in Proposition~\ref{proposition:melenk_rojik_operator} as well as the approximation properties of the employed spaces to find $\tilde{v}_h$ such that
  \begin{align*}
    \| \nabla \cdot (\pmb{\psi} - \Pi_{p_v}^{\mathrm{div}}\pmb{\psi}) \|_{L^2(\Omega)}  & \lesssim h/p_v \| \nabla \cdot \pmb{\psi} \|_{H^1(\Omega)} \lesssim h/p_v \norm{\pmb{e}^{\pmb{\varphi}} \cdot \pmb{n}}_{L^2(\Gamma)},        \\
    \| \pmb{\psi} - \Pi_{p_v}^{\mathrm{div}}\pmb{\psi} \|_{L^2(\Omega)}                 & \lesssim (h/p_v)^{1/2} \| \pmb{\psi} \|_{\pmb{H}^{1/2}(\div,\Omega)} \lesssim (h/p_v)^{1/2} \norm{\pmb{e}^{\pmb{\varphi}} \cdot \pmb{n}}_{L^2(\Gamma)}, \\
    \| (\pmb{\psi} - \Pi_{p_v}^{\mathrm{div}}\pmb{\psi}) \cdot \pmb{n} \|_{L^2(\Gamma)} & \lesssim \| \pmb{\psi} \cdot \pmb{n} \|_{L^2(\Gamma)} \lesssim  \norm{\pmb{e}^{\pmb{\varphi}} \cdot \pmb{n}}_{L^2(\Gamma)},                  \\
    \norm{ v - \tilde{v}_h }_{L^2(\Omega)}                                              & \lesssim (h/p_s)^{3/2} \norm{ v }_{H^{3/2}(\Omega)} \lesssim (h/p_s)^{3/2} \norm{\pmb{e}^{\pmb{\varphi}} \cdot \pmb{n}}_{L^2(\Gamma)},       \\
    \norm{ v - \tilde{v}_h }_{H^1(\Omega)}                                              & \lesssim (h/p_s)^{1/2} \norm{ v }_{H^{3/2}(\Omega)} \lesssim (h/p_s)^{1/2}  \norm{\pmb{e}^{\pmb{\varphi}} \cdot \pmb{n}}_{L^2(\Gamma)},      \\
    \norm{ v - \tilde{v}_h }_{L^2(\Gamma)}                                              & \lesssim h/p_s \norm{ v }_{H^{3/2}(\Omega)} \lesssim h/p_s \norm{\pmb{e}^{\pmb{\varphi}} \cdot \pmb{n}}_{L^2(\Gamma)},
  \end{align*}
  which in turn gives after summarizing and canceling one power of $\norm{\pmb{e}^{\pmb{\varphi}} \cdot \pmb{n}}_{L^2(\Gamma)}$ on both sides of the estimate
  \begin{equation*}
    \norm{\pmb{e}^{\pmb{\varphi}} \cdot \pmb{n}}_{L^2(\Gamma)}
    \lesssim h/p \| ( \pmb{e}^{\pmb{\varphi}}, e^u ) \|_b + (h/p)^{1/2} \left[ \|  \nabla e^u \|_{L^2(\Omega)} + \| \pmb{e}^{\pmb{\varphi}} \|_{L^2(\Omega)} \right] + \| (\pmb{\varphi} - \tilde{\pmb{\varphi}}_h) \cdot \pmb{n} \|_{L^2(\Gamma)}.
  \end{equation*}
  Applying Theorems~\ref{theorem:e_phi_suboptimal_l2_error_estimate_robin} and~\ref{theorem:grad_e_u_suboptimal_l2_error_estimate_robin}
  to estimate $\| \pmb{e}^{\pmb{\varphi}} \|_{L^2(\Omega)}$ and $\|  \nabla e^u \|_{L^2(\Omega)}$ yields the result.
\end{proof}

\begin{remark}
  Theorem~\ref{theorem:e_phi_normal_trace_l2_error_estimate_robin} seems optimal in the following sense:
  Given $f \in L^2(\Omega)$ and $g \in H^{1/2}(\Gamma)$ the shift theorem gives $u \in H^2(\Omega)$ and consequently $\pmb{\varphi} \in \pmb{H}^1(\Omega)$.
  Theorem~\ref{theorem:e_phi_normal_trace_l2_error_estimate_robin} gives
  \begin{equation*}
    \norm{\pmb{e}^{\pmb{\varphi}} \cdot \pmb{n}}_{L^2(\Gamma)}
    \lesssim h^{3/2} \norm{u}_{H^2(\Omega)}
    + h^{3/2} \norm{\pmb{\varphi}}_{H^1(\Omega)}
    + h^{1/2} \| \pmb{\varphi} \cdot \pmb{n} \|_{H^{1/2}(\Gamma)}
    + h \norm{\nabla \cdot \pmb{\varphi}}_{L^2(\Omega)}
    \lesssim h^{1/2} (\norm{f}_{L^2(\Omega)} + \norm{g}_{H^{1/2}(\Gamma)}),
  \end{equation*}
  which is the rate expected from a best approximation argument.\eremk
\end{remark}

We are in position to derive an optimal estimate for $\norm{\nabla e^u}_{L^2(\Omega)}$ using the estimate given in Theorem~\ref{theorem:e_phi_normal_trace_l2_error_estimate_robin}.

\begin{theorem}[Optimal estimate for $\norm{\nabla e^u}_{L^2(\Omega)}$ ---
    Robin version of {\cite[Theorem~4.10]{bernkopf-melenk22}}]\label{theorem:grad_e_u_optimal_l2_error_estimate_robin}
  Let Assumption~\ref{assumption:smax_shift} be valid for some $\smax \geq 0$. Let $( \pmb{\varphi}_h , u_h )$ be the FOSLS approximation of $( \pmb{\varphi} , u )$.
  Set  $e^u = u-u_h$. Then,
  for any $\tilde{\pmb{\varphi}}_h \in \RTBDM$, $\tilde{u}_h \in \Sp$, there holds 
  \begin{equation*}
    \norm{\nabla e^u}_{L^2(\Omega)}
    \lesssim
    \norm{ u - \tilde{u}_h }_{H^1(\Omega)} +
    \left( \frac{h}{p} \right)^{1/2} \| \pmb{\varphi} - \tilde{\pmb{\varphi}}_h \|_{L^2(\Omega)} +
    \left( \frac{h}{p} \right)^{1/2} \| (\pmb{\varphi} - \tilde{\pmb{\varphi}}_h) \cdot \pmb{n} \|_{L^2(\Gamma)} +
    \frac{h}{p} \| \nabla \cdot (\pmb{\varphi} - \tilde{\pmb{\varphi}}_h ) \|_{L^2(\Omega)}. 
  \end{equation*}
\end{theorem}

\begin{proof}
We refine the proof of Theorem~\ref{theorem:grad_e_u_suboptimal_l2_error_estimate_robin} making use of Theorem~\ref{theorem:e_phi_normal_trace_l2_error_estimate_robin}.  Reentering the proof of Theorem~\ref{theorem:grad_e_u_suboptimal_l2_error_estimate_robin} we therein estimated
  \begin{align}
\nonumber 
    \langle (\pmb{\varphi} - \IhGamma \pmb{\varphi}) \cdot \pmb{n} , \pmb{e}^{\pmb{\psi}} \cdot \pmb{n} \rangle_{\Gamma}
     & \leq \| (\pmb{\varphi} - \IhGamma \pmb{\varphi}) \cdot \pmb{n} \|_{L^2(\Gamma)} \| ( \pmb{e}^{\pmb{\psi}}, e^v ) \|_b \\
 \label{eq:theorem:grad_e_u_optimal_l2_error_estimate_robin-10}
     & \lesssim \trinorm{ \pmb{\varphi} - \IhGamma \pmb{\varphi}} \norm{\nabla e^u}_{L^2(\Omega)}.                           
  \end{align}
This estimate can now be improved: Theorem~\ref{theorem:e_phi_normal_trace_l2_error_estimate_robin} gives together with the available regularity assertions for the dual solution 
the bound 
  \begin{align*}
    \| \pmb{e}^{\pmb{\psi}} \cdot \pmb{n} \|_{L^2(\Gamma)}
     & \lesssim \left(\frac{h}{p}\right)^{1/2} \norm{v - \tilde{v}_h}_{H^1(\Omega)}
    + \left(\frac{h}{p}\right)^{1/2} \| \pmb{\psi} - \tilde{\pmb{\psi}}_h \|_{L^2(\Omega)}
    + \| (\pmb{\psi} - \tilde{\pmb{\psi}}_h) \cdot \pmb{n} \|_{L^2(\Gamma)}
    + \frac{h}{p} \| \nabla \cdot (\pmb{\psi} - \tilde{\pmb{\psi}}_h) \|_{L^2(\Omega)} \\
     & \lesssim \left(\frac{h}{p}\right)^{1/2} \norm{\nabla e^u}_{L^2(\Omega)};
  \end{align*}
in turn this enables us to sharpen the bound for $\| (\pmb{e}^{\pmb{\psi}}, e^v ) \|_b$ and improve 
    the last estimate in (\ref{eq:theorem:grad_e_u_optimal_l2_error_estimate_robin-10}) to get
  \begin{equation*}
    \langle (\pmb{\varphi} - \IhGamma \pmb{\varphi}) \cdot \pmb{n} , \pmb{e}^{\pmb{\psi}} \cdot \pmb{n} \rangle_{\Gamma}
    \lesssim (h/p)^{1/2} \trinorm{ \pmb{\varphi} - \IhGamma \pmb{\varphi}} \norm{\nabla e^u}_{L^2(\Omega)}.
  \end{equation*}
  All other estimates in the proof of Theorem~\ref{theorem:grad_e_u_suboptimal_l2_error_estimate_robin} stay the same.
  Canceling one power of $\norm{\nabla e^u}_{L^2(\Omega)}$ on both sides and collecting the terms yields
  \begin{equation*}
    \norm{\nabla e^u}_{L^2(\Omega)}
    \lesssim
    \norm{ u - \tilde{u}_h }_{H^1(\Omega)} +
    \left( \frac{h}{p} \right)^{1/2} \trinorm{ \pmb{\varphi} - \IhGamma \pmb{\varphi}} +
    \frac{h}{p} \| \nabla \cdot ( \pmb{\varphi} - \IhGamma \pmb{\varphi}) \|_{L^2(\Omega)}.
  \end{equation*}
  Finally exploiting the estimates of the operator $\IhGamma$ in Lemma~\ref{lemma:properties_of_IhGamma} we arrive at the asserted estimate.
\end{proof}

Before turning to the estimate for $\norm{e^u}_{L^2(\Omega)}$ we first derive a slightly better version of Theorem~\ref{theorem:e_phi_suboptimal_l2_error_estimate_robin}.
To that end we first analyze the convergence of the corresponding dual solution:

\begin{lemma}[Convergence of dual solution for $\pmb{e}^{\pmb{\varphi}}$]\label{lemma:convergence_of_dual_solution_phi_robin}
  Let Assumption~\ref{assumption:smax_shift} be valid for some $\smax \geq 0$. Let $( \pmb{\varphi}_h , u_h )$ be the FOSLS approximation of $( \pmb{\varphi} , u )$.
  Set $e^u = u-u_h$ and $\pmb{e}^{\pmb{\varphi}} = \pmb{\varphi}-\pmb{\varphi}_h$.
  Let $(\pmb{\psi}, v) \in \productspacerobin$ be the dual solution given by Theorem~\ref{theorem:duality_argument_phi_robin} with $\pmb{\eta} = \pmb{e}^{\pmb{\varphi}}$.
  Let $( \pmb{\psi}_h , v_h )$ be the FOSLS approximation of $( \pmb{\psi} , v )$. 
  Denote $e^v = v-v_h$ and $\pmb{e}^{\pmb{\psi}} = \pmb{\psi}-\pmb{\psi}_h$.
  Then,
  \begin{align*}
    \| ( \pmb{e}^{\pmb{\psi}} , e^v ) \|_b                 & \lesssim \norm{\pmb{e}^{\pmb{\varphi}}}_{L^2(\Omega)},                                \\
    \norm{e^v}_{L^2(\Omega)}                               & \lesssim \frac{h}{p} \norm{\pmb{e}^{\pmb{\varphi}}}_{L^2(\Omega)},                    \\
    \norm{e^v}_{L^2(\Gamma)}                               & \lesssim \left(\frac{h}{p}\right)^{1/2} \norm{\pmb{e}^{\pmb{\varphi}}}_{L^2(\Omega)}, \\
    \norm{\nabla e^v}_{L^2(\Omega)}                        & \lesssim \left(\frac{h}{p}\right)^{1/2} \norm{\pmb{e}^{\pmb{\varphi}}}_{L^2(\Omega)}, \\
    \| \pmb{e}^{\pmb{\psi}} \|_{L^2(\Omega)}               & \lesssim \norm{\pmb{e}^{\pmb{\varphi}}}_{L^2(\Omega)},                                \\
    \| \pmb{e}^{\pmb{\psi}} \cdot \pmb{n} \|_{L^2(\Gamma)} & \lesssim \left(\frac{h}{p}\right)^{1/2} \norm{\pmb{e}^{\pmb{\varphi}}}_{L^2(\Omega)}.
  \end{align*}
\end{lemma}

\begin{proof}
  Theorem~\ref{theorem:duality_argument_phi_robin} gives $\pmb{\psi} \in \pmb{L}^2(\Omega)$, $\nabla \cdot \pmb{\psi} \in H^1(\Omega)$, $\pmb{\psi} \cdot \pmb{n} \in H^{1/2}(\Omega)$, and $v \in H^2(\Omega)$, which can be controlled by $\|\pmb{e}^{\pmb{\varphi}}\|_{L^2(\Omega)}$. Stability of the least squares method gives 
  \begin{equation}
\label{eq:lemma:convergence_of_dual_solution_phi_robin-10}
    \| ( \pmb{e}^{\pmb{\psi}} , e^v ) \|_b \leq \|(\pmb{\psi},v)\|_b 
    \lesssim \norm{\pmb{e}^{\pmb{\varphi}}}_{L^2(\Omega)}.
  \end{equation}
  Lemma~\ref{lemma:e_u_suboptimal_l2_error_estimate_robin} provides 
  \begin{equation*}
    \norm{e^v}_{L^2(\Omega)} \lesssim h/p \| ( \pmb{e}^{\pmb{\psi}} , e^v ) \|_b,
  \end{equation*}
  and together with (\ref{eq:lemma:convergence_of_dual_solution_phi_robin-10}) we arrive at the second estimate.
  The third one follows by a multiplicative trace inequality together with the second estimate and the norm equivalence theorem in conjunction with the first estimate of the present lemma:
  \begin{equation*}
    \norm{e^v}_{L^2(\Gamma)} \lesssim \norm{e^v}_{L^2(\Omega)}^{1/2} \norm{e^v}_{H^1(\Omega)}^{1/2} \lesssim (h/p)^{1/2} \| ( \pmb{e}^{\pmb{\psi}} , e^v ) \|_b \lesssim (h/p)^{1/2} \norm{\pmb{e}^{\pmb{\varphi}}}_{L^2(\Omega)}.
  \end{equation*}
  Theorems~\ref{theorem:grad_e_u_optimal_l2_error_estimate_robin},~\ref{theorem:e_phi_suboptimal_l2_error_estimate_robin} and~\ref{theorem:e_phi_normal_trace_l2_error_estimate_robin} then yield the remaining three estimates by combining the regularity assertions for the dual solution with the approximation properties of the finite element spaces. 
\end{proof}

\begin{theorem}[Suboptimal but improved estimate for $\norm{\pmb{e}^{\pmb{\varphi}}}_{L^2(\Omega)}$ ---
    Robin version of {\cite[Thm.~{4.8}]{bernkopf-melenk22}}]\label{theorem:e_phi_suboptimal_improved_l2_error_estimate_robin}
    Let Assumption~\ref{assumption:smax_shift} be valid for some $\smax \geq 0$. Let $( \pmb{\varphi}_h , u_h )$ be the FOSLS approximation of $( \pmb{\varphi} , u )$. Set  $e^u = u-u_h$ and $\pmb{e}^{\pmb{\varphi}} = \pmb{\varphi}-\pmb{\varphi}_h$. Then,
  for any $\tilde{u}_h \in \Sp$, $\tilde{\pmb{\varphi}}_h \in \RTBDM$, there holds 
  \begin{equation*}
    \norm{\pmb{e}^{\pmb{\varphi}}}_{L^2(\Omega)}
    \lesssim \norm{u - \tilde{u}_h}_{L^2(\Omega)}
    + \left(\frac{h}{p}\right)^{1/2} \norm{u - \tilde{u}_h}_{H^1(\Omega)}
    + \norm{\pmb{\varphi} - \tilde{\pmb{\varphi}}_h}_{L^2(\Omega)}
    + \| (\pmb{\varphi} - \tilde{\pmb{\varphi}}_h) \cdot \pmb{n} \|_{L^2(\Gamma)}
    + \frac{h}{p} \norm{\nabla \cdot (\pmb{\varphi} - \tilde{\pmb{\varphi}}_h)}_{L^2(\Omega)}. 
  \end{equation*}
\end{theorem}

\begin{proof}
  We proceed as in the proof of Theorem~\ref{theorem:grad_e_u_suboptimal_l2_error_estimate_robin}. Let $( \pmb{e}^{\pmb{\psi}} , e^v )$ be the FOSLS approximation error 
of the dual solution $(\pmb{\psi},v)$ given by Theorem~\ref{theorem:duality_argument_phi_robin} (duality argument for the vector variable) corresponding to the right-hand side 
$\pmb{\eta} = \pmb{e}^{\pmb{\varphi}}$. 
We have that $\|v\|_{H^2(\Omega)}$, $\|\pmb{\psi}\|_{L^2(\Omega)}$, $\|\nabla \cdot \pmb{\psi}\|_{H^1(\Omega)}$, $\|\pmb{\psi}\cdot\pmb{n}\|_{H^{1/2}(\Gamma)}$ 
are controlled by $\|\pmb{e}^{\pmb{\varphi}}\|_{L^2(\Omega)}$.
  As before for any $\tilde{\pmb{\varphi}}_h \in \RTBDM$, $\tilde{u}_h \in \Sp$
  \begin{equation*}
    \norm{\pmb{e}^{\pmb{\varphi}}}_{L^2(\Omega)}^2
    = b( ( \pmb{\varphi} - \tilde{\pmb{\varphi}}_h , u - \tilde{u}_h ), ( \pmb{e}^{\pmb{\psi}} , e^v ) ).
  \end{equation*}
  We again choose $\tilde{\pmb{\varphi}}_h = \IhGamma \pmb{\varphi}$,
  repeatedly use properties of the operator $\IhGamma$ collected in Lemma~\ref{lemma:properties_of_IhGamma},
  utilize the regularity properties of the dual solution given in Theorem~\ref{theorem:duality_argument_phi_robin}, and 
apply Lemma~\ref{lemma:convergence_of_dual_solution_phi_robin} to get:  
  \begin{equation*}
    \begin{alignedat}{2}
      ( \gamma (u - \tilde{u}_h) , \nabla \cdot \pmb{e}^{\pmb{\psi}} )_{\Omega}
      &\lesssim \norm{ u - \tilde{u}_h }_{L^2(\Omega)} \| ( \pmb{e}^{\pmb{\psi}}, e^v ) \|_b \\
      &\lesssim \norm{ u - \tilde{u}_h }_{L^2(\Omega)} \norm{\pmb{e}^{\pmb{\varphi}}}_{L^2(\Omega)} , \\
      \langle - \alpha (u - \tilde{u}_h) , \pmb{e}^{\pmb{\psi}} \cdot \pmb{n} - \alpha e^v \rangle_{\Gamma}
      &\lesssim \norm{ u - \tilde{u}_h }_{L^2(\Gamma)} \left[ \| \pmb{e}^{\pmb{\psi}} \cdot \pmb{n} \|_{L^2(\Gamma)} + \| e^v \|_{L^2(\Gamma)} \right] \\
      &\lesssim (h/p)^{1/2} \norm{ u - \tilde{u}_h }_{L^2(\Gamma)} \norm{\pmb{e}^{\pmb{\varphi}}}_{L^2(\Omega)} , \\
      ( \nabla (u - \tilde{u}_h) , \pmb{e}^{\pmb{\psi}} )_{\Omega}
      &= -( u - \tilde{u}_h , \nabla \cdot \pmb{e}^{\pmb{\psi}} )_{\Omega} + \langle u - \tilde{u}_h , \pmb{e}^{\pmb{\psi}} \cdot \pmb{n} \rangle_{\Gamma} \\
      &\lesssim \left[ \norm{ u - \tilde{u}_h }_{L^2(\Omega)} + (h/p)^{1/2} \norm{ u - \tilde{u}_h }_{L^2(\Gamma)} \right] \norm{\pmb{e}^{\pmb{\varphi}}}_{L^2(\Omega)}, \\
      ( \gamma (u - \tilde{u}_h) , \gamma e^v )_{\Omega}
      &\lesssim \norm{ u - \tilde{u}_h }_{L^2(\Omega)} \norm{ e^v }_{L^2(\Omega)} \\
      &\lesssim h/p \norm{ u - \tilde{u}_h }_{H^1(\Omega)} \norm{\pmb{e}^{\pmb{\varphi}}}_{L^2(\Omega)} , \\
      ( \nabla (u - \tilde{u}_h) , \nabla e^v )_{\Omega}
      &\lesssim \norm{\nabla (u - \tilde{u}_h) }_{L^2(\Omega)} \norm{ \nabla e^v }_{L^2(\Omega)} \\
      &\lesssim (h/p)^{1/2} \norm{ u - \tilde{u}_h }_{H^1(\Omega)} \norm{\pmb{e}^{\pmb{\varphi}}}_{L^2(\Omega)}, \\
      ( \pmb{\varphi} - \IhGamma \pmb{\varphi} , \nabla e^v + \pmb{e}^{\pmb{\psi}} )_{\Omega}
      &\leq \| \pmb{\varphi} - \IhGamma \pmb{\varphi} \|_{L^2(\Omega)} \| ( \pmb{e}^{\pmb{\psi}}, e^v ) \|_b \\
      &\lesssim \trinorm{ \pmb{\varphi} - \IhGamma \pmb{\varphi}} \norm{\pmb{e}^{\pmb{\varphi}}}_{L^2(\Omega)}, \\
      \langle (\pmb{\varphi} - \IhGamma \pmb{\varphi}) \cdot \pmb{n} , \pmb{e}^{\pmb{\psi}} \cdot \pmb{n} - \alpha e^v \rangle_{\Gamma}
      &\leq \| (\pmb{\varphi} - \IhGamma \pmb{\varphi}) \cdot \pmb{n} \|_{L^2(\Gamma)} \| ( \pmb{e}^{\pmb{\psi}}, e^v ) \|_b \\
      &\lesssim \trinorm{ \pmb{\varphi} - \IhGamma \pmb{\varphi}} \norm{\pmb{e}^{\pmb{\varphi}}}_{L^2(\Omega)}, \\
      ( \nabla \cdot (\pmb{\varphi} - \IhGamma \pmb{\varphi}) , \nabla \cdot \pmb{e}^{\pmb{\psi}} )_{\Omega}
      &= ( \nabla \cdot (\pmb{\varphi} - \IhGamma \pmb{\varphi}) , \nabla \cdot (\pmb{\psi} - \tilde{\pmb{\psi}}_h) )_{\Omega} \\
      &\leq \| \nabla \cdot (\pmb{\varphi} - \IhGamma \pmb{\varphi}) \|_{L^2(\Omega)} \| \nabla \cdot (\pmb{\psi} - \tilde{\pmb{\psi}}_h) \|_{L^2(\Omega)} \\
      &\lesssim h/p \| \nabla \cdot (\pmb{\varphi} - \IhGamma \pmb{\varphi}) \|_{L^2(\Omega)} \norm{\pmb{e}^{\pmb{\varphi}}}_{L^2(\Omega)}
    \end{alignedat}
  \end{equation*}
  Canceling one power of $\norm{\pmb{e}^{\pmb{\varphi}}}_{L^2(\Omega)}$ on both sides and summarizing we find
  \begin{equation*}
    \norm{\pmb{e}^{\pmb{\varphi}}}_{L^2(\Omega)}
    \lesssim
    \norm{ u - \tilde{u}_h }_{L^2(\Omega)}
    + (h/p)^{1/2} \norm{ u - \tilde{u}_h }_{L^2(\Gamma)}
    + (h/p)^{1/2} \norm{ u - \tilde{u}_h }_{H^1(\Omega)}
    + \trinorm{ \pmb{\varphi} - \IhGamma \pmb{\varphi}}
    + h/p \| \nabla \cdot (\pmb{\varphi} - \IhGamma \pmb{\varphi}) \|_{L^2(\Omega)}.
  \end{equation*}
  A trace estimate, using the estimates of the operator $\IhGamma$ in Lemma~\ref{lemma:properties_of_IhGamma}, and collecting the terms yields the result.
\end{proof}

\begin{lemma}[Convergence of dual solution for $e^u$ ---
    Robin version of {\cite[Lem.~{4.11}]{bernkopf-melenk22}}]\label{lemma:convergence_of_dual_solution_u_robin}
    Let Assumption~\ref{assumption:smax_shift} be valid for some $\smax \geq 0$. Let $( \pmb{\varphi}_h , u_h )$ be the FOSLS approximation of $( \pmb{\varphi} , u )$.
  Set  $e^u = u-u_h$ and $\pmb{e}^{\pmb{\varphi}} = \pmb{\varphi}-\pmb{\varphi}_h$.
  Let $(\pmb{\psi}, v) \in \productspacerobin$ be the dual solution given by Theorem~\ref{theorem:duality_argument_robin} with $w = e^u$.
  Furthermore, let $( \pmb{\psi}_h , v_h )$ be the FOSLS approximation of $( \pmb{\psi} , v )$ and
  denote $e^v = v-v_h$ and $\pmb{e}^{\pmb{\psi}} = \pmb{\psi}-\pmb{\psi}_h$.
  Then,
  \begin{align*}
    \| ( \pmb{e}^{\pmb{\psi}} , e^v ) \|_b                 & \lesssim \frac{h}{p} \norm{e^u}_{L^2(\Omega)},                    \\
    \norm{e^v}_{L^2(\Omega)}                               & \lesssim \left(\frac{h}{p}\right)^2 \norm{e^u}_{L^2(\Omega)},     \\
    \norm{e^v}_{L^2(\Gamma)}                               & \lesssim \left(\frac{h}{p}\right)^{3/2} \norm{e^u}_{L^2(\Omega)}, \\
    \| \pmb{e}^{\pmb{\psi}} \|_{L^2(\Omega)}               & \lesssim
    \begin{cases}
      h \norm{e^u}_{L^2(\Omega)}                              & \textit{if }  \RTBDM = \pmb{\mathrm{RT}}_{0}(\mathcal{T}_h), \\
      \left(\frac{h}{p}\right)^{\min(\smax+1, 3/2)} \norm{e^u}_{L^2(\Omega)} & \textit{else},
    \end{cases}                                                                                                \\
    \| \pmb{e}^{\pmb{\psi}} \cdot \pmb{n} \|_{L^2(\Gamma)} & \lesssim
    \begin{cases}
      h \norm{e^u}_{L^2(\Omega)}                              & \textit{if } \RTBDM = \pmb{\mathrm{RT}}_{0}(\mathcal{T}_h), \\
      \left(\frac{h}{p}\right)^{\min(\smax+1, 3/2)} \norm{e^u}_{L^2(\Omega)} & \textit{else}.
    \end{cases}
  \end{align*}
\end{lemma}

\begin{proof}
  Theorem~\ref{theorem:duality_argument_robin} gives 
    $\pmb{\psi} \in \pmb{H}^{\min(\smax+1, 2)}(\Omega)$, 
    $\nabla \cdot \pmb{\psi} \in H^2(\Omega)$, 
    $\pmb{\psi} \cdot \pmb{n} \in H^{3/2}(\Gamma)$ 
    and $v \in H^2(\Omega)$, which are controlled
by $\|e^u\|_{L^2(\Omega)}$. In view of the optimality of the $b$-norm we have: 
  \begin{equation}
\label{eq:foo-10}
    \| ( \pmb{e}^{\pmb{\psi}} , e^v ) \|_b
    \lesssim \frac{h}{p} \norm{e^u}_{L^2(\Omega)}.
  \end{equation}
  By Lemma~\ref{lemma:e_u_suboptimal_l2_error_estimate_robin} we have
  \begin{equation*}
    \norm{e^v}_{L^2(\Omega)} \lesssim \frac{h}{p} \| ( \pmb{e}^{\pmb{\psi}} , e^v ) \|_b,
  \end{equation*}
  which together with (\ref{eq:foo-10}) gives the second estimate of the lemma. 
  The third estimate of the lemma follows by a multiplicative trace inequality together with the second estimate and the norm equivalence theorem in conjunction with the first estimate of the present lemma:
  \begin{equation*}
    \norm{e^v}_{L^2(\Gamma)} \lesssim \norm{e^v}_{L^2(\Omega)}^{1/2} \norm{e^v}_{H^1(\Omega)}^{1/2} \lesssim (h/p)^{3/2} \| ( \pmb{e}^{\pmb{\psi}} , e^v ) \|_b \lesssim (h/p)^{3/2} \norm{e^u}_{L^2(\Omega)}.
  \end{equation*}
  Theorems~\ref{theorem:e_phi_suboptimal_l2_error_estimate_robin} and~\ref{theorem:e_phi_normal_trace_l2_error_estimate_robin} then yield the remaining estimates by exploiting the regularity of the dual solution and the approximation properties of the employed spaces.
\end{proof}

\begin{theorem}[Optimal estimate for $\norm{e^u}_{L^2(\Omega)}$ ---
    Robin version of {\cite[Theorem~4.12]{bernkopf-melenk22}}]\label{theorem:e_u_optimal_l2_error_estimate_robin}
    Let Assumption~\ref{assumption:smax_shift} be valid for some $\smax \geq 0$. Let $( \pmb{\varphi}_h , u_h )$ be the FOSLS approximation of $( \pmb{\varphi} , u )$.
  Furthermore, let  $e^u = u-u_h$. Then,
  for any $\tilde{\pmb{\varphi}}_h \in \RTBDM$, $\tilde{u}_h \in \Sp$, there holds 
  \[
    \norm{e^u}_{L^2(\Omega)}
    \lesssim
    \begin{cases}
      h \norm{ u - \tilde{u}_h }_{H^1(\Omega)} +
      h \| \pmb{\varphi} - \tilde{\pmb{\varphi}}_h \|_{L^2(\Omega)} + \\
      \quad h \| (\pmb{\varphi} - \tilde{\pmb{\varphi}}_h) \cdot \pmb{n} \|_{L^2(\Gamma)} + 
      h \| \nabla \cdot (\pmb{\varphi} - \tilde{\pmb{\varphi}}_h ) \|_{L^2(\Omega)} & \textit{if }\RTBDMZero = \pmb{\mathrm{RT}}_{0}(\mathcal{T}_h),  \\ \\
      h \norm{ u - \tilde{u}_h }_{H^1(\Omega)} +
      h^{{\min(\smax+1,3/2)}} \| \pmb{\varphi} - \tilde{\pmb{\varphi}}_h \|_{L^2(\Omega)} + \\
      \quad  h^{{\min(\smax+1,3/2)}} \| (\pmb{\varphi} - \tilde{\pmb{\varphi}}_h) \cdot \pmb{n} \|_{L^2(\Gamma)} +
      h \| \nabla \cdot (\pmb{\varphi} - \tilde{\pmb{\varphi}}_h ) \|_{L^2(\Omega)} & \textit{if }\RTBDMZero = \pmb{\mathrm{BDM}}_{1}(\mathcal{T}_h), \\ \\
      \frac{h}{p} \norm{ u - \tilde{u}_h }_{H^1(\Omega)} +
      \left( \frac{h}{p} \right)^{{\min(\smax+1,3/2)}} \| \pmb{\varphi} - \tilde{\pmb{\varphi}}_h \|_{L^2(\Omega)} + \\
      \quad \left( \frac{h}{p} \right)^{{\min(\smax+1,3/2)}} \| (\pmb{\varphi} - \tilde{\pmb{\varphi}}_h) \cdot \pmb{n} \|_{L^2(\Gamma)} + 
      \left( \frac{h}{p} \right)^2 \| \nabla \cdot (\pmb{\varphi} - \tilde{\pmb{\varphi}}_h ) \|_{L^2(\Omega)} & \textit{else.}
    \end{cases}
  \]
\end{theorem}

\begin{proof}
  We proceed as in the proof of Theorem~\ref{theorem:grad_e_u_suboptimal_l2_error_estimate_robin} with $( \pmb{e}^{\pmb{\psi}} , e^v )$ denoting the FOSLS approximation of the dual solution given by Theorem~\ref{theorem:duality_argument_robin} (duality argument for the scalar variable) applied to $w = e^u$
  As before for any $\tilde{\pmb{\varphi}}_h \in \RTBDM$, $\tilde{u}_h \in \Sp$
  \begin{equation*}
    \norm{e^u}_{L^2(\Omega)}^2
    = b( ( \pmb{\varphi} - \tilde{\pmb{\varphi}}_h , u - \tilde{u}_h ), ( \pmb{e}^{\pmb{\psi}} , e^v ) ).
  \end{equation*}
  We again choose $\tilde{\pmb{\varphi}}_h = \IhGamma \pmb{\varphi}$,
  utilize properties of the operator $\IhGamma$ collected in Lemma~\ref{lemma:properties_of_IhGamma},
  make use of the regularity assertions for the dual solution of Theorem~\ref{theorem:duality_argument_robin}, and apply Lemma~\ref{lemma:convergence_of_dual_solution_u_robin}:  
  \begin{equation*}
    \begin{alignedat}{2}
      ( \gamma (u - \tilde{u}_h) , \nabla \cdot \pmb{e}^{\pmb{\psi}} + \gamma e^v )_{\Omega}
      &\lesssim \norm{ u - \tilde{u}_h }_{L^2(\Omega)} \| ( \pmb{e}^{\pmb{\psi}}, e^v ) \|_b \\
      &\lesssim h/p \norm{ u - \tilde{u}_h }_{H^1(\Omega)} \norm{e^u}_{L^2(\Omega)} , \\
      ( \nabla (u - \tilde{u}_h) , \nabla e^v + \pmb{e}^{\pmb{\psi}} )_{\Omega}
      &\lesssim \norm{\nabla (u - \tilde{u}_h) }_{L^2(\Omega)} \| ( \pmb{e}^{\pmb{\psi}}, e^v ) \|_b \\
      &\lesssim h/p \norm{ u - \tilde{u}_h }_{H^1(\Omega)} \norm{e^u}_{L^2(\Omega)}, \\
      \langle - \alpha (u - \tilde{u}_h) , \pmb{e}^{\pmb{\psi}} \cdot \pmb{n} - \alpha e^v \rangle_{\Gamma}
      &\lesssim \norm{ u - \tilde{u}_h }_{L^2(\Gamma)} \| ( \pmb{e}^{\pmb{\psi}}, e^v ) \|_b \\
      &\lesssim h/p \norm{ u - \tilde{u}_h }_{H^1(\Omega)} \norm{e^u}_{L^2(\Omega)}, \\
      ( \pmb{\varphi} - \IhGamma \pmb{\varphi} , \nabla e^v )_{\Omega}
      &= - ( \nabla \cdot ( \pmb{\varphi} - \IhGamma \pmb{\varphi} ) , e^v )_{\Omega} + \langle (\pmb{\varphi} - \IhGamma \pmb{\varphi}) \cdot \pmb{n} , e^v \rangle_{\Gamma}  \\
      &\leq \| \nabla \cdot ( \pmb{\varphi} - \IhGamma \pmb{\varphi}) \|_{L^2(\Omega)} \norm{e^v}_{L^2(\Omega)} +  \| (\pmb{\varphi} - \IhGamma \pmb{\varphi}) \cdot \pmb{n} \|_{L^2(\Gamma)} \norm{e^v}_{L^2(\Gamma)} \\
      &\lesssim \left[ (h/p)^2 \| \nabla \cdot ( \pmb{\varphi} - \IhGamma \pmb{\varphi}) \|_{L^2(\Omega)} + (h/p)^{3/2} \trinorm{ \pmb{\varphi} - \IhGamma \pmb{\varphi}} \right] \norm{e^u}_{L^2(\Omega)}, \\
      ( \nabla \cdot (\pmb{\varphi} - \IhGamma \pmb{\varphi}) , \gamma e^v )_{\Omega}
      &\leq \| \nabla \cdot ( \pmb{\varphi} - \IhGamma \pmb{\varphi}) \|_{L^2(\Omega)} \norm{e^v}_{L^2(\Omega)} \\
      &\lesssim (h/p)^2 \|\nabla \cdot ( \pmb{\varphi} - \IhGamma \pmb{\varphi})\|_{L^2(\Omega)} \norm{e^u}_{L^2(\Omega)}, \\
      \langle (\pmb{\varphi} - \IhGamma \pmb{\varphi}) \cdot \pmb{n} , - \alpha e^v \rangle_{\Gamma}
      &\leq \| (\pmb{\varphi} - \IhGamma \pmb{\varphi}) \cdot \pmb{n} \|_{L^2(\Gamma)} \norm{e^v}_{L^2(\Gamma)} \\
      &\lesssim (h/p)^{3/2} \trinorm{ \pmb{\varphi} - \IhGamma \pmb{\varphi}} \norm{e^u}_{L^2(\Omega)}, \\
      ( \pmb{\varphi} - \IhGamma \pmb{\varphi} , \pmb{e}^{\pmb{\psi}} )_{\Omega}
      &\lesssim \| \pmb{\varphi} - \IhGamma \pmb{\varphi} \|_{L^2(\Omega)} \| \pmb{e}^{\pmb{\psi}} \|_{L^2(\Omega)} \\
      &\lesssim \begin{cases}
        h \trinorm{ \pmb{\varphi} - \IhGamma \pmb{\varphi}} \norm{e^u}_{L^2(\Omega)}                              & \textit{if }\RTBDM = \pmb{\mathrm{RT}}_{0}(\mathcal{T}_h), \\
        \left(\frac{h}{p}\right)^{{\min(\smax+1,3/2)}} \trinorm{ \pmb{\varphi} - \IhGamma \pmb{\varphi}} \norm{e^u}_{L^2(\Omega)} & \textit{else},
      \end{cases} \\
      ( \nabla \cdot (\pmb{\varphi} - \IhGamma \pmb{\varphi}) , \nabla \cdot \pmb{e}^{\pmb{\psi}} )_{\Omega}
      &= ( \nabla \cdot (\pmb{\varphi} - \IhGamma \pmb{\varphi}) , \nabla \cdot (\pmb{\psi} - \tilde{\pmb{\psi}}_h) )_{\Omega} \\
      &\leq \| \nabla \cdot (\pmb{\varphi} - \IhGamma \pmb{\varphi}) \|_{L^2(\Omega)} \| \nabla \cdot (\pmb{\psi} - \tilde{\pmb{\psi}}_h) \|_{L^2(\Omega)} \\
      &\lesssim \begin{cases}
        h \| \nabla \cdot (\pmb{\varphi} - \IhGamma \pmb{\varphi}) \|_{L^2(\Omega)} \norm{e^u}_{L^2(\Omega)}                            & \textit{if }p_v = 1,       \\
        \left(\frac{h}{p}\right)^{2} \| \nabla \cdot (\pmb{\varphi} - \IhGamma \pmb{\varphi}) \|_{L^2(\Omega)} \norm{e^u}_{L^2(\Omega)} & \textit{else},
      \end{cases} \\
      \langle (\pmb{\varphi} - \IhGamma \pmb{\varphi}) \cdot \pmb{n} , \pmb{e}^{\pmb{\psi}} \cdot \pmb{n} \rangle_{\Gamma}
      &\leq \| (\pmb{\varphi} - \IhGamma \pmb{\varphi}) \cdot \pmb{n} \|_{L^2(\Gamma)} \| \pmb{e}^{\pmb{\psi}} \cdot \pmb{n} \|_{L^2(\Gamma)}  \\
      &\lesssim \begin{cases}
        h \trinorm{ \pmb{\varphi} - \IhGamma \pmb{\varphi}} \norm{e^u}_{L^2(\Omega)}                              & \textit{if }\RTBDM = \pmb{\mathrm{RT}}_{0}(\mathcal{T}_h), \\
        \left(\frac{h}{p}\right)^{{\min(\smax+1,3/2)}} \trinorm{ \pmb{\varphi} - \IhGamma \pmb{\varphi}} \norm{e^u}_{L^2(\Omega)} & \textit{else}.
      \end{cases}
    \end{alignedat}
  \end{equation*}
  Canceling one power of $\norm{e^u}_{L^2(\Omega)}$ on both sides, using the estimates of the operator $\IhGamma$ and collecting the terms yields the result.
\end{proof}

\begin{corollary}\label{corollary:summary_of_estimates_for_f_in_higher_order_sobolev_space_robin}
  Let $\Gamma$ be smooth, $f \in H^s(\Omega)$ and $g \in H^{s+1/2}(\Gamma)$ for some $s \geq 0$, 
  and denote $C_{f,g} \coloneqq \norm{f}_{H^s(\Omega)} + \norm{g}_{H^{s+1/2}(\Gamma)}$.
  Then the solution to $(\ref{eq:model_problem_first_order_system_robin})$ satisfies
  $u \in H^{s+2}(\Omega)$, $\pmb{\varphi} \in \pmb{H}^{s+1}(\Omega)$, $\pmb{\varphi} \cdot \pmb{n} \in \pmb{H}^{s+1/2}(\Gamma)$, and $\nabla \cdot \pmb{\varphi} \in H^s(\Omega)$.
  Let $( \pmb{\varphi}_h , u_h )$ be the FOSLS approximation of $( \pmb{\varphi} , u )$.
  Let  $e^u = u-u_h$ and $\pmb{e}^{\pmb{\phi}} = \pmb{\varphi}-\pmb{\varphi}_h$.
  Then, for the lowest order case $p_v = 1$,
  \begin{equation*}
    \norm{e^u}_{L^2(\Omega)} \lesssim
    h^{\min(s+1, 2)} \norm{f}_{H^s(\Omega)}.
  \end{equation*}
  For $p_v > 1$ there holds
  \\
  \begin{center}
    \begin{tabular}{|c|c|}                   \hline
      $ \RTBDM = \RT $                                                                                        & $ \RTBDM = \BDM $ \\ \hline
      $ \norm{e^u}_{L^2(\Omega)} \lesssim \left(\frac{h}{p}\right)^{\min(s+1, p_s, p_v + 1/2) + 1 } C_{f,g} $ &
      $ \norm{e^u}_{L^2(\Omega)} \lesssim \left(\frac{h}{p}\right)^{\min(s+1, p_s, p_v + 1) + 1 } C_{f,g} $                       \\ \hline
    \end{tabular}
  \end{center}
  Furthermore, the estimates
  \\
  \begin{center}
    \begin{tabular}{|c|c|}                   \hline
      $ \RTBDM = \RT $                                                                                           & $ \RTBDM = \BDM $ \\ \hline
      $ \norm{\nabla e^u}_{L^2(\Omega)} \lesssim \left(\frac{h}{p}\right)^{\min(s+1, p_s, p_v + 1/2) } C_{f,g} $ &
      $ \norm{\nabla e^u}_{L^2(\Omega)} \lesssim \left(\frac{h}{p}\right)^{\min(s+1, p_s, p_v + 1) } C_{f,g} $                       \\ \hline
    \end{tabular}
  \end{center}
  and
  \\
  \begin{center}
    \begin{tabular}{|c|c|}                   \hline
      $ \RTBDM = \RT $                                                                                                          & $ \RTBDM = \BDM $ \\ \hline
      $ \norm{\pmb{e}^{\pmb{\varphi}}}_{L^2(\Omega)} \lesssim \left(\frac{h}{p}\right)^{\min(s+1/2, p_s + 1/2, p_v) } C_{f,g} $ &
      $ \norm{\pmb{e}^{\pmb{\varphi}}}_{L^2(\Omega)} \lesssim \left(\frac{h}{p}\right)^{\min(s+1/2, p_s + 1/2, p_v + 1) } C_{f,g} $.                \\ \hline
    \end{tabular}
  \end{center}
  hold. Finally we have
  \\
  \begin{center}
    \begin{tabular}{|c|c|}                   \hline
      $ \RTBDM = \RT $                                                                                                                        & $ \RTBDM = \BDM $ \\ \hline
      $ \norm{\pmb{e}^{\pmb{\varphi}} \cdot \pmb{n}}_{L^2(\Gamma)} \lesssim \left(\frac{h}{p}\right)^{\min(s+1/2, p_s + 1/2, p_v) } C_{f,g} $ &
      $ \norm{\pmb{e}^{\pmb{\varphi}} \cdot \pmb{n}}_{L^2(\Gamma)} \lesssim \left(\frac{h}{p}\right)^{\min(s+1/2, p_s + 1/2, p_v + 1) } C_{f,g} $.                \\ \hline
    \end{tabular}
  \end{center}
\end{corollary}

\begin{proof}
  By the smoothness of $\Gamma$, Assumption~\ref{assumption:smax_shift} holds for any $\smax$. The regularity of $\pmb{\varphi}$ follows from $\pmb{\varphi} = - \nabla u$.
  We next inspect the quantities in the estimates of Theorems
  ~\ref{theorem:e_phi_normal_trace_l2_error_estimate_robin},
  ~\ref{theorem:grad_e_u_optimal_l2_error_estimate_robin},
  ~\ref{theorem:e_phi_suboptimal_improved_l2_error_estimate_robin}, and
  ~\ref{theorem:e_u_optimal_l2_error_estimate_robin}:
  \begin{align*}
    \norm{ u - \tilde{u}_h }_{L^2(\Omega)}
     & \lesssim (h/p)^{\min(s+1, p_s)+1} \norm{u}_{H^{s+2}(\Omega)} \lesssim (h/p)^{\min(s+1, p_s)+1} C_{f,g},              \\
    \norm{ u - \tilde{u}_h }_{H^1(\Omega)}
     & \lesssim (h/p)^{\min(s+1, p_s)} \norm{u}_{H^{s+2}(\Omega)} \lesssim (h/p)^{\min(s+1, p_s)} C_{f,g},                  \\
    \norm{ u - \tilde{u}_h }_{L^2(\Gamma)}
     & \lesssim (h/p)^{\min(s+1, p_s)+1/2} \norm{u}_{H^{s+2}(\Omega)} \lesssim (h/p)^{\min(s+1, p_s)+1/2} C_{f,g},          \\
    \| \pmb{\varphi} - \tilde{\pmb{\varphi}}_h \|_{L^2(\Omega)}
     & \lesssim
    \begin{cases}
      (h/p)^{\min(s+1, p_v)}   \norm{ \pmb{\varphi} }_{H^{s+1}(\Omega)} \lesssim (h/p)^{\min(s+1, p_v)}   C_{f,g} & \RTBDM = \RT,  \\
      (h/p)^{\min(s+1, p_v+1)} \norm{ \pmb{\varphi} }_{H^{s+1}(\Omega)} \lesssim (h/p)^{\min(s+1, p_v+1)} C_{f,g} & \RTBDM = \BDM,
    \end{cases}                                                                                             \\
    \| (\pmb{\varphi} - \tilde{\pmb{\varphi}}_h) \cdot \pmb{n} \|_{L^2(\Gamma)}
     & \lesssim
    \begin{cases}
      (h/p)^{\min(s+1/2, p_v)}   \norm{ \pmb{\varphi} }_{H^{s+1}(\Omega)} \lesssim (h/p)^{\min(s+1/2, p_v)}   C_{f,g} & \RTBDM = \RT,  \\
      (h/p)^{\min(s+1/2, p_v+1)} \norm{ \pmb{\varphi} }_{H^{s+1}(\Omega)} \lesssim (h/p)^{\min(s+1/2, p_v+1)} C_{f,g} & \RTBDM = \BDM,
    \end{cases}                                                                                             \\
    \| \nabla \cdot (\pmb{\varphi} - \tilde{\pmb{\varphi}}_h ) \|_{L^2(\Omega)}
     & \lesssim (h/p)^{\min(s, p_v)} \norm{ \nabla \cdot \pmb{\varphi} }_{H^{s}(\Omega)} \lesssim h^{\min(s, p_v)} C_{f,g}.
  \end{align*}
  The bounds in Theorems~\ref{theorem:e_phi_normal_trace_l2_error_estimate_robin},
  ~\ref{theorem:grad_e_u_optimal_l2_error_estimate_robin},
  ~\ref{theorem:e_phi_suboptimal_improved_l2_error_estimate_robin} and
  ~\ref{theorem:e_u_optimal_l2_error_estimate_robin}
  together with the above estimates gives the asserted rates.
\end{proof}

\begin{remark}\label{remark:suboptimality_of_estimate_for_e_phi_in_l2_vs_e_phi_normal}
  Note that Corollary~\ref{corollary:summary_of_estimates_for_f_in_higher_order_sobolev_space_robin} predicts the same rates 
  for $\norm{\pmb{e}^{\pmb{\varphi}}}_{L^2(\Omega)}$ and $\norm{\pmb{e}^{\pmb{\varphi}} \cdot \pmb{n}}_{L^2(\Gamma)}$. 
  This again suggests the suboptimality of the estimate for $\norm{\pmb{e}^{\pmb{\varphi}}}_{L^2(\Omega)}$.\eremk
\end{remark}

\section{Numerical examples}\label{section:numerical_examples_robin}

Our numerical examples are obtained with $hp$-FEM code NETGEN/NGSOLVE by J.~Sch\"oberl,~\cite{schoeberlNGSOLVE,schoeberl97}.
In Example~\ref{example:numerics_singular_solution_robin} we consider a right-hand side 
$f \in \cap_{\varepsilon >0} \left(H^{1/2-\varepsilon}(\Omega)\right)\setminus H^{1/2}(\Omega)$  so that 
$u \in \cap_{\varepsilon >0} H^{5/2-\varepsilon}(\Omega)$ and
$\pmb{\varphi} \in \cap_{\varepsilon >0} \pmb{H}^{3/2-\varepsilon}(\Omega)$.
In all graphs presented, we plot the actual numerical results (red dots),
the rate that is guaranteed by Corollary~\ref{corollary:summary_of_estimates_for_f_in_higher_order_sobolev_space_robin} 
(in black with the number written out near the bottom right),
and a reference line for the best rate possible with the employed space $\Sp$ or $\RTBDM$ given the Sobolev regularity of the solution $u$ (in blue with the number written out near the top left).

\begin{example}\label{example:numerics_singular_solution_robin}
  Our computational domain $\Omega$ is the unit sphere in $\mathbb{R}^2$, and we take 
  $f(x,y) = \mathbbm{1}_{[0,1/2]}(\sqrt{x^2 + y^2})$, which is a step function
  supported by a disk of radius $1/2$. In (\ref{eq:model_problem_robin}) we set $\gamma = 2$ 
and $\alpha = 1$. The exact solution $u$ is determined by the condition $\partial_n u = 0$ on $\Gamma$. 
  The right-hand side boundary data $g$ is calculated according to the choice $\alpha = 1$.  
  The solution has finite regularity $u \in H^{5/2-\varepsilon}(\Omega)$ for all $\varepsilon > 0$. 
  We perform both the $h$-version and the $p$-version of the FOSLS method.
  The solution $u$ is in fact piecewise smooth, 
  but the meshes employed for both the $h$-version and the $p$-version are not aligned with the regions of smoothness of $u$, 
  but rather such that the meshes do not resolve the circle of radius $1/2$,
  where the solution $u$ has limited regularity.
  Regarding the $h$-version, we employ every combination of $\RTBDM$ and $\Sp$ 
  for $p_v, p_s \in \{1,2,3,4,5\}$.
  For the $p$-version, we select a fixed mesh with mesh size $h \approx 0.6$ and 
chose $\RTBDM = \RT$ and $p_v = p_s = p$, for $p = 1, \cdots, 27$.
  The numerical results for the $h$-version are plotted 
  in Figures~\ref{figure:l2_error_u_RT_robin_singular} and~\ref{figure:l2_error_u_BDM_robin_singular} for $\norm{e^u}_{L^2(\Omega)}$,
  in Figures~\ref{figure:l2_error_grad_u_RT_robin_singular} and~\ref{figure:l2_error_grad_u_BDM_robin_singular} for $\norm{\nabla e^u}_{L^2(\Omega)}$, and 
  in Figures~\ref{figure:l2_error_phi_RT_robin_singular} and~\ref{figure:l2_error_phi_BDM_robin_singular} for $\norm{\pmb{e}^{\pmb{\varphi}}}_{L^2(\Omega)}$.
  The numerical results for the $p$-version are plotted
  in Figure~\ref{figure:robin_p_version_singular}. We make the following observations: 
  \begin{itemize}
  \item In Figures~\ref{figure:l2_error_u_RT_robin_singular} and~\ref{figure:l2_error_u_BDM_robin_singular} 
  for $\norm{e^u}_{L^2(\Omega)}$, we 
  observe that the convergence rates asserted in Corollary~\ref{corollary:summary_of_estimates_for_f_in_higher_order_sobolev_space_robin} 
  are attained. However, for the lowest order case $p_v = 1$, the theoretical results seem suboptimal.
  \item In Figures~\ref{figure:l2_error_grad_u_RT_robin_singular} and~\ref{figure:l2_error_grad_u_BDM_robin_singular} 
  for $\norm{\nabla e^u}_{L^2(\Omega)}$, the convergence rates guaranteed by Corollary~\ref{corollary:summary_of_estimates_for_f_in_higher_order_sobolev_space_robin} 
  are achieved and are optimal in terms of the regularity of the data. 
  \item In Figures~\ref{figure:l2_error_phi_RT_robin_singular} and~\ref{figure:l2_error_phi_BDM_robin_singular} 
  for $\norm{\pmb{e}^{\pmb{\varphi}}}_{L^2(\Omega)}$, we observe better convergence than 
  ensured by Corollary~\ref{corollary:summary_of_estimates_for_f_in_higher_order_sobolev_space_robin}, 
  which is in agreement with our comments in
  Remarks~\ref{remark:suboptimal_l2_error_estimate_robin} and~\ref{remark:suboptimality_of_estimate_for_e_phi_in_l2_vs_e_phi_normal}.
  \item In Figure~\ref{figure:robin_p_version_singular} for the $p$-version of the FOSLS method, we observe a convergence rate 
  predicted by Corollary~\ref{corollary:summary_of_estimates_for_f_in_higher_order_sobolev_space_robin} and that it is 
  optimal for $\norm{e^u}_{L^2(\Omega)}$ and $\norm{\nabla e^u}_{L^2(\Omega)}$.
  We again note the suboptimality of our estimates for $\norm{\pmb{e}^{\pmb{\varphi}}}_{L^2(\Omega)}$.
  Finally, for $\norm{\pmb{e}^{\pmb{\varphi}} \cdot \pmb{n}}_{L^2(\Gamma)}$ we remark 
  convergence $\mathcal{O}(p^{-2.5})$ (with a green reference line), 
  whereas Corollary~\ref{corollary:summary_of_estimates_for_f_in_higher_order_sobolev_space_robin}
  merely guarantees convergence $\mathcal{O}(p^{-1})$. Since 
  the solution $(\pmb{\varphi}, u)$ is smooth near $\Gamma$ a local error analysis 
  near $\Gamma$ could possibly explain this super-convergence behavior.
  \end{itemize}
\end{example}

\begin{figure}[H]
  \centering
  \includegraphics[width=\textwidth]{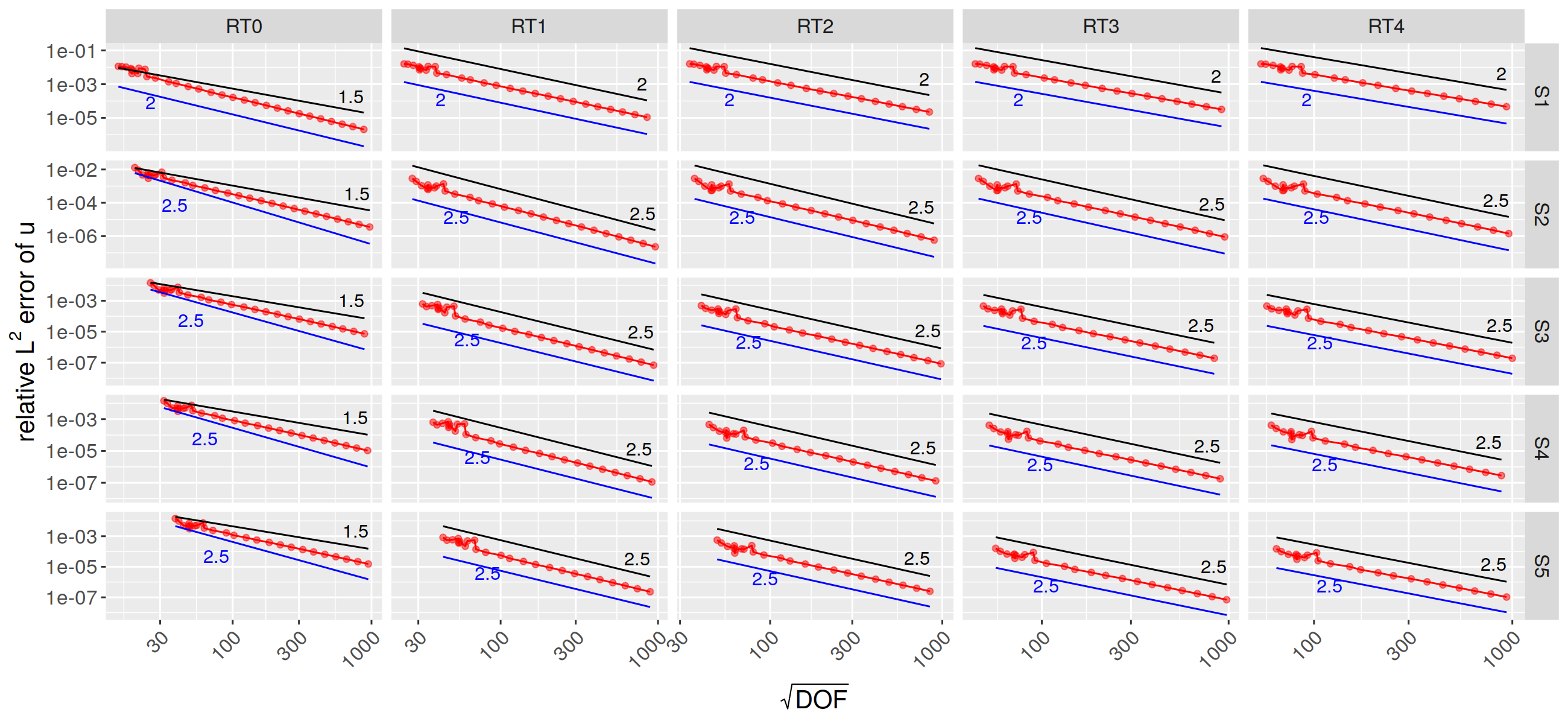}
  \caption{
    $h$-convergence of $\norm{e^u}_{L^2(\Omega)}$ using $\RTBDM  = \RT $, see Example~\ref{example:numerics_singular_solution_robin}.
  }
  \label{figure:l2_error_u_RT_robin_singular}
\end{figure}

\begin{figure}[H]
  \centering
  \includegraphics[width=\textwidth]{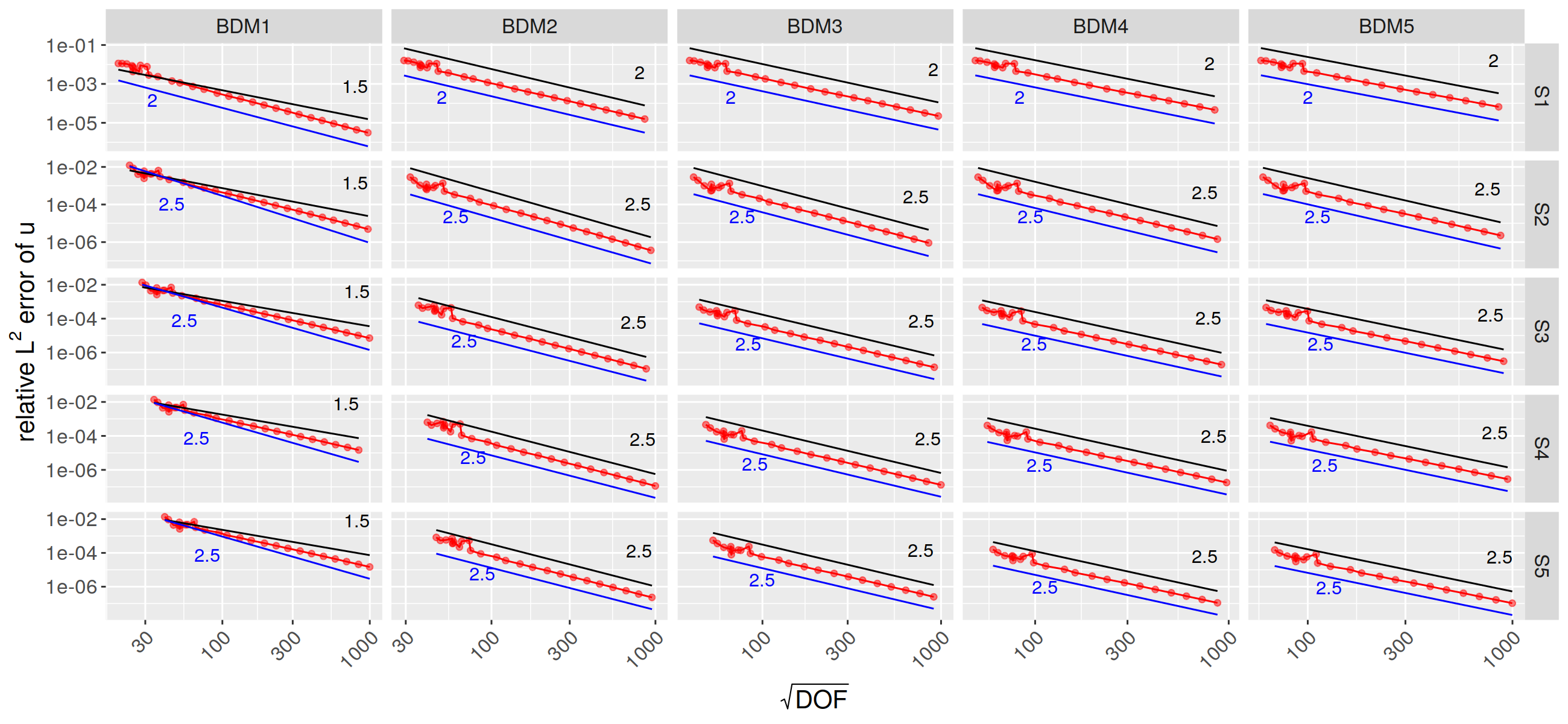}
  \caption{
    $h$-convergence of $\norm{e^u}_{L^2(\Omega)}$ using $\RTBDM  = \BDM $, see Example~\ref{example:numerics_singular_solution_robin}.
  }
  \label{figure:l2_error_u_BDM_robin_singular}
\end{figure}

\begin{figure}[H]
  \centering
  \includegraphics[width=\textwidth]{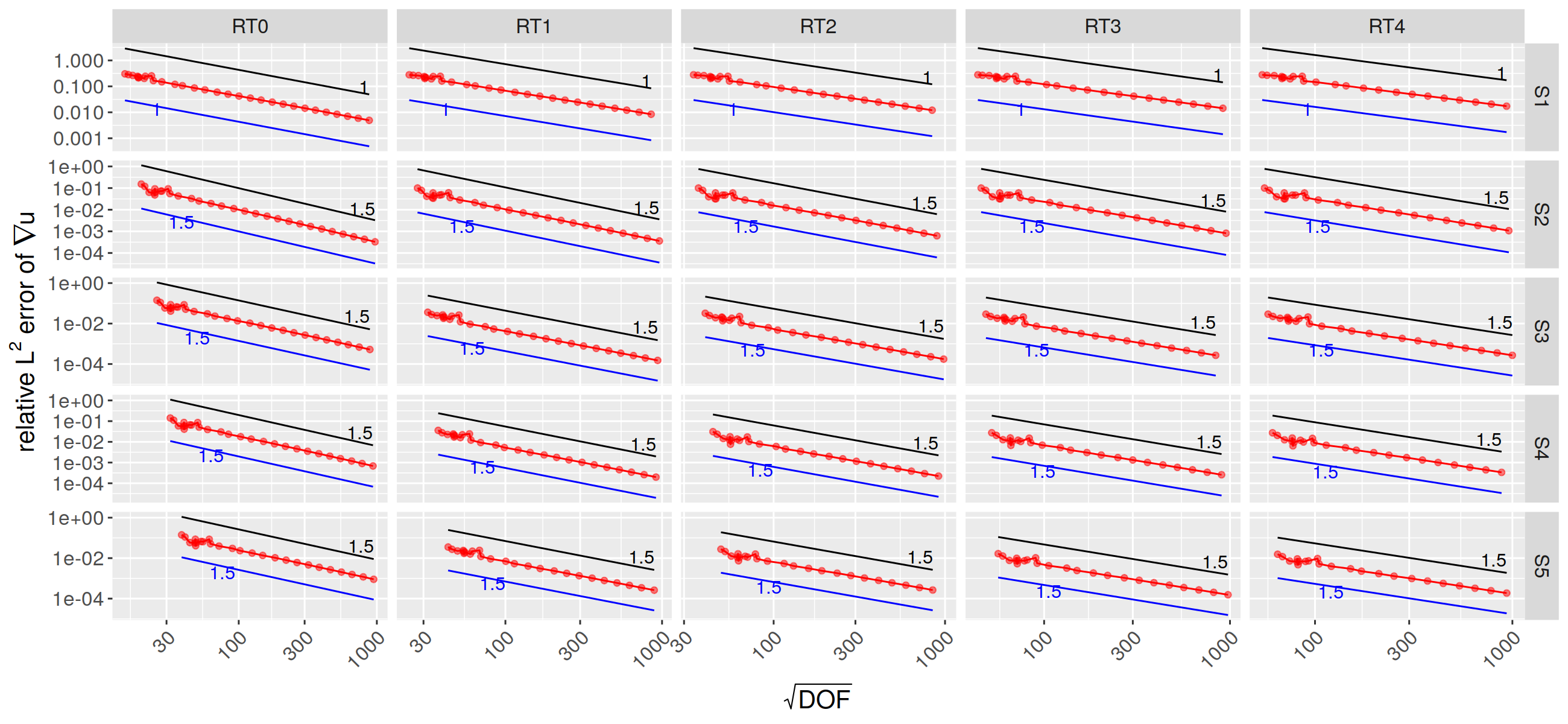}
  \caption{
    $h$-convergence of $\norm{\nabla e^u}_{L^2(\Omega)}$ using $\RTBDM  = \RT $, see Example~\ref{example:numerics_singular_solution_robin}.
  }
  \label{figure:l2_error_grad_u_RT_robin_singular}
\end{figure}

\begin{figure}[H]
  \centering
  \includegraphics[width=\textwidth]{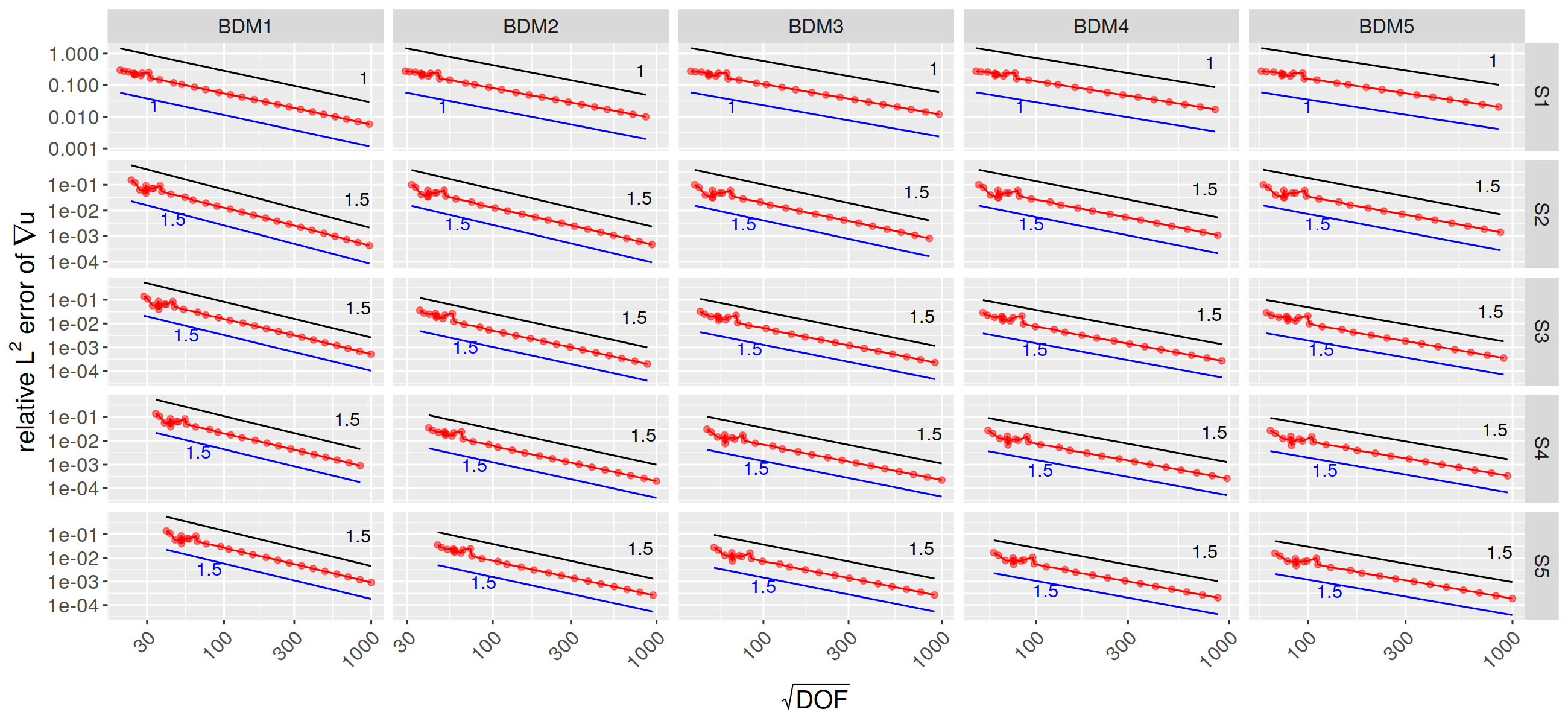}
  \caption{
    $h$-convergence of $\norm{\nabla e^u}_{L^2(\Omega)}$ using $\RTBDM  = \BDM $, see Example~\ref{example:numerics_singular_solution_robin}.
  }
  \label{figure:l2_error_grad_u_BDM_robin_singular}
\end{figure}

\begin{figure}[H]
  \centering
  \includegraphics[width=\textwidth]{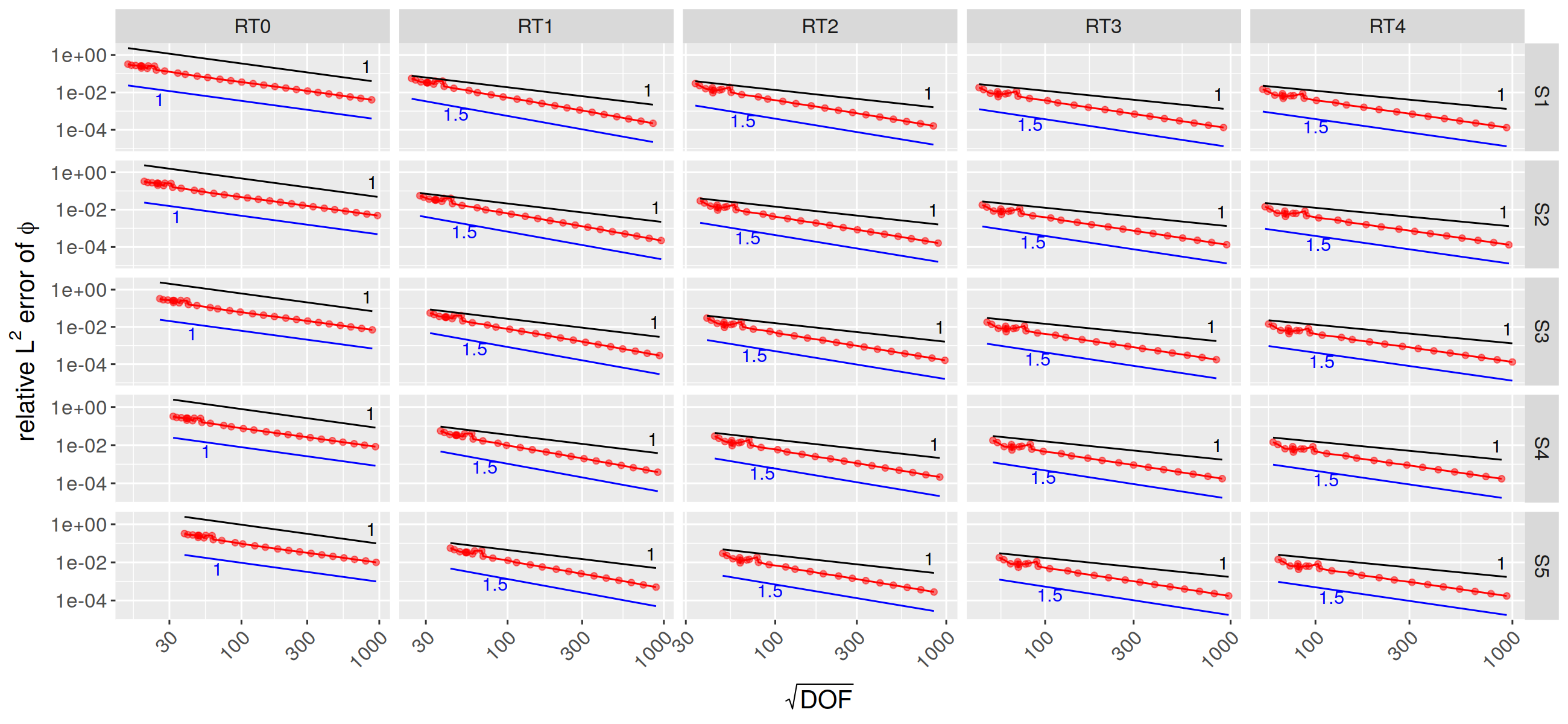}
  \caption{
    $h$-convergence of $\norm{\pmb{e}^{\pmb{\varphi}}}_{L^2(\Omega)}$ using $\RTBDM  = \RT $, see Example~\ref{example:numerics_singular_solution_robin}.
  }
  \label{figure:l2_error_phi_RT_robin_singular}
\end{figure}

\begin{figure}[H]
  \centering
  \includegraphics[width=\textwidth]{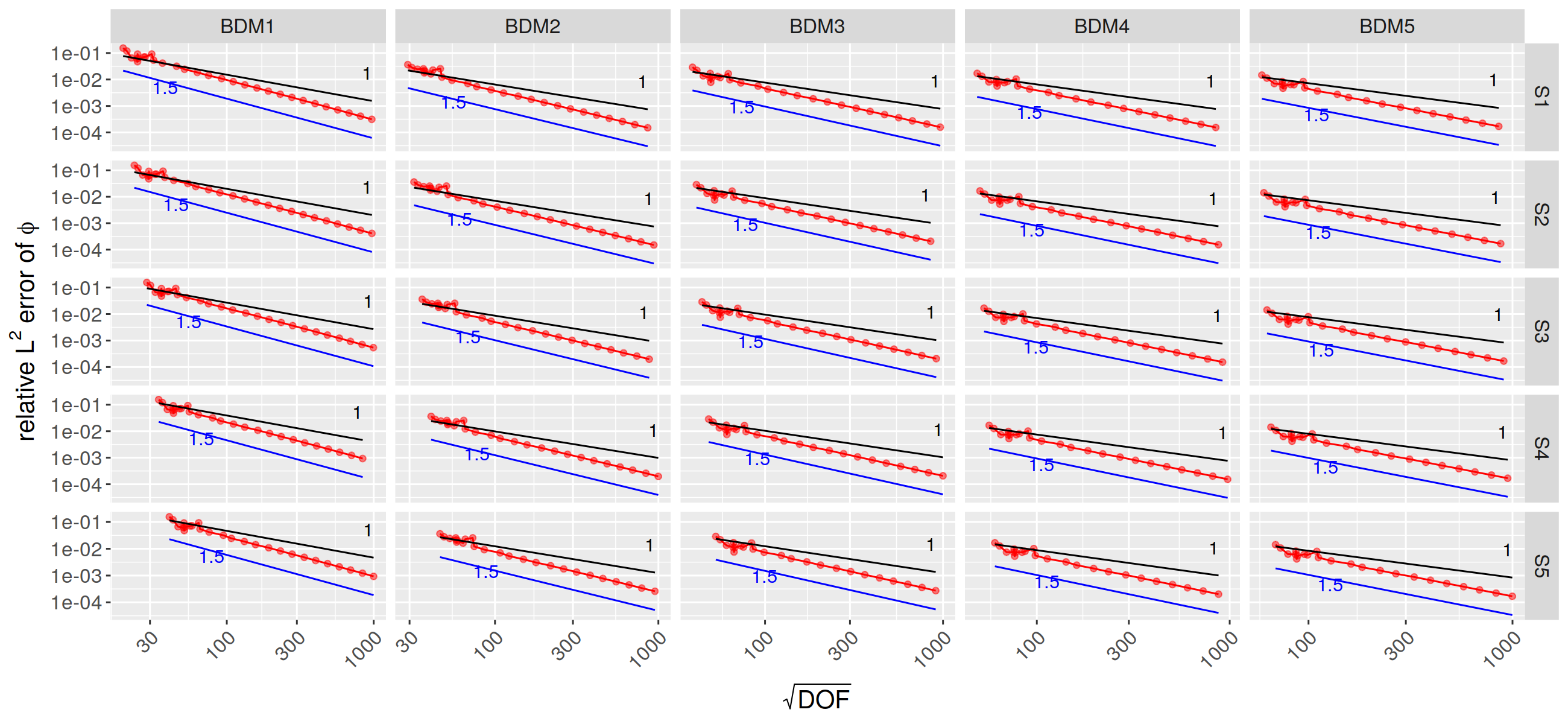}
  \caption{
    $h$-convergence of $\norm{\pmb{e}^{\pmb{\varphi}}}_{L^2(\Omega)}$ using $\RTBDM  = \BDM $, see Example~\ref{example:numerics_singular_solution_robin}.
  }
  \label{figure:l2_error_phi_BDM_robin_singular}
\end{figure}



\begin{figure}[H]
  \centering
  \includegraphics[width=\textwidth]{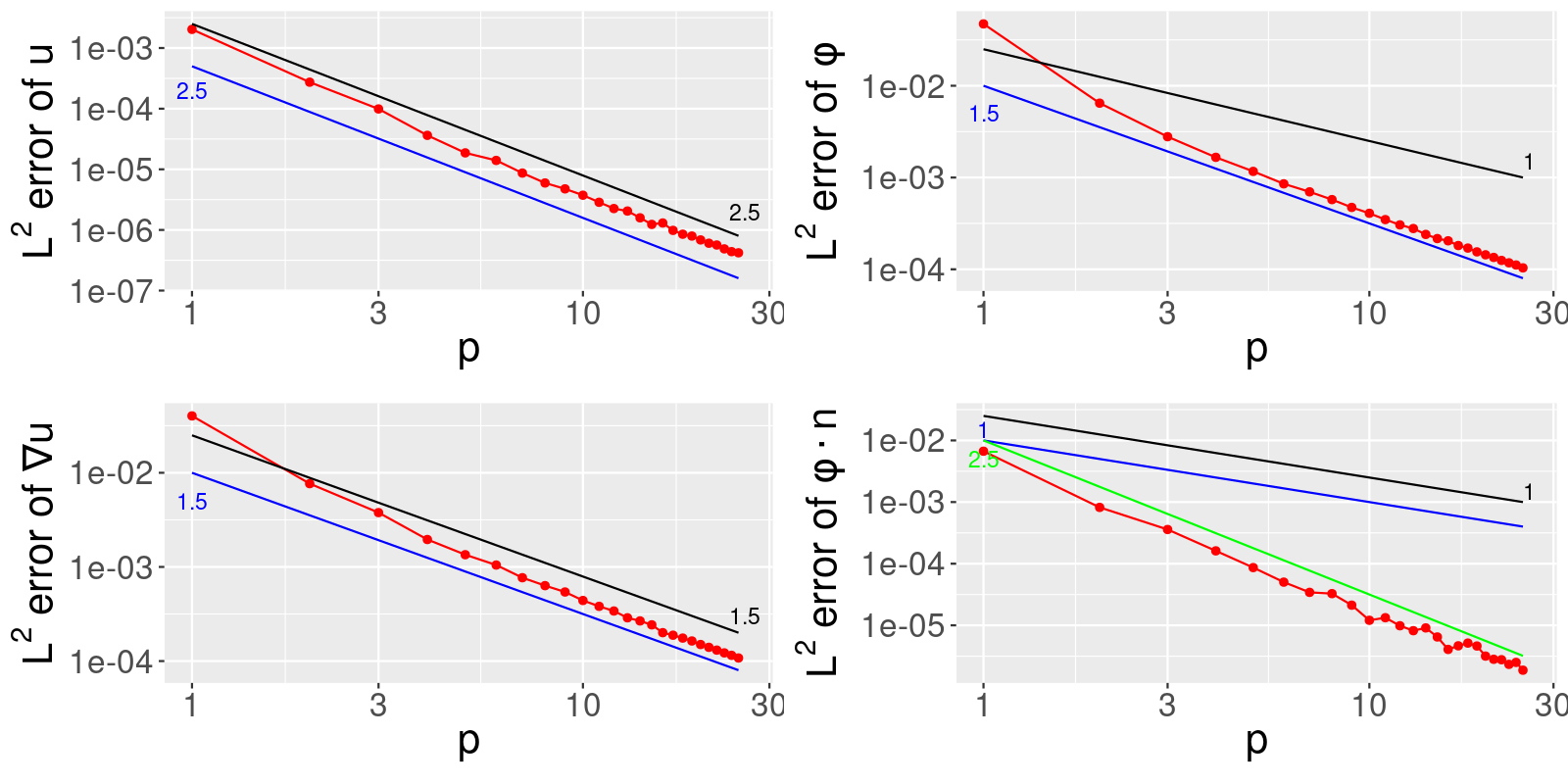}
  \caption{
    $p$-convergence of 
    $\norm{e^u}_{L^2(\Omega)}$ (top left),
    $\norm{\nabla e^u}_{L^2(\Omega)}$ (bottom left),
    $\norm{\pmb{e}^{\pmb{\varphi}}}_{L^2(\Omega)}$ (top right),
    $\norm{\pmb{e}^{\pmb{\varphi}} \cdot \pmb{n}}_{L^2(\Gamma)}$ (bottom right) 
    using $\RTBDM  = \RT$ and $\Sp$ with $p_v=p_s$, see Example~\ref{example:numerics_singular_solution_robin}.
  }
  \label{figure:robin_p_version_singular}
\end{figure}

\bibliographystyle{alpha}
\bibliography{literature.bib}

\end{document}